\documentclass[10pt,reqno]{amsart}

\usepackage{color}
\usepackage[dvipsnames]{xcolor}

\usepackage[english]{babel}

\usepackage{bbm,amsmath,amsthm,epsfig,latexsym,marvosym,mathrsfs}
\usepackage{amsfonts}
\usepackage{amsmath,stackengine}
\usepackage{amssymb,a4wide}
\usepackage{esint,xfrac}

\usepackage[colorlinks=true,linkcolor=blue,citecolor=red]{hyperref}

\definecolor{myc}{RGB}{36,107,2}

\def\R{\mathbb{R}}
\def\N{\mathbb{N}}

\def\F{\mathcal{F}}
\def\H{\mathcal{H}}
\def\L{\mathcal{L}}

\def\I{\mathcal{I}}

\def\d{\,\mathrm{d}}

\def\wt{\stackrel{*}{\rightharpoonup}}
\def\Tan{\mathrm{Tan}}
\def\spt{\mathrm{spt}\,}
\def\e{\varepsilon}

\def\mm{\mathfrak{m}}
\def\ZZ{\mathbb{B}}
\def\LL{\mathbb{L}}
\def\euno{\mathrm{e}_1}
\def\edue{\mathrm{e}_2}
\def\en{\mathrm{e}_n}

\def\mat{\mathbb{A}}

\newcommand{\res}{\mathop{\hbox{\vrule height 7pt width .5pt depth 0pt \vrule height .5pt width 6pt depth 0pt}}\nolimits}

\newtheorem{theorem}{Theorem}[section]
\newtheorem{corollary}[theorem]{Corollary}
\newtheorem{proposition}[theorem]{Proposition}
\newtheorem{lemma}[theorem]{Lemma}
\newtheorem{remark}[theorem]{Remark}

\theoremstyle{definition}

\numberwithin{equation}{section}
\numberwithin{figure}{section}

\definecolor{mygreen}{RGB}{21,118,40} 

\author{M. Caroccia}
\address{DiMaI, Universit\`a di Firenze, and Scuola Normale Superiore di Pisa}
\email{caroccia.marco@gmail.com}
\author{M. Focardi}
\address{DiMaI, Universit\`a di Firenze}
\email{matteo.focardi@unifi.it}
\author{N. Van Goethem}
\address{Centro de Matem\'atica e Aplica\c c\~{o}es Fundamentais, Universidade de Lisboa}
\email{vangoeth@fc.ul.pt}

\makeindex
\title[On the integral representation of variational functionals on $BD$]{On the integral representation of variational functionals on $BD$}


\begin{document}
\begin{abstract}
 Following the global method for relaxation we prove an integral representation result for a large class of variational functionals naturally defined 
 on the space of functions with Bounded Deformation. Mild additional continuity assumptions are required on
 the functionals. 
 \end{abstract}

\maketitle
 
\tableofcontents

\section{Introduction}
In linearized elasticity one route is to consider the displacement (or the velocity) field $u$ as basic model variable. In this case, the deformation (or strain) tensor of an elastic body is given by the symmetric gradient of $u$, i.e., $e(u):=\frac{1}{2}(\nabla u+\nabla^tu)$. Therefore, the study of well-posedness of the PDE system of linear elasticity was at the origin of the study of the differential operator $e(u)$  and in particular its coerciveness properties, first analysed by Korn in 1906 \cite{Korn} and followed by plenty of refinements  to this date (see for instance \cite{Horgan} for a survey). In linearized elasticity, the variational approach consists in minimizing the stored elastic energy (which is quadratic in the strain) minus the work of the external forces.
However, as soon as elasto-plasticity is considered, two main problems are faced: first, the observed stress-strain relation in plasticity is not linear anymore, resulting in a less-than quadratic, sometimes linear relation between the stored elastic energy and the strain. Here we refer to the pioneer work by Suquet on well posedness in perfect plasticity \cite{Suquet3} itself based on preliminary work on the distributional operator of bounded deformation published in \cite{Suquet1,Suquet2,TS1,TS2,Kohn}. Specifically, in the above quoted works the authors study the properties of the differential operator $Eu:=\frac{1}{2}(Du+D^tu)$ where $D$ stands for the distributional derivative that generalizes the gradient $\nabla$ to account for discontinuous fields $u$ . In this way the space $BD(\Omega)$ of function with Bounded Deformation on the open subset $\Omega$ of $\R^n$ has been introduced as the space of $L^1(\Omega;\R^n)$ vector fields $u$ whose symmetrized distributional derivative $Eu$ is a Radon measure (see \cite{Temam1,ambrosio1997fine}, see also section \ref{prelBD} to which we refer for the notation used in this introduction on $BD$ maps). 
Moreover, $e(u)$ is proven to be the density of the absolutely continuous part of $Eu$. 

The second issue arising in plastic problems is that concentration phenomena observed in plasticity require some weak notion of deformation that allow for slip or boundary concentration of strain for instance. Indeed, these effects are well handled in $BD(\Omega)$ by the so-called singular part of the deformation measure field.  It should also be said, that these aforementioned two issues are related, since a linear growth of the stored elastic energy prevents coerciveness in Sobolev spaces. Thus, bounds in the non-reflexive space $L^1$ require to consider limit of sequences in the space of Radon measures, and hence, again, justifies the choice of the space $BD(\Omega)$ when dealing with elasto-plastic models. For these models, the associated general bulk stored elastic energy reads as the integral
\begin{equation}\label{e:int1}
F_0(u):=\int_\Omega f_0\big(x,u(x),e(u)(x)\big)\d x,
\end{equation}
where $f_0$ has linear growth in the third variable and satisfies suitable assumptions (see Section~\ref{ss:relax}),
and $u\in BD(\Omega)$ is such that $Eu$ is absolutely continuous with respect to $\L^n\res\Omega$, namely $u\in LD(\Omega)$. 
To account also for singular effects a more general energy expression reads as
\begin{equation}\label{e:int2}
F_1(u):=\int_\Omega f_1\big(x,u(x),e(u)(x)\big)\d x
+\int_{J_u} g_1\left(x,u^-(x),u^+(x),\nu_u(x)\right)d\mathcal{H}^{n-1}
\end{equation}
where $f_1$ has linear growth in the third variable, $f_1$ and $g_1$ satisfy suitable assumptions (see Section~\ref{ss:relaxbulksurface}),
and $u\in SBD(\Omega)$, the subspace of $BD(\Omega)$ where the singular part $E^su$ of the measure $Eu$ is 
concentrated on the $(n-1)$-rectifiable set of approximate discontinuity points $J_u$
(see Section~\ref{prelBD} for the precise notation).

It is a classical problem in the Calculus of Variations to determine the lower semicontinuous envelope of the energies in \eqref{e:int1} and \eqref{e:int2} in order to find the limits of minimizing sequences lying in the larger space $BD(\Omega)$. 
More precisely, let $F$ be the functional either in \eqref{e:int1} or in \eqref{e:int2} if $u$ belongs to $LD(\Omega)$ or $SBD(\Omega)$  respectively, and $+\infty$ otherwise on $L^1(\Omega;\R^n)$. 
Then, the $L^1(\Omega;\R^n)$ lower semicontinuous envelope of the functional $F$, that is the greatest functional less or equal than $F$ which is $L^1$ lower semicontinuous, is given by 
\[
 \F(u):=\inf\Big\{\liminf_{j\to+\infty}F(u_j):\,u_j\to u \text{ in $L^1(\Omega;\R^n)$}\Big\},
\]
provided some coercivity assumptions on the integrands are imposed (cf. \cite{fonseca2007modern}).

A suitable localized version of the functional $\F$ to the family of open subsets of $\Omega$ turns out to be a variational functional according to Dal Maso and Modica \cite{DalMasoModica80} naturally defined on the space $BD(\Omega)$ (see Section~\ref{ss:variational functional} below and the discussion in what follows). 
Therefore, more generally, we consider variational functionals in this sense and prove for them in Theorem~\ref{thm:mainthm} an integral representation result following closely the celebrated 
global method for relaxation developed in Bouchitt{\'e}, Fonseca and Mascarenhas \cite{bouchitte1998global} to deal with the 
analogous problem for functionals defined either on Sobolev spaces or on the space $BV$ of functions with Bounded Variation (see \cite{ContiFocardiIurlano17} for an extensive survey of the subject and an exhaustive bibliography).

The integral representation result in Theorem~\ref{thm:mainthm} will be applied to the above mentioned relaxation problems for the 
functionals of the type \eqref{e:int1} and \eqref{e:int2}. 
More precisely, in Theorems~\ref{t:relaxationLD} and \ref{t:relaxationSBD} we show that the resulting $L^1$ lower semicontinuous envelope 
$\F$ has indeed an integral representation of the type
\begin{align}\label{e:int3}
\F(u) =\int_\Omega&  f\big(x,u(x),e(u)(x)\big)\d x\nonumber\\
&+\int_{J_u}g\big(x,u^-(x),u^+(x),\nu_u(x)\big)\d \H^{n-1}(x)
+\int_\Omega  f^\infty\left(x,u(x),\frac{\d E^cu}{\d|E^cu|}(x)\right)\d|E^cu|
\end{align}
for all $u\in BD(\Omega)$. Here, $f^\infty$ denotes the (weak) recession function of the integrand $f$.
Moreover, a characterization of the energy densities $f$ and $g$ is given in terms of asymptotic Dirichlet problems involving $F$ itself with boundary values related to the infinitesimal behaviours of the function $u$ around the base point $x$.

Apart from the usual lower semicontinuity (and therefore locality), growth conditions and measure theoretical properties to be satisfied by the functional $\F$ (see assumptions (H1)-(H3) 
in Section~\ref{s:mainthm}), we impose two conditions expressing continuity of the energy functional with respect to specific family of rigid motions. 
More precisely, continuity with respect to translations both in the dependent and independent variable is stated in (H4). Such a condition is used for instance in \cite{bouchitte1998global} 
in the $BV$ setting to express the energy density of the Cantor part in terms of the recession function $f^\infty$ of the bulk energy density.
Additionally, in the current $BD$ setting we need to require further assumption (H5), that expresses continuity of the energy with respect to infinitesimal rigid motions. In turn, this condition implies that the bulk energy density depends only on the symmetric part of the relevant matrix. Condition (H5) is crucial for our arguments
both from a technical side and conceptually as we discuss in details in Section \ref{s:discussionH5}. 

In this respect, we emphasize that all integral representation, relaxation, lower semicontinuity results 
available in literature for energies defined on $BD(\Omega)$ (see e.g. \cite{barroso2000relaxation, ebobisse2001note,rindler2011lower,de2017characterization,arroyorabasadephilippisrindler,kosibarindler,DalMasoOrlandoToader}) 
are based on a stronger version of (H5) that imposes invariance of the energy with respect to infinitesimal rigid motions (cf. Remark \ref{r:H5bis}). From a mechanical perspective such a condition reflects a restriction on the material behaviour. Therefore, it is preferable to avoid it, also because of its controversy in the continuum mechanics community (see \cite{Svendsen1999,Murdoch2005}). 
Note that the quoted invariance property with respect to superposed infinitesimal rigid motions would imply 
the integrands in \eqref{e:int3} to be independent of $u$. In our result, though, this explicit dependence is kept, as was the case in the $BV(\Omega)$ setting \cite{bouchitte1998global}. Finally, let us point out that assumption (H5) is actually not needed to give a partial integral representation result on the subspace $SBD(\Omega)$, or more generally on $BD(\Omega)$ but only for the volume and surface terms of the energies, as already noticed in \cite{ebobisse2001note}.

Further possible applications of our main theorem are in the field of homogenization problems 
(cf. Section~\ref{ss:homogenization}), or more generally to problems in which the determination 
of variational limits in terms of $\Gamma$-convergence of energies defined on 
$BD$ are involved (see e.g. \cite{maso2012introduction}, \cite{focardi2014asymptotic,ContiFocardiIurlano18,CVG}). 

We mention that integral representation results for energies defined on distinguished subspaces of $BD$ 
(in particular satisfying a different set of growth conditions different from (H2)) have been recently 
obtained either in the superlinear case in the $2$ dimensional framework in \cite{ContiFocardiIurlano17}, 
or in the space of Caccioppoli affine functions in \cite{FrSol19}.

Let us now summarize the contents of the paper. In Section~\ref{s:mainthm} we state Theorem~\ref{thm:mainthm}
the main result of the paper, all the preliminaries needed to prove it are provided in Section~\ref{sec:blowupcantor}. 
Section~\ref{s:blowupCantor2} focuses on the analysis of the Cantor part of the energy and more precisely on its integral representation. 
In turn, those results are used in Section~\ref{s:proofmainresult} to establish Theorem~\ref{thm:mainthm}. 
Applications of Theorem~\ref{thm:mainthm} to several issues related to energies with linear growth defined on $BD$ are 
studied in Section~\ref{s:applications}.
More precisely, the relaxation of variational integrals is the topic of Section~\ref{ss:relax},
the lower semicontinuity either of bulk or of bulk and surface energies is investigated in Section~\ref{ss:relaxbulksurface}, 
eventually Section~\ref{ss:homogenization} deals with the periodic homogenization of bulk type energies.
In the final Section~\ref{s:discussionH5} an example of a quasiconvex function $f:\mathbb{M}^{n\times n}\to[0,+\infty)$ being bounded
by the norm of the symmetric part of the relevant matrix, but on the other hand depending non-trivially also on the skew-symmetric part, 
is provided following a celebrated example by M\"uller \cite[Theorem~1]{Muller1992}. The issue of relaxation on $BD$ for the associated 
bulk energy functional is discussed, highlighting the role of assumption (H5) to deduce Theorem~\ref{thm:mainthm} and related open problems.

\section{Main result}\label{s:mainthm}

\subsection{Basic notation}\label{ss:basic}
The unitary vectors of the standard coordinate basis of $\R^n$ will be denoted by $\euno,\ldots,\en$.
$\mathbb{M}^{n\times n}$ stands for the set of all $n\times n$ matrices and 
$\mathbb{M}_{sym}^{n\times n}$, $ \mathbb{M}_{skew}^{n\times n}$ for the subsets of all symmetric 
and skew-symmetric matrices, respectively. 

For a given a set $E$ we adopt the notation $E(x_0,r):=x_0+rE$ for the rescaled copy of size $r>0$ translated in $x_0$. 
In particular, $Q^{\nu}(x_0,r)$ stands for any cube centered at $x_0$, with edge length $r$ and with one face orthogonal 
to $\nu$. 
We also adopt the convention that, whenever 
$\nu$ is omitted then $\nu=\mathrm{e}_i$ for some $i\in\{1,\ldots,n\}$, and the corresponding cube $Q(x_0,r)$ is oriented according to the coordinate directions.

In what follows, we shall often consider a function $\Psi:(0,+\infty)\to [0,+\infty)$ such that 
$\displaystyle{\lim_{t\to 0^+}}\Psi(t)=0$, referring to it as a modulus of continuity. 

Finally, throughout the paper $\Omega$ will denote a non empty, bounded, open subset of $\R^n$ with Lipschitz boundary,
and $\mathcal{O}(\Omega)$, $\mathcal{O}_{\infty}(\Omega)$ denote the families of all the open 
subsets of $\Omega$ and of all the open subsets of $\Omega$ with Lipschitz boundary, respectively.

\subsection{Framework and main result}\label{ss:variational functional}

We consider a class of local energies typically arising in variational problems: 
$\mathcal{F}:L^1(\Omega;\R^n)\times \mathcal{O}(\Omega)\to[0,+\infty]$, 
and $BD(\Omega)$ is the set of maps with {\em{Bounded Deformation}}
(for the precise definition, the notation used in what follows, and several properties of $BD$ 
functions see Section~\ref{sec:blowupcantor}). 
We assume that the following properties are in force on $\mathcal{F}$:
\begin{itemize}
\item[(H1)] $\F(\cdot, A) $ is strongly $L^1(A;\R^n)$ lower semicontinuous for all $A\in\mathcal{O}(\Omega)$\footnote{In the rest of the paper, we shall write $L^1$ lower semicontinuous 
in place of strongly $L^1$ lower semicontinuous for the sake of simplicity.};
\item[(H2)] There exists a constant $C>0$ such that for every 
$(u,A)\in BD(\Omega)\times\mathcal{O}(\Omega)$,
        \begin{equation}\label{eqn:growth condition}
           \frac1C|Eu|(A)\leq \F(u,A) \leq C (\L^n(A)+|E u|(A));
        \end{equation}
\item[(H3)] $\mathcal{F}(u,\cdot)$ is the restriction to $\mathcal{O}(\Omega)$ of a Radon measure for every $u\in BD(\Omega)$;
\item[(H4)] There exists a modulus of continuity $\Psi$ such that 
        \begin{equation}\label{eqn:continuity of energy hyp}
            |\mathcal{F}(\mathrm{v}+u(\cdot - x_0),x_0+A) - \mathcal{F}(u,A)|\leq 
            \Psi(|x_0|+|\mathrm{v}|)(\L^n(A)+|E u|(A) )
        \end{equation}
       for all $(u,A,\mathrm{v},x_0)\in BD(\Omega)\times\mathcal{O}(\Omega)\times\R^n\times\Omega$, with $x_0+A\subset \Omega$;
\item[(H5)] There exists a modulus of continuity $\Psi$  such that 
 \begin{equation}\label{eqn:continuity of energy hyp 2}
 |\mathcal{F}(u+\LL(\cdot-x_0),A)-\mathcal{F}(u,A)|\leq \Psi(|\LL|\mathrm{diam}(A))(\L^n(A)+|Eu|(A))
 \end{equation}
 for every $(u,A,\LL) \in BD(\Omega)\times \mathcal{O}(\Omega)\times\mathbb{M}_{skew}^{n\times n}$, and for all $x_0\in A$.
\end{itemize}

\begin{remark}
It is well-known that assumption (H1) implies locality of $\F(\cdot,A)$ for all $A\in \mathcal{O}(\Omega)$.
Namely, if $u=v$ $\L^n$ a.e. on $A$ then $\F(u,A)=\F(v,A)$. 
\end{remark}

\begin{remark}
Hypothesis (H5) implies that the energy depends only on the symmetric gradient (see formula \eqref{eqn:invariance} and Section~\ref{s:discussionH5} 
for more details).
\end{remark}

Following the global method for relaxation introduced by Bouchitt{\'e}, Fonseca and Mascarenhas in \cite{bouchitte1998global}
we consider the \textit{local} Dirichlet problem
    \begin{equation}\label{e:mm}
    \mm(u,A):=\inf\big\{\mathcal{F}(v,A):\, v\in BD(\Omega),\ v|_{\partial A}=u|_{\partial A}\big\},
    \end{equation}
    where $(u,A)\in BD(\Omega)\times\mathcal{O}_{\infty}(\Omega)$ is given, 
    and prove the ensuing result.

\begin{theorem}\label{thm:mainthm}
Let $\mathcal{F}:L^1(\Omega;\R^n)\times \mathcal{O}(\Omega)\to[0,+\infty]$ be satisfying (H1)-(H5). 
Then, for all $(u,A)\in BD(\Omega)\times \mathcal{O}(\Omega)$
\begin{align*}
        \F(u,A)=&\int_{A} f\big(x,u(x),e(u)(x)\big) \d x+ \int_{J_u\cap A} g\big(x,u^-(x),u^+(x),\nu_u(x)\big) \d\H^{n-1}(x)\\
        &+ \int_{A} f^\infty\Big(x,u(x),\frac{\d E^cu}{\d |E^cu|}(x)\Big)\d|E^cu|(x),
    \end{align*}
where for all $(x_0,\mat,\mathrm{v},\mathrm{v}^-,\mathrm{v}^+,\nu)\in \Omega\times\mathbb{M}_{sym}^{n\times n}\times(\R^n)^4$
    \begin{equation}\label{e:f}
    f(x_0,\mathrm{v},\mat):=\limsup_{\e\rightarrow 0} \frac{\mm(\mathrm{v}+\mat(\cdot-x_0), Q(x_0,\e))}{\e^n},
    \end{equation}
\begin{equation}\label{e:g}
g(x_0,\mathrm{v}^-,\mathrm{v}^+,\nu):=\limsup_{\e\rightarrow 0} \frac{\mm( u_{\mathrm{v}^-,\mathrm{v}^+,\nu}(\cdot-x_0), Q^{\nu}(x_0,\e))}{\e^{n-1}},
\end{equation}
\begin{equation}\label{e:cantor}
f^\infty(x_0,\mathrm{v},\mat):= \limsup_{t\rightarrow+\infty}\frac{ f(x_0,\mathrm{v},t \mat)-f(x_0,\mathrm{v},0)}{t},
\end{equation}
and 
    \begin{equation}\label{e:uvpiuvmeno}
    u_{\mathrm{v}^-,\mathrm{v}^+,\nu}(y):=\left\{
        \begin{array}{cc}
           \mathrm{v}^+  & \text{if $y\cdot \nu\geq 0$} \\
           \mathrm{v}^- &  \text{otherwise}.
        \end{array}
        \right.
    \end{equation}
\end{theorem}
Note that $f^\infty$ is classically termed the \textit{weak} recession function, in contrast to the \textit{strong} recession function for which the limit is assumed to
exists (see the comments in Section~\ref{ss:relax}). 
Its finiteness is guaranteed by the linear growth of $f$ (see \eqref{e:flingr}).

We point out that the analogous result to Theorem~\ref{thm:mainthm} for functionals defined on the space $BV$ of functions with bounded variation 
has been proven under the sole assumptions (H1)-(H4) (cf. \cite[Theorems~3.7, 3.12]{bouchitte1998global}). 
A detailed discussion on the need of assumption (H5) in the $BD$ setting is the topic of Section~\ref{s:discussionH5}.
Several comparisons with the $BV$ case are discussed in Remarks~\ref{r:truncation} and \ref{r:truncation2}.

\section{Preliminaries}\label{sec:blowupcantor}

\subsection{Some results of geometric measure theory}\label{ss:GMT}

In the forthcoming blow-up procedure, it will be mandatory to obtain  limits satisfying additional structural properties. 
To this aim we introduce some useful concepts of geometric measure theory. Let $U\subseteq\R^n$ be either open or closed. Here and in what follows $\mathcal{M}_{loc}(U;\R^k)$ stands for the sets of all $\R^k$-valued Radon measures on $U$, and  $\mathcal{M}(U;\R^k)$ for the subset of all $\R^k$-valued finite Radon measures on $U$. 

Let $\mu_i$, $\mu\in\mathcal{M}_{loc}(U;\R^k)$, following \cite{AFP00} we say that $\mu_i$ locally weakly* converge to $\mu$, 
and we write $\mu_i\wt \mu$ in $\mathcal{M}_{loc}(U;\R^k)$, if
    \[
    \lim_{i\to+\infty}\int_{U} \varphi \d \mu_i =\int_{U} \varphi \d \mu \qquad\text{for all $\varphi\in \mathcal C_c^0(U)$}.
    \]
If $\mu_i$, $\mu\in\mathcal{M}(U;\R^k)$, we say that $\mu_i$ weakly* converges to $\mu$ if the condition above
holds for all $\varphi\in \mathcal C_0^0(U)$. The following property holds true.
\begin{lemma}[Proposition~1.62~(b) \cite{AFP00}]\label{lemma: weakconv}
Let $\{\mu_i\}_{i\in\N}$ be locally weakly* converging to $\mu$ in $\mathcal{M}_{loc}(U;\R^k)$, 
and $\{|\mu_i|\}_{i\in\N}$ be locally weakly* converging to $\lambda$ in $\mathcal{M}_{loc}(U)$. 
Then, $\lambda\geq|\mu|$, and $\displaystyle\mu(E)=\lim_{i\to\infty}\mu_i(E)$ for every relatively 
compact Borel subset $E$ of $U$ with $|\lambda|(\partial E)=0$.
\end{lemma}
    
We use standard notations for the push-forward of measures, and in particular, given 
$\mu\in \mathcal{M}_{loc}(\R^n;\R^k)$, $x\in\R^n$ and $\e>0$, we will often consider the push forward with the map 
$F^{x,\e}(y):=\frac{y-x}{\e}$ defined as
	\begin{equation}\label{e:pushforward}
	F_{\#}^{x,\e}\mu(A):=\mu(x+\e A).
	\end{equation}
Preiss' tangent space $\Tan(\mu,x)$ at a given point $x\in\R^n$, is defined as the subset of non zero measures 
 $\nu\in\mathcal{M}_{loc}(\R^n;\R^k)$ 
such that $\nu$ is the local weak* limit 
of $\sfrac1{c_i}F_\#^{x,\e_i}\mu$, for some sequence $\{\e_i\}_{i\in \N}\downarrow 0$ 
as $i\uparrow+\infty$ and for some positive sequence $\{\e_i\}_{i\in\N}$ 
(see \cite{mattila1999geometry}, \cite{AFP00}, \cite{rindler2018calculus}). To ensure that the total variation is preserved 
along the blow-up limit procedure we recall the ensuing result. 
\begin{lemma}[Tangent measure with unit mass, Lemma~10.6 \cite{rindler2018calculus}]\label{lem:unit mass conve} 
Let $\mu\in \mathcal{M}_{loc}(\R^n;\R^k)$. 
Then, for $|\mu|$-a.e. $x\in \R^n$ and for every 
bounded, open, convex set $K$ the following assertions hold
\begin{itemize}
	\item[(a)] There exists a tangent measure $\gamma\in \Tan(\mu,x)$ such that $|\gamma|(K)=1$, 
	$|\gamma|(\partial K)=0$;
	\item[(b)] There exists 
	$\{\e_i\}_{i\in \N}\downarrow 0$ as $i\uparrow+\infty$ such that 
	$\frac{F^{x,\varepsilon_i}_{\#}\mu}{F^{x,\e_i}_{\#}|\mu|(K)}\wt \gamma$ in $\mathcal{M}(\overline{K};\R^k)$.
\end{itemize}	 
\end{lemma}
Finally, with the help of the next result, we will be able to select a blow-up with a partial affine structure. 
\begin{theorem}[Tangent measures to tangent measures are tangent measures, Theorem~14.16 \cite{mattila1999geometry}]\label{thm:tangent measure are tangent}
Let $\mu\in \mathcal{M}_{loc}(\R^n;\R^k)$ be a Radon measure. Then for $|\mu|$-a.e. $x\in \R^n$ any $\nu\in\Tan(\mu,x)$ satisfies the following properties
\begin{itemize}
\item[(a)] For any convex set $K$, $\frac{F^{y,\rho}_{\#} \nu}{F^{y,\rho}_{\#}|\nu|(K)} \in  \Tan(\mu,x)$  for all $y\in\spt \nu$ and $\rho>0$;
\item[(b)] $\Tan(\nu,y)\subseteq  \Tan(\mu,x)$ for all $y\in\spt\nu$;
\end{itemize}
\end{theorem}
Note that the original result in \cite[Theorem~14.16 ]{mattila1999geometry} is proven for $k=1$. However, the good properties of tangent space of measures (i.e. \cite[Theorem~2.44]{AFP00} or
\cite[Lemma~10.4]{rindler2018calculus}) allow to immediately extend its validity for generic $k$.
\subsection{Preliminaries on $BD$}\label{prelBD} 
We recall next some basic properties of the space $BD$ needed for our purposes. We refer to \cite{Temam1} for classical theorems, 
while for the fine properties we refer to \cite{ambrosio1997fine} (see also \cite{ebo99b}). 

The space of functions with Bounded Deformation on $\Omega$, $BD(\Omega)$, is the set of all maps $u\in L^1(\Omega;\R^n)$ whose 
symmetrized distributional derivative $Eu$ is a matrix-valued Radon measure. It is a Banach space equipped with the norm 
$\|u\|_{BD(\Omega)}:=\|u\|_{L^1(\Omega,\R^n)}+|Eu|(\Omega)$, where $|\mu|$ stands for the total variation of the Radon measure $\mu$ (see \cite{AFP00}). 
A sequence $\{u_j\}_{j\in\N}$ is said to strictly converge to $u$ in $BD(\Omega)$ if
$$
u_j\to u \mbox{ in } L^1(\Omega;\R^n) \mbox{ and } |Eu_j|(\Omega)\to |Eu|(\Omega),
$$
as $j\to\infty$. 

As shown by Ambrosio, Coscia and Dal Maso in \cite{ambrosio1997fine}, $BD(\Omega)$ maps are approximately differentiable 
$\L^n$-a.e. in $\Omega$, the jump set is $\H^{n-1}$-rectifiable, and $Eu$ can be decomposed as
    \begin{equation}\label{e:Eu decomposition}
    E u = e(u)\L^n\res\Omega + (u^+(x)-u^-(x))\odot \nu_u \H^{n-1}\res J_u + E^cu,
    \end{equation}
where $e(u)=\frac{\nabla u+\nabla u^t}{2}$, 
$\nabla u$ is the approximate gradient of $u$, 
$u^+-u^-$ denotes the jump of $u$ over the jump set $J_u$, 
with $u^\pm$ the traces left by $u$ on $J_u$, 
$\nu_u$ is a unitary Borel vector field normal to $J_u$ (here, $a\odot b:=\frac{1}{2}(a\otimes b+b\otimes a)$, $a$, $b\in\R^n$, 
denotes the symmetrized tensor product), $E^cu$ is the Cantor part of $E u$ defined as $E^cu:=E^su\res(\Omega\setminus J_u)$ 
and $E^su:=E u - e(u)\L^n\res\Omega$ (cf. \cite[Eq. (1.2), Definition~4.1]{{ambrosio1997fine}}). 
Let $S_u$ be the complement of the set of points of approximate continuity of $u$, \cite[Theorem 6.1]{ambrosio1997fine} implies that 
$|Eu|(S_u\setminus J_u)=0$, so that $E^cu=Eu\res C_u$, where 
\begin{equation}\label{e:Cu}
C_u:=\Big\{x\in\Omega\setminus S_u:\,\textstyle{\lim_{r\downarrow 0}\frac{|Eu|(B_r(x))}{r^n}=+\infty,\,
\lim_{r\downarrow 0}\frac{|Eu|(B_r(x))}{r^{n-1}}=0}\Big\}.
\end{equation}
The limits in the definition of $C_u$ can be taken with respect to any family $K(x,r)$, with $K$ a bounded, open, convex set containing 
the origin. We shall often use the previous characterization of $E^cu$ throughout the paper.

The space of special functions of bounded deformation is then defined as
$$
SBD(\Omega)=\{u\in BD(\Omega): E^cu=0\}.
$$
The space  $\mathcal C^\infty(\Omega;\R^n)\cap W^{1,1}(\Omega;\R^n)$ is dense in $BD(\Omega)$ for the strict topology on $BD(\Omega)$.
Moreover, for $\Omega$ an open bounded set with Lipschitz boundary, there exists a surjective, bounded, linear trace operator
$\gamma: BD(\Omega)\longrightarrow L^1(\partial\Omega,\R^n)$ satisfying the following integration by parts formula: for every 
$u\in BD(\Omega)$ and $\varphi\in\mathcal C^1(\bar{\Omega})$,
$$
\int_\Omega u \odot \nabla \varphi \d x+\int_\Omega \varphi \d Eu=\int_{\partial\Omega}\varphi\, \gamma(u)\odot \nu \d\mathcal H^{n-1},
$$
with $\nu$ the unit external normal to $\partial\Omega$. The trace operator is continuous if $BD(\Omega)$ is endowed with the strict topology. For notational simplicity, in what follows $\gamma(u)$
will be denoted simply by $u$ itself.

With the same assumptions on $\Omega$, one also has the following embedding result: $BD(\Omega)\hookrightarrow L^q(\Omega,\R^n)$ is compact
for every $1\leq q<\frac{n}{n-1}$.
In view of compactness, the following holds for $\Omega$ a bounded extension domain (cf \cite{Kohn}, \cite{Temam1}): if  $\{u_j\}_{j\in\N}$ is bounded in $BD(\Omega)$ 
there exists a subsequence that converges to some $u\in BD(\Omega)$ with respect to the $L^1(\Omega;\R^n)$ topology. 
\smallskip

 We recall next a Poincar\'e inequality for $BD$ maps which has been proven in \cite[Theorem~3.1]{ambrosio1997fine}, 
 (see also \cite[Theorem 1.7.11]{ebo99b} and also \cite[Lemma~1.4.1]{ebo99b}).
 To this aim consider the space of infinitesimal rigid motions 
\begin{equation}\label{e:R}
\mathcal{R}:=\{\LL x+\mathrm{v}\,:\, \mathrm{v}\in \R^n,\, \LL\in\mathbb{M}_{skew}^{n\times n}\}\,.
\end{equation}
\begin{theorem}\label{thm:poincare}
 Let $A$ be a bounded, open, connected set with Lipschitz boundary, then 
 \begin{equation}\label{e:characterizationR}
\mathcal{R}=\{\vartheta\in\mathcal{D}'(A;\R^n):\, E\vartheta=0\}.
 \end{equation}
 Moreover, let $\mathfrak{R}:BD(A) \rightarrow \mathcal{R}$ 
 be a linear continuous map which leaves the elements of $\mathcal{R}$ fixed. 
 Then there exists a constant $c=c(A,\mathfrak{R})$ such that for all $u\in BD(A)$
    \[
    \|u-\mathfrak{R}[u]\|_{L^{\frac{n}{n-1}}(A;\R^n)}\leq c|E u|(A). 
    \]
\end{theorem}
In particular, for a bounded, open convex set $K$, denoting by $\nu_{\partial K}$ the unit normal vector to $\partial K$,
let $\mathfrak{R}_K: BD(K) \rightarrow \mathcal{R}$ be the map defined as
	\begin{equation}\label{e:RK}
\mathfrak{R}_{K}[u] (y):= \mathbb{M}_{K}[u]y+b_{K}[u],
	\end{equation}
with
	\begin{equation}\label{e:bK}
b_{K}[u]:=\fint_{K} u \d x\,,
	\end{equation}
	and
\begin{equation}\label{e:MK}
\mathbb{M}_{K}[u]:=\frac{1}{2\L^n(K)}\int_{\partial K}\big(
u\otimes \nu_{\partial K} - \nu_{\partial K} \otimes u\big)\d \H^{n-1}\,.
\end{equation}
For $BV$ maps we express the quantity $\mathbb{M}_{K}$ in \eqref{e:MK} in terms of the skew-symmetric part of 
the total variation measure.
\begin{lemma}\label{l:auxiliary}
Let $A$ be a bounded open set with Lipschitz boundary. Then, for all $u\in BV(A;\R^n)$
\begin{align*}
\int_{\partial A} \big(u \otimes \nu_{\partial A} - \nu_{\partial A} \otimes u\big)\d\H^{n-1}
=Du(A)-(Du)^t(A)\,.
\end{align*}
In particular, for all bounded, open, convex sets $K$, and 
for all $u\in BV(K;\R^n)$ we have
\[
\mathbb{M}_{K}[u]=\frac1{2\mathcal L^n(K)}\big(Du(K)-(Du)^t(K)\big)\,.
\]
\end{lemma}
\begin{proof}
To prove the first equality we use the divergence theorem for scalar $BV$ functions \cite[Corollary~3.89]{AFP00}
\[
\int_A v\,\mathrm{div}\,\Phi\, dx+\int_A\Phi\cdot dDv=
\int_{\partial A}v\, \Phi\cdot\nu_{\partial A} \d\H^{n-1}\,,
\]
for all $v\in BV(A)$, $\Phi\in \mathcal C^1(\bar{A};\R^n)$, to get for any fixed $y\in\R^n$
\[
\int_{\partial A} \big(u \otimes \nu_{\partial A}\big)\,y \d\H^{n-1}=
\int_{\partial A} (\nu_{\partial A}\cdot y)u \d\H^{n-1}=Du(A)y\,,
\]
and moreover
\[
\int_{\partial A} \big(\nu_{\partial A} \otimes u\big)\,y \d\H^{n-1}=
\int_{\partial A} (u\cdot y) \nu_{\partial A}\d\H^{n-1}=(Du)^t(A)y\,.
\]
The very definition of $\mathbb{M}_{K}[u]$ in \eqref{e:MK} and the previous computation
provide the conclusion.
\end{proof}
\begin{proposition}\label{p:Rinvariantset} 
$\mathcal{R}$ is an invariant set for $\mathfrak{R}_K$.
\end{proposition}
\begin{proof}
If $u$ is affine, i.e. $u=\LL x+\mathrm{v}$, $\LL\in\mathbb{M}_{skew}^{n\times n}$ and $\mathrm{v}\in\R^n$, 
we have $b_{K}[u] =\mathrm{v}$ by a simple computation. Furthermore, by taking into account 
that $Du=\LL$, Lemma~\ref{l:auxiliary} implies
	\begin{align*}
	\mathbb{M}_{K}[u]
=\frac{1}{2}(\LL-\LL^t)={\LL}. 
	\end{align*}
\end{proof}
As the trace theorem implies the continuity of $\mathfrak{R}_K:BD(K)\to\mathcal R$, thanks to the previous result, 
Theorem~\ref{thm:poincare} yields the existence of a constant $c>0$ 
depending only on $K$ such that for every $u\in BD(K)$ 
	\begin{equation}\label{eqn:ebo}
	\|u-\mathfrak{R}_{K}[u]\|_{L^{\frac n{n-1}}(K;\R^n)}\leq c\,|E u|(K).
	\end{equation}

\begin{remark}\label{r:scaling}
 Let $u\in BD(\e K)$, where $\e>0$ and $K$ a bounded, open, convex set $K$ containing the origin. 
 Then, for a constant $c$ depending only on $K$ we have
 \begin{equation}\label{e:scalingRK}
\|u-\mathfrak{R}_{\e K}[u]\|_{L^1(\e K;\R^n)}\leq c\,\e |E u|(\e K)\,.
 \end{equation}
This follows from H\"older inequality, \eqref{eqn:ebo} and a scaling argument by considering 
$u_\e(y):=u(\e y)$, $y\in K$, and noting that $\mathfrak{R}_{\e K}[u](\e y)=\mathfrak{R}_{K}[u_\e](y)$, 
$\|u-\mathfrak{R}_{\e K}[u]\|_{L^1(\e K;\R^n)}=\e^n\|u_\e-\mathfrak{R}_K[u_\e]\|_{L^1(K;\R^n)}$, and 
$|E u|(\e K)=\e^{n-1}|Eu_\e|(K)$.
\end{remark}

\subsection{On the Cantor part of the symmetrized distributional derivative}
Recently, the fine properties of $BD$ functions have been complemented with the analog of Alberti's rank-one theorem in the $BV$ setting.
More precisely, we recall the fundamental contribution by De Philippis and Rindler (cf. \cite{de2016structure}).
\begin{theorem}\label{thm:structureof polar}
 Let $u\in BD(\Omega)$. Then, for $|E^cu|$-a.e. $x\in \Omega$ 
    \begin{equation}\label{e:polar}
    \frac{\d E u}{\d |E u|}(x)=\frac{\eta(x)\odot \xi(x)}{|\eta(x)\odot \xi(x)|}
    \end{equation}
for some $\xi,\,\eta:\Omega\to\mathbb{S}^{n-1}$ Borel vector fields.
 \end{theorem}
Next, we state a rigidity result for $BD$ maps with constant polar vector that follows from that established 
in \cite[Theorem~2.10 (i)-(ii)]{DPR19} by taking into account that the measure on the right hand side below is 
in addition positive (see also \cite[Theorem~3.2]{DPR19}). 
 \begin{proposition}\label{prop:rindlerchar}
 If $w\in BD_{loc}(\R^n)$ is such that for some $\eta,\,\xi\in\R^n$
	\begin{equation}\label{e:cantordec}
	E w= \frac{\eta\odot \xi}{|\eta\odot \xi|}|E w|, 
	\end{equation}
then 
\begin{itemize}
 \item[(i)] if $\eta\neq\pm\xi$ 
 	\[
	w(y)=\alpha_1(y\cdot \xi)\eta+\alpha_2(y\cdot \eta)\xi+\LL y+\mathrm{v},
	\]
for some $\alpha_1,\alpha_2\in BV_{loc}(\R)$, 
$\LL\in \mathbb{M}_{skew}^{n\times n}$, $\mathrm{v}\in \R^n$;

\item[(ii)] if $\eta=\pm\xi$
 	\[
	w(y)=\alpha(y\cdot \xi)\xi
	+\LL y+\mathrm{v},
	\]
for some $\alpha\in BV_{loc}(\R)$, 
$\LL\in \mathbb{M}_{skew}^{n\times n}$, $\mathrm{v}\in \R^n$.
\end{itemize}
\end{proposition}
%
%

The next Lemma will be particularly useful when dealing with the anti-symmetric part of the gradient 
in the Cantor part of the measure $E u$ (see \eqref{e:bK} and \eqref{e:MK} for the definitions of 
$b_{K}$ and $\mathbb{M}_{K}$, respectively).
\begin{lemma}\label{lem:killtheantisym}
Let $K$ be a bounded, open, convex set containing the origin. 
For any $u\in BD(\Omega)$ and for $|E^c u|$-a.e. $x_0\in \Omega$ 
    \[
\lim_{\e\rightarrow 0} b_{K(x_0,\e)}[u]=u(x_0),\qquad \ \lim_{\e\rightarrow 0}\e |\mathbb{M}_{K(x_0,\e)}[u]|=0.
    \]
\end{lemma}
\begin{proof}
Let $u\in BD(\Omega)$ be fixed. As noticed in the preliminaries $|E^c u|$-a.e. $x_0\in\Omega$ is a point of approximate continuity for $u$, 
thus $b_{K(x_0,\e)}[u]\rightarrow u(x_0)$ as $\e\downarrow 0$.

For the second part of the statement, we use the computation in \cite[Theorem~6.5, Corollary~6.7]{ambrosio1997fine} implying that for 
$|E^c u|$-a.e. $x_0\in\Omega$ 
    \begin{equation}\label{acmmatrix}
    \lim_{\e\rightarrow 0} \fint_{B(x_0,\e)}|u-d_{B(x_0,\e)}[u]|\d y=0,
    \end{equation}
where
    \[
    d_{B(x_0,\e)}[u]:=\fint_{\partial B(x_0,\e)}u(x)\d\H^{n-1}(x).
    \]
Let $x_0\in C_u$ be a point for which \eqref{acmmatrix} holds, and recall then that $\e^{1-n}|Eu|(B(x_0,\e))\to 0$ as $\e\downarrow 0$.
Define for any $v\in BD(\Omega)$ and for $\e$ sufficiently small 
    \[
    \mathfrak{R}_{\e}^*[v](y):= d_{B(x_0,\e)}[v]+\mathbb{M}_{K(x_0,\e)}[v](y-x_0).
    \]
Arguing as in the proof of Proposition~\ref{p:Rinvariantset} it is immediate to see that $\mathcal{R}$ is an invariant set for 
$\mathfrak{R}_{\e}^*$. In particular, thanks to Theorem~\ref{thm:poincare} we infer for all $v\in BD(\Omega)$ that 
    \[
 \|v-\mathfrak{R}_{\e}^*[v]\|_{L^1(B(x_0,\e);\R^n)}
 \leq c\, \e |Ev|(B(x_0,\e))
    \]
where the constant $c$ is independent from $\e$ (this is obtained with a scaling argument similar to that in
Remark~\ref{r:scaling}). In particular, it follows that
    \begin{align*}
        \int_{B(x_0,\e)} |u(y)-d_{B(x_0,\e)}[u]-\mathbb{M}_{K(x_0,\e)}[u] (y-x_0)|\d y\leq c\, \e |E u|(B(x_0,\e)).
    \end{align*}
Therefore, by the triangular inequality we have
    \begin{align*}
\fint_{B(x_0,\e)}|\mathbb{M}_{K(x_0,\e)}[u] (y-x_0)|\d y\leq \fint_{B(x_0,\e)}|u(y)-d_{B(x_0,\e)}[u]|\d y+ C\frac{|Eu|(B(x_0,\e))}{\e^{n-1}},
    \end{align*}
and thus by the choice of $x_0$ it follows
    \[
    \e\fint_{B(\underline{0},1)}|\mathbb{M}_{K(x_0,\e)}[u]z|\d z=\fint_{B(x_0,\e)}|\mathbb{M}_{K(x_0,\e)}[u](y-x_0)|\d y\rightarrow 0.
    \]
Notice that, the quantity $\mathbb{M}\mapsto\fint_{B(\underline{0},1)}|\mathbb{M}z|\d z$ defines a norm on $\mathbb{M}^{n\times n}$, 
and thus for some constant $C$ depending only on the dimension, we have
    \[
    \e |\mathbb{M}_{K(x_0,\e)}[u]|\leq C\e\fint_{B(\underline{0},1)}|\mathbb{M}_{K(x_0,\e)}[u]z|\d z\rightarrow 0.\qedhere
    \]
\end{proof}

\subsection{Change-of-base formulas}
It is well-known that the chain rule formula does not hold in general for $BD$ maps. 
We provide a simple variation of it that will be useful throughout the paper.
\begin{lemma}\label{lem:changeofvariable}
Let $\ZZ\in\mathbb{M}^{n\times n}$ be invertible, let $w\in BD(\Omega)$ and set
	\[
	\tilde{w}(y):=\ZZ w(\ZZ^ty).
	\]
Then, $\tilde{w}\in BD(\ZZ^{-t}\Omega)$ and
\begin{align*}
E\tilde{w}&= |\det\ZZ|^{-1}
\ZZ (\ZZ^{-t}_{\#} Ew) \ZZ^t.
\end{align*}
Moreover, if $K$ is an open convex set and $v\in BD(K)$ we have
\begin{equation}\label{eqn:rigidbehavior}
	\mathfrak{R}_{\ZZ^{-t}K}[\tilde{v}](y)=\ZZ\mathfrak{R}_K[v] (\ZZ^t y) \ \ \text{for all $y\in \R^n$}.
	\end{equation}
\end{lemma}
\begin{proof}
Let $\varphi \in \mathcal C^{\infty}(\Omega;\R^n)\cap BD(\Omega)$ and define 
$\tilde{\varphi}(y):=\ZZ\varphi(\ZZ^t y)$. 
Clearly, $\tilde{\varphi}\in \mathcal C^{\infty}(\ZZ^{-t}\Omega;\R^n)$, with $\nabla \tilde{\varphi} (y)= 
\ZZ \nabla \varphi (\ZZ^t y) \ZZ^t$ and $e(\tilde{\varphi})(y)= \ZZ e(\varphi) (\ZZ^t y) \ZZ^t$.
Hence, the symmetrized distributional derivative of $\tilde{\varphi}$ is given by
	\[
	E \tilde{\varphi} =|\det\ZZ|^{-1} \ZZ\big(\ZZ^{-t}_{\#}(e(\varphi)\L^n\res\Omega)\big)\ZZ^t\,.
	\]
Finally, if $w\in BD(\Omega)$ we conclude by approximation of $w$ by smooth maps in the $BD$ strict topology.

The last assertion follows from a direct computation.
\end{proof}

\begin{remark}\label{rmk:cov}
We shall often use Lemma~\ref{lem:changeofvariable} to reduce ourselves to the case in which the two vectors $\xi$, $\eta$
in the polar decomposition of \eqref{e:polar} are actually given by $\euno$ and $\edue$. To this aim the following remarks
are useful. Let $w\in BD(K)$ be given by 
	\[
	w(y):=\psi_1(y\cdot \eta) \xi+\psi_2(y\cdot \xi)\eta
	\]
for some $\psi_1, \psi_2\in BV(\R)$ and for $\xi,\eta\in \R^n$ non-parallel unit vectors, i.e.~$\eta\neq\pm\xi$. 
Consider any invertible matrix $\ZZ\in\mathbb{M}^{n\times n}$ such that $\ZZ\eta=\euno,$ $\ZZ\xi=\edue$, 
and the associated function 
	\[
	\tilde{w}(y):=\ZZ (\psi_1(\ZZ^ty\cdot \eta)\xi+\psi_2(\ZZ^ty\cdot \xi)\eta)= \psi_1(y\cdot \euno)\edue+\psi_2(y\cdot \edue)\euno.
	\]	
Then, $\tilde{w}\in BD(\ZZ^{-t}K)$ with $E\tilde{w}=|\det\ZZ|^{-1} \ZZ (\ZZ^{-t}_{\#} Ew) \ZZ^t$. Furthermore, since
\[
\ZZ^{-t}_{\#} Ew=\frac{\eta\odot \xi}{|\eta\odot \xi|} \ZZ^{-t}_{\#}|Ew|
\]
we then have
\[
E\tilde{w}= |\det\ZZ|^{-1}
\frac{\euno\odot \edue}{|\eta\odot \xi|}  \ZZ^{-t}_{\#} |E w|,
\]
in turn implying both
\begin{align}\label{covmass1}
E\tilde{w}= \frac{\euno\odot \edue}{|\euno\odot \edue|}|E\tilde{w}|
\end{align}
and 
\begin{align}\label{covmass2}
|E\tilde{w}|=|\det\ZZ|^{-1}\frac{|\euno\odot \edue|}{|\eta\odot \xi|}  \ZZ_{\#}^{-t}|E w|.
\end{align}
In particular, we conclude that 
	\[
	|E \tilde{w}|(\ZZ^{-t}K)=|\det\ZZ|^{-1}\frac{|\euno\odot \edue|}{|\eta\odot \xi|}  |E w|(K).
	\]
\end{remark}

\subsection{On the cell problem defining $\mm$.} The next two results clarify the link between 
$\mm$ and $\mathcal{F}$. They have been originally proved in \cite{bouchitte1998global} in the $BV$ setting 
and then straightforwardly adapted to the $BD$ setting in \cite{ebobisse2001note}.
\begin{lemma}[Lemma~3.5, Remark~3.6 \cite{bouchitte1998global}, Lemma~3.2 \cite{ebobisse2001note}]\label{lem:derivative of m}
Let $u\in BD(\Omega)$, and set $\mu:=\L^n+|E^s u|$. Then, for any bounded, open, convex set $K$ containing 
the origin we have
    \[
    \lim_{r\rightarrow 0} \frac{\F(u, K(x_0,r))}{\mu(K(x_0,r))}=\lim_{r\rightarrow 0} \frac{\mm(u,K(x_0,r))}{\mu(K(x_0,r))} \qquad 
    \text{for $\mu$-a.e. $x_0\in \Omega$.}
    \]
\end{lemma}
\begin{lemma}\label{lem:continuity if m}
There exists a constant $C>0$ such that for any $u_1,u_2\in BD(\Omega)$, $A\in \mathcal{O}_{\infty}(\Omega)$ 
    \[
    |\mm(u_1;A)-\mm(u_2;A)|\leq C \int_{\partial A} |u_1-u_2|\d \H^{n-1}.
    \]
\end{lemma}
Finally, we refine Lemma~\ref{lem:continuity if m} as a consequence of assumptions (H4) and (H5).
Following \cite[Remark 3.10]{bouchitte1998global}, for all $K$ bounded, open, convex set containing the origin, 
for every 
$(\mat,\mathrm{v}_0,\mathrm{v},x_0,x)\in\mathbb{M}^{n\times n}\times 
(\R^n)^2\times (\Omega)^2$ 
and for every $\e>0$ small enough, hypothesis (H4) implies 
\begin{align}\label{e:stimammcont}
|\mm(\mathrm{v}+\mathrm{v}_0+\mat(\cdot-x-x_0),K(x+x_0,\e))-&
\mm(\mathrm{v}_0+\mat(\cdot-x_0),K(x_0,\e))|&\nonumber\\
\leq C_K\Psi(|x|+|\mathrm{v}|)\big(1+\textstyle{\frac{|\mat+\mat^t|}2}\big)\e^n,
\end{align}
and, in turn, hypothesis (H5) implies
 \begin{align}
      &|\mm(\mathrm{v}_0+\mat(\cdot -x_0),K(x_0,\e))- \mm(\mathrm{v}_0+\textstyle{\frac{\mat+\mat^t}2}(\cdot -x_0),K(x_0,\e))|\nonumber\\
      &\qquad \qquad \leq C_K\Psi\big(\e\,\mathrm{diam}(K)\textstyle{\frac{|\mat-\mat^t|}2}\big)\big(1+\textstyle{\frac{|\mat+\mat^t|}2}\big)\e^n\label{eqn:H5onm}
    \end{align}
for some constant $C_K>0$ depending on $K$ only.

\section{Analysis of the blow-ups of the Cantor part}\label{s:blowupCantor2}

In this section we show how to select a suitable blow-up limit at Cantor type points. To this aim we fix some notation: 
let $u\in BD(\Omega)$, we may assume $u$ to be extended to a map in $BD(\R^n)$ being $\Omega$ Lipschitz 
(cf. \cite[Corollary~1.6.4]{ebo99b}). By a slight abuse of notation we denote by $u$ the extended function.

With fixed a bounded, open, convex set $K$ containing the origin, 
for every $\e>0$ and $x\in\Omega$ set $K(x,\e):=x+\e K$. For $\e$ sufficiently small, consider the associated rescaled 
functions $u_{K,x,\e}:K\to\R^n$ given by
\begin{equation}\label{e:rescaled}
u_{K,x,\e}(y):=
\frac{u(x+\e y)- \mathfrak{R}_{K}[u(x+\e \cdot)](y)}{\e \frac{|E u|(K(x,\e))}{\L^n(K(x,\e)))}},
\end{equation}
where $ \mathfrak{R}_{K}$ is defined in \eqref{e:RK}. Clearly, we may assume that $|E u|(K(x,\e))>0$ 
for all $\e\in(0,\mathrm{dist}(x,\partial\Omega))$, otherwise $u$ would be an infinitesimal rigid 
motion around $x$ (cf. Eq. \eqref{e:characterizationR} in Theorem~\ref{thm:poincare}).
We analyze first basic compactness properties of the rescaled family $\{u_{K,x,\e}\}_{\e>0}$. 
\begin{proposition}\label{p:recalings compactness}
For $|Eu|$-a.e. $x\in\Omega$ there exist a sequence $\e_i\downarrow 0$, a map $w\in BD(K)$, and a measure
$\gamma\in\Tan(Eu,x)$ such that 
\begin{itemize}
\item[(i)] $\{u_{K,x,\e_i}\}_{i\in\N}$ converges to $w$ strictly in $BD(K)$, $\mathfrak{R}_{K}[w]=\underline{0}$
\item[(ii)] $\{\frac1{\L^n(K)}Eu_{K,x,\e_i}\}_{i\in\N}$ converges to $\gamma$ weakly* in $\mathcal{M}(\bar{K};\R^{n\times n})$, 
$|\gamma|(K)=|\gamma(K)|=1$, $|\gamma|(\partial K)=0$, $|\gamma|\in\Tan(|Eu|,x)$, 
$\gamma=\frac{dEu}{d|Eu|}(x)|\gamma|$ $|\gamma|$-a.e. in $\R^n$, and $Ew=\L^n(K)\gamma\res K$. 
\end{itemize}
\end{proposition}
\begin{proof}
We prove the conclusion for all Lebesgue points of the polar vector $\Omega\ni x\mapsto\frac{dEu}{d|Eu|}(x)$.
A simple computation shows that 
	\begin{equation}\label{e:Euscal}
	E u_{K,x,\e}=\L^n(K) \frac{F^{x,\e}_{\#}E u}{F^{x,\e}_{\#}|E u|(K)}\,,
	\end{equation}
(recall the definition of $F^{x,\e}_{\#}$ in \eqref{e:pushforward}) so that $|Eu_{K,x,\e}|(K)=\L^n(K)$.
Lemma~\ref{lem:unit mass conve} yields for a sequence $\e_i\downarrow 0$ that 
$\{\frac1{\L^n(K)}Eu_{K,x,\e_i}\}_{i\in\N}$ converges weakly* in $\mathcal{M}(\bar{K};M^{n\times n})$ 
to some $\gamma\in\Tan(Eu,x)$, with $|\gamma|(K)=1$,
$|\gamma|(\partial K)=0$, $|\gamma|\in\Tan(|Eu|,x)$ and such that $\gamma=\frac{dEu}{d|Eu|}(x)|\gamma|$
$|\gamma|$-a.e. in $\R^n$ (for the last two claims see \cite[Theorem~2.44]{AFP00}, \cite[Lemma~10.4]{rindler2018calculus}), 
then $|\gamma(K)|=1$. 

Moreover, by taking into account Remark~\ref{r:scaling} we get
\[
\|u_{K,x,\e}\|_{L^1(K;\R^n)}=\frac{\L^n(K)}{\e |E u|(K(x,\e))}\|u(x+\cdot)-\mathfrak{R}_{\e K}[u(x+\cdot)](\cdot)\|_{L^1(\e K;\R^n)}\,,
\]
so that for some constant depending only on $K$
\[
 \|u_{K,x,\e}\|_{L^1(K;\R^n)}\leq c\,\L^n(K)\,.
\]
The compact embedding $BD(K)\hookrightarrow L^1(K;\R^n)$ yields that we can extract a subsequence (not relabeled) with $u_{K,x,\e_i}$
converging in $L^1(K;\R^n)$ to some limit map $w$ belonging to $BD(K)$ satisfying $ \mathfrak{R}_{K}[w]=\underline{0}$. 
Therefore, $Eu_{K,x,\e_i}$ converge to $Ew$ weakly* in $\mathcal{M}(K;\R^{n\times n})$, and thus 
$Ew=\L^n(K)\gamma\res K$. The strict convergence in $BD(K)$ of $\{u_{K,x,\e_i}\}_{i\in\N}$ to $w$ 
then follows at once.
\end{proof}
In what follows, any map $w$ given by Proposition~\ref{p:recalings compactness} will be termed blow-up limit. 
If $x$ is a point of approximate differentiability or a jump point, usually introduced with a different definition 
of the rescaled maps, the blow-up limit is well-known to be unique. 
In turn, this implies that the Radon-Nikod\'ym derivative of the functional $\F$ with respect to $|Eu|$ in such points can be 
straightforwardly characterized in terms of asymptotic Dirichlet problems with boundary values given by the blow-up limit itself 
(cf. Lemma~\ref{lem:derivative of m}). 
In contrast, if $x\in C_u$ is a point satisfying \eqref{e:polar}, usually referred to as a Cantor type point, 
the blow-up limit is in general not unique. In order to overcome this difficulty, a double blow-up procedure is performed. 
By means of this argument, we can reduce ourselves to the case of a two dimensional $BV$ map which is affine in one direction.

The strategy of the proof is a slight variation of \cite[Lemma 2.14]{de2017characterization}, which is originally worked out 
in the context of generalized Young measures. We basically follow the lines of such proof by incorporating also the need
of selecting a sequence preserving the mass along the blow-up process. 
We shall improve upon the structure of blow-ups in Proposition~\ref{prop:vgblup} in section~\ref{s:finer}.

\subsection{A double blow-up procedure}

We introduce some notation necessary for the blow-up procedure. Given a couple of vectors $\xi$, $\eta\in\mathbb{S}^{n-1}$ 
(possibly $\xi=\pm\eta$), consider an orthonormal basis $\zeta_i$ of $\mathrm{span}\{\xi,\eta\}^{\perp}$ 
 (thus for $n=2$ either $\mathrm{span}\{\xi,\eta\}^{\perp}=\{\underline{0}\}$ if $\xi\neq\pm \eta$ or $i=1$ if $\xi=\pm\eta$, 
 and for $n\geq 3$ either $1\leq i\leq n-2$ if $\xi\neq\pm \eta$ or $1\leq i\leq n-1$ if $\xi=\pm\eta$), 
 then for all $\rho>0$ define the bounded, open, convex set containing the origin
 \begin{equation}\label{eqn:convexdefinition}
 P^{\xi,\eta}_{\rho}:= \big\{y\in \R^n:\,|y\cdot \eta|\leq \sfrac\rho2,\, |y\cdot \xi|\leq \sfrac12,\,
 |y\cdot \zeta_i|\leq \sfrac12\,\, \text{for $1\le i\le n-2$}\big\}, 
 	\end{equation}
 if $\xi\neq \pm\eta$, and otherwise if $\xi=\pm\eta$
 \begin{equation}\label{eqn:convexdefinition1}
 P^{\xi,\eta}_{\rho}:= \big\{y\in \R^n:\,|y\cdot \eta|\leq \sfrac\rho2,\,
 |y\cdot \zeta_i|\leq \sfrac12\,\, \text{for $1\le i\le n-1$}\big\}. 
 \end{equation}
 We underline that the role of $\eta$ and $\xi$ is not symmetric in the definition of  $P_{\rho}^{\xi,\eta}$
(in this respect see the comments right before Case~1 in the ensuing proof). Moreover, we do not highlight the dependence of 
$P_{\rho}^{\xi,\eta}$ on $\zeta_i$ not to further overburden the notation. In any case the specific choice of $\zeta_i$ is not relevant 
for the arguments that follow.

With this notation at hand, we can state the key result to prove the integral representation of the Cantor part (recall the definitions
of $F^{x,\e}_{\#}Eu$ given in \eqref{e:pushforward} and that of $u_{K,x,\e}$ given in \eqref{e:rescaled}).

\begin{proposition}[blow-up at $|E^cu|$-a.e. point]\label{prop:vgblupRinDe}
Let $u\in BD(\Omega)$. Then for $|E^c u|$-a.e. $x\in \Omega$ and for every $\rho>0$ there exist an infinitesimal sequence $\{\e_i\}_{i\in \N}$,
vectors $\xi,\,\eta\in\mathbb{S}^{n-1}$, bounded, open, convex set containing the origin $P^{x}_{\rho}:=P^{\xi,\eta}_{\rho}$ (cf. \eqref{eqn:convexdefinition} and \eqref{eqn:convexdefinition1}), and a map $v_{\rho} \in BV(P^{x}_{\rho};\R^n)$ such that 
$u_{P^{x}_{\rho},x,\e_i}$ converge to $v_{\rho}$ strict in $BD(P^{x}_\rho)$, where
\begin{equation}\label{eqncase1}
 v_{\rho}(y)= \bar{\psi}_{\rho}( y\cdot \eta) \xi + (y\cdot \xi) \bar{\beta}_{\rho}\,\eta+\LL_{\rho}y +\mathrm{v}_{\rho},
\end{equation}
for some $\bar{\psi}_{\rho}\in BV_{loc}(\R)$, $\bar{\beta}_{\rho} \in \R$, 
$\mathrm{v}_{\rho}\in \R^n$, and $\LL_{\rho}\in \mathbb{M}_{skew}^{n\times n}$.
Moreover, $v_{\rho}$ satisfies
\begin{align}\label{eqn:relationasymptotics}
	& \mathfrak{R}_{P^{x}_{\rho}}[v_{\rho}]=\underline{0},\qquad
	|E v_{\rho}|(P^{x}_{\rho})=\L^n(P^{x}_{\rho}).
\end{align}
and
	\[
	\frac{\d E v_{\rho}}{\d |E v_{\rho} |}(y)=\frac{\eta\odot \xi}{|\eta\odot \xi|}\qquad 
	\text{$|E v_{\rho}|$-a.e. on $P^{x}_{\rho}$}.
	\]
\end{proposition}
\begin{proof}
We divide the proof in two steps, each corresponding to a blow-up procedure.
\smallskip

\noindent\textbf{First blow-up.} 
We perform a first blow-up in a point $x$ satisfying all the conditions listed in (I)-(III) that follows. 
More precisely, consider the subset of points $x\in C_u$ (so that $x$ is a point of approximate continuity 
of $u$ (see \eqref{e:Cu})) having the following additional properties:
\begin{itemize}
\item[(I)] $\frac{\d E u}{\d |E u|}(x)= \frac{\d E^c u}{\d |E^c u|}(x)= \frac{\eta(x)\odot \xi(x)}{|\eta(x)\odot \xi(x)|}$, 
for some vectors $\xi(x),\eta(x)\in\mathbb{S}^{n-1}$, and $\Tan(|Eu|,x)=\Tan(|E^cu|,x)$;
\item[(II)] Proposition~\ref{p:recalings compactness} holds true with $K:=\ZZ^t Q(\underline{0},\rho)$ if $\eta(x)\neq\pm\xi(x)$,
where $\mathbb{B}$ is any invertible matrix  such that $\mathbb{B}\eta(x)=\euno$ and $\mathbb{B}\xi(x)=\edue$,  
and $K:=P_{\rho}^{\xi(x),\eta(x)}$ if $\eta(x)=\pm\xi(x)$: we extract a subsequence (not relabeled) such that 
$u_{K,x,\e_i}$ converge to some map $w$ strict in $BD(K)$, with $ \mathfrak{R}_{K}[w]=\underline{0}$, 
$\{\frac1{\L^n(K)}Eu_{K,x,\e_i}\}_{i\in\N}$ converges weakly* in $\mathcal{M}(\bar{K};\R^{n\times n})$ to some 
$\gamma\in\Tan(Eu,x)$ such that $|\gamma|(K)=|\gamma(K)|=1$, $|\gamma|(\partial K)=0$, 
$|\gamma|\in\Tan(|Eu|,x)$, $\gamma=\frac{dEu}{d|Eu|}(x)|\gamma|$ 
$|\gamma|$-a.e. in $\R^n$, and $Ew=\L^n(K)\gamma\res K$;

\item[(III)] $\Tan(|Ew|, z)\subseteq \Tan(|E^c u|,x)$ for all $z\in \spt|Ew|$.
\end{itemize}
Notice that the set of points where either (I) or (II) or (III) fails is $|E^c u|$-negligible thanks to 
Theorem~\ref{thm:structureof polar}, to the locality of Preiss' tangent space to a measure, to 
Proposition~\ref{p:recalings compactness} itself and to Theorem~\ref{thm:tangent measure are tangent}. 
Proposition~\ref{prop:rindlerchar} describes the structure of the blow-up limit $w$ in (II) in details: 
 \begin{itemize}
 \item if $\eta(x)\neq\pm\xi(x)$: we can find two maps $\alpha_1,\alpha_2\in BV_{loc}(\R)$, 
 $\LL_\rho\in \mathbb{M}_{skew}^{n\times n}$ and $\mathrm{v}_\rho\in \R^n$, such that 
 \[
 w(y)=\alpha_1(y\cdot \eta(x)) \xi(x)+ \alpha_2(y\cdot \xi(x)) \eta(x)+\LL_{\rho}y+\mathrm{v}_{\rho}\,,
 \]

\item if $\eta(x)=\pm\xi(x)$: we can find a map $\alpha\in BV_{loc}(\R)$, 
$\LL_\rho\in \mathbb{M}_{skew}^{n\times n}$ and $\mathrm{v}_\rho\in \R^n$, such that
 \[
 w(y)=\alpha(y\cdot \xi(x)) \xi(x)+\LL_{\rho}y+\mathrm{v}_{\rho}\,.
 \]
\end{itemize}
Thus, if $\eta(x)=\pm\xi(x)$ we conclude by setting $v_\rho:=w$ and with  $P_\rho^x:=P_{\rho}^{\xi(x),\eta(x)}$. 
Otherwise, if $\eta(x)\neq\pm\xi(x)$,  we are forced to take a second blow-up to prove that (at least) one between the 
$\alpha_k$'s can be taken affine.  
\smallskip 

\noindent\textbf{Second blow-up if $\eta(x)\neq\pm\xi(x)$.}
First, we change variables by means of the invertible matrix $\mathbb{B}$ introduced in item (II) above.
Following Remark~\ref{rmk:cov}, with $\tilde{\mathbb{L}}_{\rho}:=\mathbb{B}\mathbb{L}_{\rho}\mathbb{B}^t$ and 
$\tilde{\mathrm{v}}_{\rho}:=\mathbb{B}\mathrm{v}_{\rho}$,  we consider the associated map 
\[\tilde{w}(y)=\alpha_1(y\cdot \euno) \edue+ \alpha_2(y\cdot \edue)\euno
+\tilde{\mathbb{L}}_{\rho} y+\tilde{\mathrm{v}}_{\rho}\,.
\]
We blow-up $\tilde{w}$ around a suitable point $y\in\ZZ^{-t}K=
Q(\underline{0},\rho)$, distinguishing two cases depending on the distributional 
derivatives of the $\alpha_k$'s. We note that the vectors $\eta$ and $\xi$ in the statement correspond exactly 
to the two vectors provided by the polar decomposition of $E^cu$ at $x$, 
$\eta(x)$ and $\xi(x)$, respectively, if $D^s \alpha_1\neq 0$. In this case we 
also set $P_\rho^x:=P_\rho^{\xi(x),\eta(x)}$, the latter set being introduced in \eqref{eqn:convexdefinition}.
Instead, if $D^s \alpha_1=0$ and $D^s \alpha_2\neq 0$,
then $\eta$ corresponds to $\xi(x)$ and $\xi$ to $\eta(x)$, and 
$P_\rho^x:=P_\rho^{\eta(x),\xi(x)}$.
Finally, if both $D^s \alpha_1=D^s \alpha_2=0$, we choose
$\eta=\eta(x)$ and $\xi=\xi(x)$ and set $P_\rho^x:=P_\rho^{\xi(x),\eta(x)}$
(actually, the opposite choice would be fine as well).

\smallskip

\textbf{Case $1$:} \textit{either $D^s \alpha_1\neq 0$ or $D^s \alpha_2\neq 0$}. Without loss of generality 
we may assume that $D^s \alpha_1 \neq 0$. Set $\tilde{P}_{\rho}^{x}:= \ZZ^{-t} P_{\rho}^{x}$ 
(the matrix $\ZZ$ has been introduced above), and select a point $y\in Q(\underline{0},\rho)$ such that 
\begin{itemize}
\item[(a)] $\frac{\d E \tilde{w}}{\d |E \tilde{w}|}(y)
=\frac{\euno\odot \edue}{|\euno\odot \edue|}$, and $\delta^{-n}|E\tilde{w}|(\tilde{P}^x_\rho(y,\delta))\to+\infty$ 
as $\delta\downarrow 0$;
\item[(b)] Proposition~\ref{p:recalings compactness} holds true: $\tilde{w}_{\tilde{P}_{\rho}^{x},y,\delta_i}$, for 
some sequence $\delta_i\downarrow 0$, converge to a map $\tilde{g}_\rho$ strict in $BD(\tilde{P}_{\rho}^{x})$ with 
$\mathfrak{R}_{\tilde{P}^x_\rho}[\tilde{g}_\rho]=\underline{0}$, 
and $\{\frac1{\L^n(\tilde{P}_{\rho}^{x})}E\tilde{w}_{\tilde{P}_{\rho}^{x},y,\delta_i}\}_{i\in\N}$ converges weakly* in 
$\mathcal{M}(\bar{\tilde{P}}_{\rho}^{x};\R^{n\times n})$ to some $\tilde{\gamma}\in\Tan(E\tilde{w},y)$ such that 
$|\tilde\gamma|(\tilde{P}_{\rho}^{x})=|\tilde\gamma(\tilde{P}_{\rho}^{x})|=1$, 
$|\tilde{\gamma}|(\partial \tilde{P}_{\rho}^{x})=0$, $|\tilde{\gamma}|\in\Tan(|E\tilde{w}|,y)$, and 
$\tilde\gamma=\frac{\d E \tilde{w}}{\d |E \tilde{w}|}(y)|\tilde\gamma|$ $|\tilde\gamma|$-a.e. on $\R^n$,
$E\tilde{g}_\rho=\L^n(\tilde{P}_{\rho}^{x})\tilde{\gamma}\res \tilde{P}_{\rho}^{x}$; 

\item[(c)] $\displaystyle \lim_{r\rightarrow 0}\fint_{y\cdot \edue-r}^{y\cdot \edue+r}
| \alpha_2'(s)- \alpha_2'(y\cdot \edue)| \d s =
\displaystyle \lim_{r\rightarrow 0} \frac{1}{2r}|D^s  \alpha_2|(y\cdot \edue-r,y\cdot \edue+r)=0$.
\end{itemize}	 
Since for $\L^1$-a.e. $t\in \R$ it holds
\[
  \lim_{r\rightarrow 0} \fint_{ t-r}^{t+r} | \alpha_2'(s)- \alpha_2'(t)| \d s =0,\quad
  \lim_{r\rightarrow 0} \frac{1}{2r}|D^s \alpha_2|(t-r,t+r)=0,
\]
conditions (c) is true for $|D^s  \alpha_1| \otimes \L^{n-1}$-a.e. $y\in Q(\underline{0},\rho)$. 
First, note that $|D^s\alpha_1|\otimes \L^{n-1}$ is non trivial by assumption; then if $I\subseteq\R$ is the subset of 
points of full $\L^1$ measure for which the previous two conditions hold true, we conclude that 
\[
\big(|D^s  \alpha_1| \otimes \L^{n-1}\big)\big((-\sfrac\rho2,\sfrac\rho 2)\times((-\sfrac\rho2,\sfrac\rho 2)\setminus I)
\times(-\sfrac\rho2,\sfrac\rho2)^{n-2}\big)=0.
\]
On the other hand, conditions (a)-(b) hold $|E^s\tilde{w}|$-a.e. on $Q(\underline{0},\rho)$ in view of 
Theorem~\ref{thm:structureof polar}, the decomposition of $Ew$ in \eqref{e:Eu decomposition}, and 
Proposition~\ref{p:recalings compactness} itself. 
Since the measures $|D^s \alpha_1|\otimes \L^{n-1}$ and $|D\alpha_1| \otimes |D^s \alpha_2|\otimes \L^{n-2}$ are mutually singular,  
the measure $|D^s \alpha_1|\otimes \L^{n-1}$ is absolutely continuous with respect to $|E^s \tilde{w}|$, 
thus we conclude that (a)-(c) hold for $|D^s\alpha_1|\otimes \L^{n-1}$-a.e. $y\in Q(\underline{0},\rho)$.

Then, fix a point $y\in Q(\underline{0},\rho)$ for which (a)-(c) are satisfied,
we show that the second blow-up limit $\tilde{g}_\rho$ satisfies
\begin{equation}\label{e:grho structure 1}
\tilde{g}_\rho(z)=\psi_\rho(z\cdot \euno)\,\edue+\tilde{\mathbb{L}}_{\rho} z+\tilde{\mathrm{v}}_{\rho}\,,
\end{equation}
for some $\psi_\rho\in BV_{loc}(\R)$, $\tilde{\mathrm{v}}_\rho\in\R^n$, and $\tilde{\LL}_\rho\in\mathbb{M}^{n\times n}_{skew}$.
We make a slight abuse of notation since the latter two quantities maybe different from those in the definition of $\tilde{w}$,
but as their role is inessential we keep the same symbols.
To this aim, we split $E \tilde{w}_{\tilde{P}^{x}_{\rho},y,\delta_i}$ as follows
\[
 E \tilde{w}_{\tilde{P}^{x}_{\rho},y,\delta_i}=E \tilde{w}^{(1)}_i+E \tilde{w}^{(2)}_i\,,
\]
where we have set
\begin{equation*}
\tilde{w}^{(1)}_i(z):=\frac{\alpha_1((y+\delta_i z)\cdot\mathrm{e}_1)\mathrm{e}_2-
\mathfrak{R}_{\tilde{P}^{x}_\rho}[\alpha_1((y+\delta_i \cdot)\cdot\mathrm{e}_1)\mathrm{e}_2](z)}
{\delta_i\frac{|E \tilde{w}|(\tilde{P}^{x}_\rho(y,\delta_i))}{\L^n(\tilde{P}^{x}_\rho(y,\delta_i))}}\,,
\end{equation*}
and
\begin{equation*}
\tilde{w}^{(2)}_i(z):=\frac{\alpha_2((y+\delta_i z)\cdot\mathrm{e}_2)\mathrm{e}_1-
\mathfrak{R}_{\tilde{P}^{x}_\rho}[\alpha_2((y+\delta_i \cdot)\cdot\mathrm{e}_2)\mathrm{e}_1](z)}
{\delta_i\frac{|E \tilde{w}|(\tilde{P}^{x}_\rho(y,\delta_i))}{\L^n(\tilde{P}^{x}_\rho(y,\delta_i))}}\,.
\end{equation*}
First, we claim that as $i\uparrow+\infty$
\begin{equation}\label{e:claim w}
|E\tilde{w}^{(2)}_i|(\tilde{P}^{x}_\rho)
\to 0\,.
\end{equation}
Indeed, noting that
\begin{align*}
E \tilde{w}^{(2)}_i=\L^n(\tilde{P}^{x}_{\rho}) \frac{\euno\odot \edue}{|E \tilde{w}|(\tilde{P}^{x}_{\rho}(y,\delta_i))} 
\Big(\delta_i^n \alpha_2'((y+\delta_i \cdot )\cdot \edue)\L^n\res\tilde{P}^{x}_\rho
+F^{y,\delta_i}_{\#}\big((\L^1\otimes D^s \alpha_2 \otimes \L^{n-2})\res\tilde{P}^{x}_\rho\big)\Big)\,.
\end{align*}
On setting $t_i:=\sfrac 1{|E\tilde{w}|(\tilde{P}^{x}_{\rho}(y,\delta_i))}$, and noting that $t_i \delta_i^n$ is infinitesimal 
in view of the choice of $y$ above, if $\varphi\in \mathcal C^0_c(\tilde{P}^{x}_{\rho};\R^n)$ and 
$ \alpha:= \alpha_2'(y\cdot \edue)$, we conclude that
	\begin{align*}
	t_i \delta_i^n&\left|\int_{\tilde{P}^{x}_{\rho}}
	\varphi(z):(\euno\odot \edue) \alpha'_2((y+\delta_i z)\cdot \edue)\d z\right|
	\leq t_i  \int_{\tilde{P}^{x}_{\rho}(y,\delta_i)} \left|\varphi\left(\frac{x-y}{\delta_i}\right)\right| 
	| \alpha'_2(x\cdot \edue)| \d x\\
	&  \leq t_i\|\varphi\|_{\infty}  \int_{\tilde{P}^{x}_{\rho}(y,\delta_i)} | \alpha'_2(x\cdot \edue)- \alpha| \d x
	+ t_i\delta_i^n|\alpha|\|\varphi\|_{\infty}  \mathcal{L}^n(\tilde{P}^{x}_{\rho}) \\&
	= \rho\|\varphi\|_{\infty}t_i \delta_i^{n} 
	\fint_{(y\cdot \edue)-\sfrac{\delta_i}2}^{(y\cdot \edue)+\sfrac{\delta_i}2}  | \alpha'_2(s)- \alpha| \d s
	+ o(1)\,, 
	\end{align*}
which vanishes as $i\uparrow+\infty$ (cf. item (b) above). Arguing similarly, we get that
\begin{align*}
t_i&\Big|\int_{\tilde{P}^{x}_{\rho}} \varphi(z)\,\d F_{\#}^{y,\delta_i}
\big((\L^1\otimes D^s \alpha_2 \otimes \L^{n-2})\res\tilde{P}^{x}_\rho\big)(z)\Big|\\
&\leq \rho\|\varphi\|_{\infty}\L^n(\tilde{P}^{x}_{\rho})t_i\delta_i^{n-1}
|D^s \alpha_2|(y\cdot \edue-\sfrac{\delta_i}2,y\cdot \edue+\sfrac{\delta_i}2)
\end{align*}
is infinitesimal for $i\uparrow+\infty$, as well (cf. item (c) above).

In conclusion, \eqref{e:claim w} implies that $\tilde{w}_i^{(2)}
$ converges to $0$ strongly in $BD(\tilde{P}^x_\rho)$ in view of Remark~\ref{r:scaling}, 
and then the limit of $\tilde{w}_{\tilde{P}^{x}_{\rho},y,\delta_i}$ is equal to that of 
$\tilde{w}^{(1)}_i$. 
Moreover, the latter coincides with the $L^1(\tilde{P}^{x}_{\rho};\R^n)$ limit of 
\begin{equation*}
\frac{\alpha_1((y+\delta_i z)\cdot\mathrm{e}_1)\mathrm{e}_2-
\mathfrak{R}_{\tilde{P}^{x}_\rho}[\alpha_1((y+\delta_i \cdot)\cdot\mathrm{e}_1)\mathrm{e}_2](z)}
{\delta_i\frac{|E(\alpha_1(\cdot\cdot\mathrm{e}_1)\edue)|(\tilde{P}^{x}_\rho(y,\delta_i))}{\L^n(\tilde{P}^{x}_\rho(y,\delta_i))}}\,.
\end{equation*}
as $|E\tilde{w}_{\tilde{P}^{x}_{\rho},y,\delta_i}|(\tilde{P}^{x}_{\rho})
=\mathcal L^n(\tilde{P}^{x}_{\rho})$ and $|E\tilde{w}^{(2)}_i|(\tilde{P}^{x}_\rho)$ 
is infinitesimal. Thus, by the blow-up theory for $BV$ functions (cf. \cite[Proposition~3.77, Theorem~3.95]{AFP00}), 
Lemma~\ref{l:auxiliary} and Remark~\ref{r:scaling} we conclude \eqref{e:grho structure 1} with
\[
\tilde{g}_\rho(z)=\psi_\rho(z\cdot \euno)\,\edue-\frac1{\sqrt{2}}\big(\edue\otimes\euno-\euno\otimes\edue\big)z\,.
\]
\smallskip

\textbf{Case $2$:} \textit{$D^s\alpha_1=D^s\alpha_2=0$}. 
In this case $\tilde{w}\in W^{1,1}(\tilde{P}^{x}_{\rho},\mathbb{R}^n)$, where $\tilde{P}_{\rho}^{x}:= \ZZ^{-t} P_{\rho}^{x}$
(the matrix $\ZZ$ has been introduced in item (II) above).
Arguing as in Case~1, we select a point $y\in Q(\underline{0},\rho)$ such that 
	\begin{itemize}
\item[(a')] $\frac{\d E \tilde{w}}{\d |E \tilde{w}|}(y)=\frac{\euno\odot \edue}{|\euno\odot \edue|}$;
\item[(b')] Proposition~\ref{p:recalings compactness} holds true: $\tilde{w}_{\tilde{P}_{\rho}^{x},y,\delta_i}$, for 
some sequence $\delta_i\downarrow 0$, converge to a map $\tilde{g}_\rho$ strict in $BD(\tilde{P}_{\rho}^{x})$ with 
$\mathfrak{R}_{\tilde{P}^{x}_{\rho}}[\tilde{g}_\rho]=\underline{0}$, and
$\{\frac1{\L^n(\tilde{P}_{\rho}^{x})}E\tilde{w}_{\tilde{P}_{\rho}^{x},y,\delta_i}\}_{i\in\N}$ converges weakly* in 
$\mathcal{M}(\bar{\tilde{P}}_{\rho}^{x};\R^{n\times n})$ to some $\tilde{\gamma}\in\Tan(E\tilde{w},y)$ such that 
$|\tilde\gamma|(\tilde{P}_{\rho}^{x})=|\tilde\gamma(\tilde{P}_{\rho}^{x})|=1$, 
$|\tilde{\gamma}|(\partial \tilde{P}_{\rho}^{x})=0$, $|\tilde{\gamma}|\in\Tan(|E\tilde{w}|,y)$, 
$\tilde\gamma=\frac{\d E \tilde{w}}{\d |E \tilde{w}|}(y)|\tilde\gamma|$ $|\tilde\gamma|$-a.e. on $\R^n$,
and $E\tilde{g}_\rho=\L^n(\tilde{P}_{\rho}^{x})\tilde{\gamma}\res \tilde{P}_{\rho}^{x}$;
\item[(c')] $\displaystyle \lim_{r\rightarrow 0} \fint_{(y\cdot \mathrm{e}_k)-r}^{(y\cdot \mathrm{e}_k)+r} 
	| \alpha_k'(s)- \alpha_k'(y\cdot \mathrm{e}_k)| \d s =0 $ for $k\in\{1,2\}$;
\item[(d')] $| \alpha_1'(y\cdot \euno)+ \alpha_2'(y\cdot \edue)|\neq 0$.
\end{itemize}	 
Note that conditions (a')-(c') hold for $\L^n$-a.e. $y\in Q(\underline{0},\rho)$, and (d') for a set of positive Lebesgue
measure in $Q(\underline{0},\rho)$ as $|E\tilde{w}|(Q(\underline{0},\rho))=\frac{|\euno\odot\edue|}{|\eta\odot\xi|}\L^n(Q(\underline{0},\rho))\neq 0$
thanks to \eqref{covmass2} and (II). 
As a consequence of all these conditions, 
$\delta^{-n}|E\tilde{w}|(\tilde{P}^x_\rho(y,\delta))\to|e(\tilde{w})(y)|\mathcal{L}^n(\tilde{P}_{\rho}^{x})
\neq 0$ as $\delta\downarrow 0$.
Thus, by blowing-up the function $\tilde{w}$ at one such point along the sequence of radii 
$\delta_i$ given by (b'), we may infer that, up to extracting a subsequence not relabeled,
$\tilde{w}_{\tilde{P}^{x}_{\rho},y,\delta_i}$ converge strictly in $BD(\tilde{P}^{x}_{\rho})$ to 
\[
\tilde{g}_\rho(z)=\frac{\alpha_1'(y\cdot\euno)+\alpha_2'(y\cdot\edue)}{2|e(\tilde{w})(y)|}
\big((z\cdot\euno)\,\edue+(z\cdot\edue)\,\euno\big)\,.
\]
To get the latter formula, we use that in this setting $\alpha_1$ and $\alpha_2$ are $W^{1,1}_{loc}(\R)$, and that condition (c') 
and Lemma~\ref{l:auxiliary} hold true.
\smallskip

\noindent\textbf{Conclusion in case $\eta(x)\neq\xi(x)$.} 
Both in case~$1$ and $2$ we have selected a point 
$y\in Q(\underline{0},\rho)$, a function $\tilde{g}_\rho\in BD(\tilde{P}_{\rho}^x)$ 
and a measure $\tilde{\gamma}\in\text{Tan}(E\tilde{w},y)$
satisfying: $\tilde{g}_\rho$ is affine in (at least) one direction between $\euno$ and $\edue$ (cf. \eqref{e:grho structure 1}) 
and 
$|\tilde\gamma|(\tilde{P}_{\rho}^{x})=|\tilde\gamma(\tilde{P}_{\rho}^{x})|=1$, 
$|\tilde{\gamma}|(\partial \tilde{P}_{\rho}^{x})=0$, $|\tilde{\gamma}|\in\Tan(|E\tilde{w}|,y)$, 
$\tilde\gamma=\frac{\euno\odot \edue}{|\euno\odot \edue|}|\tilde\gamma|$ $|\tilde\gamma|$-a.e. on $\R^n$,
and $E\tilde{g}_\rho=\L^n(\tilde{P}_{\rho}^{x})\tilde{\gamma}\res \tilde{P}_{\rho}^{x}$.

Lemma~\ref{lem:changeofvariable} implies that
$\text{Tan}(Ew,\ZZ^ty)=\ZZ^{-1}\big(\ZZ_{\#}^t\text{Tan}(E\tilde{w},y)\big)\ZZ^{-t}$. Therefore, by setting 
$g_{\rho}:=\ZZ^{-1} \tilde{g}_{\rho}(\ZZ^{-t}\cdot)$ and
$\gamma:=\ZZ^{-1}\big(\ZZ^{t}_{\#}\tilde\gamma\big)\ZZ^{-t}$ we deduce that 
$\gamma\in \text{Tan}(Ew,\ZZ^t y)\subseteq \Tan(E^c u,x)$ and that
$|\gamma|\in \text{Tan}(|Ew|,\ZZ^t y)\subseteq  \Tan(|E^c u|,x)$ in view of condition (III). Moreover, 
$Eg_\rho=\L^n(P_{\rho}^{x})\gamma\res P_{\rho}^{x}$.

Let $\{r_i\}_{i\in \N}$ be a sequence of radii and $\{c_i\}_{i\in \N}$ 
of positive constants  such that $c_i F^{x,r_i}_{\#} E u$ converge locally weakly* to $\gamma$ in 
$\mathcal{M}_{loc}(\R^n;\R^{n\times n})$, and $c_i F^{x,r_i}_{\#}|Eu|$ converge locally weakly* 
to $|\gamma|$ on $\mathcal{M}(\R^n)$. Then, $|\gamma|(\partial {P}_{\rho}^{x})=0$,
by taking into account that $|\tilde\gamma|(\partial \tilde{P}_{\rho}^{x})=0$. 
Thus, Lemma~\ref{lemma: weakconv} 
yields
	\begin{align*}
	c_i F^{x,r_i}_{\#} |Eu|(P_{\rho}^{x}) \rightarrow 
	|E g_\rho|(P_{\rho}^{x})\stackrel{\eqref{covmass2}}{=}|\det\ZZ|\frac{|\eta\odot\xi|}{|\euno\odot\edue|}
	\L^n(\tilde{P}_{\rho}^{x})=\frac{|\eta\odot\xi|}{|\euno\odot\edue|}\L^n(P_{\rho}^{x})\,.
	\end{align*}
Hence, $\L^n(P_{\rho}^{x}) \frac{F^{x,r_i}_{\#}E u}{F^{x,r_i}_{\#}|E u|(P_{\rho}^{x})} \wt  \frac{|\euno\odot\edue|}{|\eta\odot\xi|} Eg_\rho$
in $\mathcal{M}(P_{\rho}^{x};\R^{n\times n})$, and
$\L^n(P_{\rho}^{x}) \frac{F^{x,r_i}_{\#} |E u|}{F^{x,r_i}_{\#} |E u|(P_{\rho}^{x})} \wt \frac{|\euno\odot\edue|}{|\eta\odot\xi|} |Eg_\rho|$
in $\mathcal{M}(P_{\rho}^{x})$.
Furthermore, note that the rescaled maps $u_{P_\rho^{x},x,r_i}$ converge strongly in 
$L^1(P_\rho^{x};\R^n)$ to a map $v_{\rho}\in BD(P_{\rho}^{x})$ and $Eu_{P_\rho^{x},x,r_i}$ converge weakly* to 
$Ev_\rho$ in $\mathcal{M}(P_\rho^{x},\R^{n\times n})$, up to a subsequence not relabeled. 
Therefore, we find $Ev_{\rho}= \frac{|\euno\odot\edue|}{|\eta\odot\xi|} E g_\rho$ as measures on $P_\rho^{x}$ 
(cf. Remark~\ref{rmk:cov}). In view of 
Eq. \eqref{e:characterizationR} in Theorem~\ref{thm:poincare}, 
there is an infinitesimal rigid motion such that 
\[
v_{\rho}(z)=\frac{|\euno\odot\edue|}{|\eta\odot\xi|} g_{\rho}(z)+\LL_\rho z+\mathrm{v}_\rho\,.
\]
In addition, $|E v_{\rho}|(P_{\rho}^{x})=\L^n(P_{\rho}^{x})$, so that in particular $u_{P_\rho^{x},x,r_i}$ converge 
to $v_{\rho}$ strictly in $BD(P_\rho^{x})$.
We complete the proof of \eqref{eqn:relationasymptotics} by deducing that 
$\mathfrak{R}_{P_{\rho}^{x}}[v_{\rho}]=\underline{0}$ from $\mathfrak{R}_{P_{\rho}^{x}}[u_{P_\rho^{x},x,r_i}]=\underline{0}$
for all $i$ and the above mentioned strict convergence in $BD(P_\rho^{x})$.
\end{proof}
\begin{remark}
Notice that the parallelogram $P^x_{\rho}$ produced by Proposition~\ref{prop:vgblupRinDe}, along which the sequences $u_{P_{\rho}^x,x,\e}$ converges strictly in $BD(P_{\rho}^x)$ to a function $v_{\rho}$ affine in one direction, has been chosen in a way that the short edge (the edge of size $\rho$) is oriented exactly in the direction of non-affinity of $v_{\rho}$. The non-affine direction is exactly the one we need to control in order to conclude the representation Theorem~\ref{thm:mainthm} and thus, the fact that the parallelograms $P^x_{\rho}$ are well oriented plays a crucial role in the argument that will follow. 
\end{remark}

\subsection{Finer analysis of the blow-up limits}\label{s:finer} 
We proceed next with the investigation 
of some properties of the blow-up limits provided by Proposition~\ref{prop:vgblupRinDe} 
that follow by exploiting their structure evidenced in Eqs. \eqref{eqncase1}-\eqref{eqn:relationasymptotics}.
Similar results are available in the $BD$ setting in case the base point is either a point of approximate differentiability 
or a jump point. The analogue of the ensuing result is also well-known for $BV$ functions (see for instance \cite[Theorem~3.95]{AFP00}).

We will state some technical lemmas that will allow us to identify in a more precise way the blow-up limits.
To this aim, for a function $\psi\in BV((a,b))$ we denote by $\psi(a)$, $\psi(b)$ the right and left traces in $a$, $b\in\R$, respectively.
We start with the case $\xi\neq\pm\eta$.
\begin{lemma}\label{lem:blwpCASEONE}
Let $\{v_{\rho}\}_{\rho>0}\subseteq BV(P ^{\xi,\eta}_{\rho};\R^n)$, $\xi,\,\eta\in\R^n\setminus\{\underline{0}\}$ 
with $\xi\neq\pm \eta$, be a sequence of functions such that
	\[
	v_{\rho}(x)=\bar{\psi}_{\rho}(x\cdot \eta) \xi+ 
	(x\cdot \xi)\bar{\beta}_{\rho} \,\eta
	+\LL_{\rho} x+\mathrm{v}_{\rho} 
	\]
for some $\bar{\psi}_{\rho}\in BV\big((-\sfrac \rho2,\sfrac \rho2)\big), \bar{\beta}_{\rho} \in \R$,
$\mathrm{v}_{\rho}\in \R^n, \LL_{\rho}\in \mathbb{M}_{skew}^{n\times n}$. 
Assume also that 
	\[
	\mathfrak{R}_{P^{\xi,\eta}_{\rho}}[v_{\rho}]=\underline{0},\qquad |E v_{\rho}|(P^{\xi,\eta}_{\rho})=\L^n(P^{\xi,\eta}_{\rho})\,,
	\]
Then, $v_{\rho}$ can be re-written as
	\begin{align*}
	v_{\rho}(x)= \psi_{\rho}(x\cdot \eta) \xi+\frac{
	x\cdot \xi}{2|\eta\odot \xi|}\eta,
	\end{align*}
for some $\psi_{\rho}\in BV\big((-\sfrac \rho2,\sfrac \rho2)\big)$ with zero average such that 
$\{\sfrac{\psi_{\rho}(\cdot)}\rho\}_{\rho>0}$ is uniformly bounded in 
$L^\infty\big((-\sfrac\rho2,\sfrac\rho2)\big)$ and
\begin{align}\label{eqn:blw2CASEONE}
 D \psi_{\rho}\big((-\sfrac \rho2,\sfrac \rho2)\big)
  =\frac\rho{2|\eta\odot \xi|}\,.
\end{align}
\end{lemma}
	\begin{proof}
Set $R_{\rho}:=\big(-\sfrac\rho2,\sfrac\rho2\big)\times \big(-\sfrac12,\sfrac12\big)^{n-1}$,
and let $\mathbb{B}\in \mathbb{M}^{n\times n}$ be any invertible matrix such that $\mathbb{B}\eta=\euno, \mathbb{B}\xi=\edue$ and mapping $R_{\rho}$ onto $P_{\rho}^{\xi,\eta}$. 
Namely, $P_{\rho}^{\xi,\eta}=\mathbb{B}^t R_{\rho}$. By invoking Lemma~\ref{lem:changeofvariable} and Remark~\ref{rmk:cov} we can infer that
    \[
    \tilde{v}_{\rho}(y):=\mathbb{B} v_{\rho}(\mathbb{B}^t y)
    \]
satisfies $\tilde{v}_{\rho}\in BD(R_{\rho})$ and
    \begin{equation*}
    \tilde{v}_{\rho}(y)=\bar{\psi}_{\rho}(y\cdot \euno) \edue+
    (y\cdot \edue)\bar{\beta}_{\rho}\,\euno +
	\tilde{\mathbb{L}}_{\rho} y+\tilde{\mathrm{v}}_{\rho},
	\end{equation*}
with $\tilde{\mathbb{L}}_{\rho}=\mathbb{B}\mathbb{L}_{\rho}\mathbb{B}^t$, $\tilde{\mathrm{v}}_{\rho}=\mathbb{B}\mathrm{v}_{\rho}$. 
Moreover $\mathfrak{R}_{R_{\rho}}[\tilde{v}_{\rho}]=\underline{0}$, and
	\begin{equation}\label{totalmass}
|E \tilde{v}_{\rho}|(R_{\rho})=\frac{|\euno\odot \edue|}{|\eta\odot \xi|} 
\rho,\qquad  E \tilde{v}_\rho = \frac{\euno\odot \edue}{|\euno\odot \edue|} |E \tilde{v}_\rho|.
\end{equation}	
\smallskip

\noindent\textbf{Step $1$:} \textit{Identification and
properties of $ \psi_{\rho}$ and $ \beta_{\rho}$.}
Condition $\mathfrak{R}_{R_{\rho}}[\tilde{v}_{\rho}]=\underline{0}$ is equivalent to
$\mathbb{M}_{R_{\rho}}[\tilde{v}_{\rho}]=b_{R_{\rho}}[\tilde{v}_{\rho}]=\underline{0}$
(see \eqref{e:bK} for the definition of $b_{R_{\rho}}$ and \eqref{e:MK} for that of $\mathbb{M}_{R_{\rho}}$).
In turn, from these equalities we get that
	\begin{align*}\label{eqn:shape1}
	\tilde{\mathrm{v}}_{\rho}+\Big(\fint_{R_{\rho}} \bar{\psi}_{\rho} (y\cdot \euno)\d y\Big)\edue
	=\tilde{\mathrm{v}}_{\rho}+\Big(\fint_{-\sfrac \rho2}^{\sfrac \rho2} \bar{\psi}_{\rho}(t)  \d t\Big)\edue=\underline{0}\,.
  \end{align*}
Thanks to Lemma~\ref{l:auxiliary} we get
\begin{align*}
\mathbb{M}_{R_{\rho}} \left[\Big( \bar{\psi}_{\rho}(y\cdot \euno)\,-\fint_{-\sfrac \rho2}^{\sfrac \rho2} \bar{\psi}_{\rho}(t)  \d t\Big)\edue\right]
=\frac{1}{2\rho} \big(\bar{\psi}_{\rho}( \sfrac \rho2) -\bar{\psi}_{\rho}(-\sfrac \rho2)\big)(\edue\otimes \euno-\euno\otimes \edue)\,,
\end{align*}
and 
\begin{align*}
\mathbb{M}_{R_{\rho}} \left[ (y\cdot \edue)\bar{\beta}_{\rho} \euno\right]=& \frac{\bar{\beta}_{\rho}}{2}  (\euno\otimes \edue-\edue\otimes \euno)\,.
\end{align*}
Therefore, recalling that $\mathbb{M}_{R_{\rho}}\big[\tilde{\LL}_{\rho}y\big]=\tilde{\LL}_{\rho}y$, we conclude that for every $y\in R_{\rho}$
\begin{align*}
\tilde{\LL}_{\rho}y+\mathbb{M}_{R_{\rho}}[\tilde{v}_{\rho}]y=
\tilde{\LL}_{\rho}y+\kappa_{\rho}(\edue\otimes \euno-\euno\otimes \edue)y=\underline{0}\,,
\end{align*}
where we have set 	
\begin{equation}\label{e:krho}
\kappa_{\rho}:= \frac{1}{2\rho}\big(\bar{\psi}_{\rho}(\sfrac \rho2)-\bar{\psi}_{\rho}(-\sfrac \rho2)-\bar{\beta}_{\rho}\rho \big).
\end{equation}
In particular, by defining
	\[
	\psi_{\rho}(t):=\bar{\psi}_{\rho}(t)-\kappa_{\rho} t - \fint_{-\sfrac\rho2}^{\sfrac\rho2} \bar{\psi}_{\rho}(t) \d t, 
	\qquad  \beta_{\rho}:=\bar{\beta}_{\rho}+\kappa_{\rho},
	\]
	$\psi_{\rho}\in BV((-\sfrac\rho2,\sfrac\rho2))$ has zero average, and 
	\begin{equation}\label{shapeinside}
	\tilde{v}_{\rho}(y)= \psi_{\rho}(x\cdot \euno) \edue + (x\cdot \edue) \beta_{\rho}\euno.
	\end{equation}
Moreover, from the very definitions of $\tilde{v}_\rho$, $\psi_\rho$ and $\beta_\rho$ we see that
	\[
	E \tilde{v}_{\rho}=(\euno\odot \edue)\big(D \bar{\psi}_{\rho} \otimes \L^{n-1} +\bar{\beta}_{\rho}\L^n\big)
	=(\euno\odot \edue)\big(D \psi_{\rho} \otimes \L^{n-1} +\beta_{\rho}\L^n\big).
	\]
Let $I\subseteq (-\sfrac \rho2,\sfrac \rho2)$ be $\L^1$-measurable, then
\begin{align}\label{number 1}
	\frac{\euno\odot \edue}{|\euno\odot \edue|}&|E \tilde{v}_{\rho}| (I\times  (-\sfrac12,\sfrac12)^{n-1})
	=E \tilde{v}_{\rho} (I\times  (-\sfrac12,\sfrac12)^{n-1})\nonumber\\&
	=(\euno\odot \edue)\big( D \bar{\psi}_{\rho} (I)+   \bar{\beta}_{\rho}\L^1(I) \big)  
	=(\euno\odot \edue)\big( D \psi_{\rho} (I)+   \beta_{\rho}\L^1(I) \big).  
\end{align}
In particular, if $I=(-\sfrac\rho2,\sfrac\rho 2)$ by exploiting \eqref{totalmass} we conclude 
	\begin{equation}\label{eqn:totmass}
	D\bar{\psi}_{\rho}\big((-\sfrac \rho2,\sfrac \rho2)\big)+\bar{\beta}_{\rho}\rho
	=D\psi_{\rho}\big((-\sfrac \rho2,\sfrac \rho2)\big)+\beta_{\rho}\rho
	=\frac{\rho}{|\eta\odot\xi|}. 
\end{equation}
\smallskip

\noindent\textbf{Step $2$:} \textit{The value of $\beta_{\rho}$}.
Let $\I:=\{I\subseteq(-\sfrac \rho2,\sfrac \rho2)\quad\L^1\, \text{-measurable}:\L^1(I)>0\}$, then from \eqref{number 1} we deduce that 
	\[
	\inf_{\I}\frac{D\bar{\psi}_{\rho}(I)}{\L^1(I)}+\bar{\beta}_{\rho} \geq0, \qquad 
\inf_{\I}\frac{D \psi_{\rho}(I)}{\L^1(I)}+\beta_{\rho} \geq0\,\,.
	\]
Hence, $t\mapsto\bar{\psi}_{\rho}(t)+\bar{\beta}_{\rho}t$ and $t\mapsto \psi_{\rho}(t)+\beta_{\rho}t$ 
are monotone non-decreasing functions. 

Thus, thanks to the definition of $\kappa_\rho$ we infer that $D \psi_{\rho} (-\sfrac \rho2,\sfrac \rho2)=\beta_\rho\rho$ as
\begin{align*}
D \psi_{\rho}&\big((-\sfrac \rho2,\sfrac \rho2)\big)=D \bar{\psi}_{\rho}\big((-\sfrac \rho2,\sfrac \rho2)\big)-\kappa_{\rho}\rho\\&
\stackrel{\eqref{e:krho}}{=}\frac{1}{2}\big(D \bar{\psi}_{\rho}\big((-\sfrac \rho2,\sfrac \rho2)\big)+\bar{\beta}_{\rho}\rho\big)
\stackrel{\eqref{eqn:totmass}}{=}\frac{1}{2}\big(D \psi_{\rho}\big((-\sfrac \rho2,\sfrac \rho2)\big)+\beta_{\rho}\rho\big).
\end{align*}
In conclusion, we get 
	\begin{equation}\label{massinside}
	D \psi_{\rho} \big((-\sfrac \rho2,\sfrac \rho2)\big)
	=
	\beta_{\rho}\rho=\frac{\rho}{2|\eta\odot\xi|}.
	\end{equation}
\smallskip

\noindent\textbf{Step $3$:} \textit{Inverse change of variables and conclusion}. 
By combining \eqref{shapeinside} and \eqref{massinside} we are thus led to
    \[
    \tilde{v}_{\rho}(y)=\psi_{\rho}(x\cdot \euno)\edue + \frac{(x\cdot \edue) }{2|\eta\odot\xi|} \euno\,.
    \]
Due to the very definition, $v_{\rho}(y):=\mathbb{B}^{-1} \tilde{v}_{\rho} (\mathbb{B}^{-t} y)$, thus we get that
    \[
    v_{\rho}(y)=\psi_{\rho}(x\cdot \eta)\xi + \frac{x\cdot \xi}{2|\eta\odot \xi|} \eta\,.
    \]
Finally, by taking into account that  $t\mapsto h_\rho(t):=\psi_{\rho}(t)+\beta_{\rho}t$ 
is monotone non-decreasing with zero average we get (by \cite[Remark~3.50]{AFP00} and a simple scaling argument)
\[
\|h_\rho\|_{L^\infty((-\sfrac \rho2,\sfrac \rho2))}\leq |Dh_\rho|\big((-\sfrac \rho2,\sfrac \rho2)\big)=\frac{\rho}{|\eta\odot \xi|}.
\]
The statement for $\psi_\rho$ then follows at once.
	\end{proof}
Similarly, we can characterize the case $\xi=\pm \eta$.
\begin{lemma}\label{lem:blwpCASETHREE}
Let $\{v_{\rho}\}_{\rho>0}\subseteq BV(P^{\xi,\eta}_{\rho};\R^n)$, $\xi,\,\eta\in\R^n\setminus\{\underline{0}\}$ 
with $\xi=\pm \eta$, be a family of functions such that
	\[
	v_{\rho}(x)=\bar{\psi}_{\rho}(x\cdot \eta) \eta+\LL_{\rho} x+\mathrm{v}_{\rho} 
	\]
for some $\mathrm{v}_{\rho}\in \R^n$, $\LL_{\rho}\in \mathbb{M}_{skew}^{n\times n}$, 
$\bar{\psi}_{\rho}\in BV\big((-\sfrac \rho2,\sfrac \rho2)\big)$. 
Assume also that, 
	\[
	\mathfrak{R}_{P^{\xi,\eta}_{\rho}}[v_{\rho}]=\underline{0},\qquad|E v_{\rho}|(P^{\xi,\eta}_{\rho})=
	\L^n(P^{\xi,\eta}_{\rho})\,.
	\]
Then, $v_{\rho}$ can be re-written as
	\begin{align*}
	v_{\rho}(x)= & \psi_{\rho}(x\cdot \eta) \eta
	\end{align*}
for a non-decreasing function $\psi_{\rho}\in BV\big((-\sfrac \rho2,\sfrac \rho2)\big)$ with zero average. 
Moreover, $\{\sfrac{\psi_{\rho}(\cdot)}\rho\}_{\rho>0}$ is uniformly bounded in 
$L^\infty\big((-\sfrac \rho2,\sfrac \rho2)\big)$, and
\begin{align}\label{eqn:blw2CASETHREE}
 |D \psi_{ \rho}|\big((-\sfrac \rho2,\sfrac \rho2)\big)
 =\frac{\rho}{|\eta\odot \eta|}\,.
\end{align}
\end{lemma}

Let us summarize the results contained in Proposition~\ref{prop:vgblupRinDe}, Lemma~\ref{lem:blwpCASEONE} and
Lemma~\ref{lem:blwpCASETHREE} in the following statement (see \eqref{e:rescaled} for the definition of $u_{K,x,\e}$, and
Eqs. \eqref{eqn:convexdefinition}, \eqref{eqn:convexdefinition1} for those of $P_{\rho}^{\xi,\eta}$).
\begin{proposition}[Selecting a good blow-up $|E^cu|$-a.e.]\label{prop:vgblup}
Let $u\in BD(\Omega)$. Then for $|E^cu|$-a.e. $x\in\Omega$ 
there exist vectors $\xi,\,\eta\in\mathbb{S}^{n-1}$, $\{\rho_j\}_{j\in \N}$, $\rho_j\downarrow 0$ as $j\uparrow+\infty$, 
and for all $j\in\N$ a sequence $\{\e_{i,j}\}_{i\in \N}$, with $\e_{i,j}\downarrow 0$ as $i\uparrow+\infty$, 
a bounded, open, convex set containing the origin $P^x_j:=P^{\xi,\eta}_{\rho_j}$ and a function $v_j$ such that 
$u_{P^x_j,x,\e_{i,j}}$ converge to $v_j$ strictly in $BD(P^x_j)$ as $i\uparrow+\infty$, with
	\begin{itemize}
	\item[(a)] if $\xi\neq \pm\eta$: 
\[
v_{j}(y):=\psi_{j}(y\cdot \eta)\xi+\frac{y\cdot\xi}{2|\eta\odot \xi|}\eta
\]
for some map $\psi_{j}\in BV\big((-\sfrac{\rho_j}{2},\sfrac{\rho_j}{2})\big)$ with zero average such that
	\begin{align*}
D\psi_{j} \big((-\sfrac{\rho_j}{2},\sfrac{\rho_j}{2})\big)
=\frac{\rho_j}{2|\eta\odot \xi|},\qquad
\sup_{j\in \N} \left\|\sfrac{\psi_{j} (\cdot)}{\rho_j} \right\|_{L^\infty((-\sfrac{\rho_j}2,\sfrac{\rho_j}2))}<+\infty;
\end{align*}
\item[(b)] if $\xi=\pm \eta$: 
\[
v_{j}(y)=\psi_{j}(y\cdot \eta)\eta
\]
for some non-decreasing map $\psi_{j}\in BV\big((-\sfrac{\rho_j}{2},\sfrac{\rho_j}{2})\big)$ with zero average such that
\begin{align*}
|D\psi_{j}|\big((-\sfrac{\rho_j}{2},\sfrac{\rho_j}{2})\big)
=\frac{\rho_j}{|\eta\odot\eta|},\qquad
\sup_{j\in \N}\left\|\sfrac{\psi_{j} (\cdot)}{\rho_j} \right\|_{L^\infty((-\sfrac{\rho_j}2,\sfrac{\rho_j }2))}<+\infty.
\end{align*}
\end{itemize}
 \end{proposition}
\begin{proof}
We prove the result for the subset of points $F$ for which Proposition~\ref{prop:vgblupRinDe} holds true.
In particular, $|E^c u|(\Omega\setminus F)=0$. One such point $x$ being fixed, note that given any infinitesimal 
sequence $\{\rho_j\}_{j\in\N}$ we can extract a subsequence (not relabeled) along which the maps $v_{\rho_j}$ 
provided by Proposition~\ref{prop:vgblupRinDe} are affine in one single direction (either $\eta$ 
or $\xi$) provided in the statement.
Without loss of generality we denote such a direction by $\xi$ to be coherent with the notation of 
Proposition~\ref{prop:vgblupRinDe} itself.

Assume first $\xi\neq \pm \eta$. Thanks to Proposition~\ref{prop:vgblupRinDe} we can find a sequence of scales 
$\{\e_{i,j}\}_{i\in \N}$ such that the rescaled maps converge strictly in $BD(P_{j}^x)$ to a map $v_{\rho_j}$  as
in the statement there. By using Lemma~\ref{lem:blwpCASEONE} we conclude.

Finally, if $\xi=\pm\eta$ we argue similarly by using Lemma~\ref{lem:blwpCASETHREE} rather than Lemma~\ref{lem:blwpCASEONE}.
\end{proof}

\section{Proof of the main result}\label{s:proofmainresult}
We first recall the results in \cite[Theorem~3.3, Remark~3.5]{ebobisse2001note}. 
The original statement concerns integral representation of the volume and jump energy 
densities of functionals satisfying (H1)-(H4) and a more stringent version of (H5) (cf. Remark~\ref{r:H5bis}) 
and for functions in the subspace $SBD(\Omega)$. In what follows we state the result for the full space 
$BD(\Omega)$. Indeed, the same proof works with no difference since it is obtained via 
the global method for relaxation, hinging on a blow-up argument and the 
characterization of the energy densities in terms of the Dirichlet cell formulas defining 
$\mm$. We notice that (H4) and (H5) are actually not needed for the integral representation of the bulk and surface 
terms of the energy.
\begin{lemma}\label{lem:goodpointsdensity}
   Let $\F$ be satisfying (H1)-(H3). Then, for every $u\in BD(\Omega)$ 
    \begin{itemize}
        \item[(a)] for $\L^n$-a.e. $x_0\in \Omega$ 
            \[
            \lim_{\e\rightarrow 0} \frac{\F(u,Q(x_0,\e))}{\e^n}= f(x_0,u(x_0),\nabla u(x_0)),
            \]
            where $f$ denotes the function defined in \eqref{e:f}; 
        \item[(b)] for $\H^{n-1}$-a.e. $x_0\in  J_u$ 
             \[
            \lim_{\e\rightarrow 0} \frac{\F(u,Q^{\nu_u(x_0)}(x_0,\e))}{\e^{n-1}}= 
            g(x_0,u^-(x_0),u^+(x_0),\nu_u(x_0))
            \]
            where $g$ denotes the function defined in \eqref{e:g}, 
    \end{itemize}
\end{lemma}

It should be noted that Lemma~\ref{lem:goodpointsdensity} and the lower semicontinuity of the integral on $W^{1,1}$ 
implies that the Borel functions $f$ and $g$ are respectively quasiconvex
(see \cite[Definition~5.25]{AFP00}, see also formulas \eqref{e:sqc0}, \eqref{e:sqcPER} below) and $BV$ elliptic (see \cite[Definition~5.13]{AFP00}).

By taking into account (H2), we conclude that there exists a constant $C>0$ such that 
for every $(x,\mathrm{v},\mat)\in\Omega\times\R^n\times\mathbb{M}^{n\times n}$
\begin{equation}\label{e:flingr}
\frac1C\Big|\textstyle{\frac{\mat+\mat^t}{2}}\Big|\leq f(x,\mathrm{v},\mat)
 \leq C\Big(1+\Big|\textstyle{\frac{\mat+\mat^t}{2}}\Big|\Big),
\end{equation}
and that for every $(x,\mathrm{v}^-,\mathrm{v}^+)\in\Omega\times(\R^n)^2$
\[
 \frac1C|(\mathrm{v}^+-\mathrm{v}^-)\odot\nu|\leq g(x,\mathrm{v}^-,\mathrm{v}^+,\nu)\leq 
 C|(\mathrm{v}^+-\mathrm{v}^-)\odot\nu|.
\]
Several other properties of $f$ and $g$ can be inferred according to the invariance 
properties satisfied by the functional $\F$ (cf. \cite[Remark 3.8]{bouchitte1998global}). 
For instance, assumption (H5) implies that $f$ depends only on the symmetric part of 
the relevant matrix.
Indeed, from \eqref{eqn:H5onm} we immediately deduce, for all $(x_0,\mathrm{v}_0,\mat)\in\Omega\times\R^n\times \mathbb{M}^{n\times n}
$, 
(thanks also to item (a) in Lemma \ref{lem:goodpointsdensity}) that
    \begin{align}\label{eqn:invariance}
   f(x_0,\mathrm{v}_0,\mat)&=\lim_{\e\rightarrow 0}\frac{\mm(\mathrm{v}_0+\mat(\cdot -x_0),Q(x_0,\e))}{\e^n}\nonumber\\
    &=\lim_{\e\rightarrow 0}\frac{\mm(\mathrm{v}_0+\textstyle{\frac{\mat+\mat^t}2}(\cdot -x_0),Q(x_0,\e))}{\e^n}
    =f\big(x_0,\mathrm{v}_0,\textstyle{\frac{\mat+\mat^t}{2}}\big).
    \end{align}
Therefore, in this case we deduce that $f$ is symmetric quasiconvex. Namely, for every bounded open set $D\subset\R^n$, for $\L^n$ a.e. $x\in\Omega$,
for all $(\mathrm{v},\mat)\in\R^n\times\mathbb{M}^{n\times n}_{sym}$ and for all $\varphi\in \mathcal C^1_c(D;\R^n)$
	\begin{equation}\label{e:sqc0}
	f(x,\mathrm{v},\mat)\leq \fint_{D} f(x,\mathrm{v},\mat+e(\varphi)(y))\d y,
	\end{equation}
	or, equivalently, for all $\varphi\in \mathcal C^1(Q^\nu(\underline{0},1);\R^n)$ that are $Q^\nu(\underline{0},1)$-periodic, 
it holds
	\begin{equation}\label{e:sqcPER}
	f(x,\mathrm{v},\mat)\leq \fint_{Q^\nu(\underline{0},1)} f(x,\mathrm{v},\mat+e(\varphi)(y))\d y.
	\end{equation}

\begin{remark}\label{r:H5bis}
If, in addition, we strengthen (H5) to
 \begin{equation}\label{verystrongh5}
 \mathcal{F}(u+\LL(\cdot-x_0)+\mathrm{v},A)=\mathcal{F}(u,A)
 \end{equation}
for every $(u,A,\mathrm{v},\LL,x_0) \in BD(\Omega)\times \mathcal{O}(\Omega)\times\R^n\times\mathbb{M}_{skew}^{n\times n}\times \Omega$, 
 then the cell formulas imply $f\big(x,\mathrm{v},\frac{\mat+\mat^t}2\big)=f\big(x,\frac{\mat+\mat^t}2\big)$ and 
 $g(x,\mathrm{v}^-,\mathrm{v}^+,\nu)=g(x,\mathrm{v}^+-\mathrm{v}^-,\nu)$. 
\end{remark}

\subsection{Preliminary lemmas}

We now exploit the result in Section~\ref{s:blowupCantor2}, in particular 
Proposition~\ref{prop:vgblup}, 
to deduce the asymptotic behavior of the energy around $|E^cu|$-a.e. point along
the same line developed in \cite[Lemma 3.9]{bouchitte1998global}. 
We keep the notation introduced in Proposition~\ref{prop:vgblup} and to simplify it 
we set $\psi_j:=\psi_{\rho_j}$ and $P_j^{x_0}:=P_{\rho_j}^{x_0}$ (for the definition of
$\mathfrak{R}_K$ see \eqref{e:RK}).
\begin{lemma}\label{lem:shapeofF}
Let $\F$ satisfy (H1)-(H4). Then, for every $u\in BD(\Omega)$ and 
for $|E^c u|$-a.e $x_0\in \Omega$ there exist a sequence $\{\rho_j\}_{j\in \N}$ infinitesimal as $j\uparrow+\infty$, 
and for all $j\in \N$ an infinitesimal sequence $\{\e_{i,j}\}_{i\in \N}$ as $i\uparrow+\infty$, such that 
    \begin{align}\label{eqn:keyequality}
        \frac{\d \F(u,\cdot)}{\d |E u|}(x_0)
        =\lim_{j \rightarrow +\infty}\limsup_{i\rightarrow +\infty}
        \frac{\mm\left(w_{i,j}, P_j^{x_0}(x_0,\e_{i,j})\right) }{|E u|(P_j^{x_0}(x_0,\e_{i,j}))}
    \end{align}
    where
    \begin{equation}\label{e:wij}
    w_{i,j}(y):=\mathfrak{R}_{P_j^{x_0}(x_0,\e_{i,j})}[u] (y-x_0)
    +\textstyle{\frac{|E u|(P_j^{x_0}(x_0,\e_{i,j}))}{\L^n(P_j^{x_0}(x_0,\e_{i,j})))}\frac{ \eta\odot \xi}{|\eta\odot \xi|}}(y-x_0).
\end{equation}
\end{lemma}
\begin{proof} We consider the subset of points of $C_u$ for which Proposition~\ref{prop:vgblupRinDe} (and hence Proposition~\ref{prop:vgblup}) is valid. 
For one such point $x_0\in C_u$ consider the corresponding 
vectors $\xi$ and $\eta\in\mathbb{S}^{n-1}$. 
Note that by Lemma~\ref{lem:derivative of m} for any $j\in\N$
    \begin{equation}\label{eqn:convtoF}
    \frac{\d \F(u,\cdot)}{\d |E u|}(x_0)=\lim_{i\rightarrow +\infty} \frac{\mm(u, P_j^{x_0}(x_0,\e_{i,j}))}{|E u|(P_j^{x_0}(x_0,\e_{i,j}))}.
   \end{equation}
\smallskip

\noindent\textbf{Case~1:} \textit{$\eta\neq \pm \xi$.} 
By Proposition~\ref{prop:vgblup} we have that for every $j\in\N$
	\[
	u_{i,j}:=u_{P_j^{x_0},x_0,\e_{i,j}}\rightarrow v_j:=\psi_j(y\cdot \eta) \xi +\frac{y\cdot \xi}{2|\eta \odot \xi|}\eta\qquad\text{ strict in $BD(P_j^{x_0})$}.
	\]
Define, for some constant $c_j$ to be specified in what follows, the functions
    \begin{align}\label{e:vij}
        v_{i,j}(y):=w_{i,j}(y)+\e_{i,j}\textstyle{\frac{|E u|(P_j^{x_0}(x_0,\e_{i,j}))}{\L^n(P_j^{x_0}(x_0,\e_{i,j}))}} c_j \xi.
    \end{align}
As by Remark~\ref{r:scaling}
\[
 \mathfrak{R}_{P^{x_0}_{j}}[u(x_0+\e_{i,j}\cdot )]\left(\frac{y-x_0}{\e_{i,j}}\right)
 =\mathfrak{R}_{P^{x_0}_{j}(x_0,\e_{i,j})}[u](y-x_0),
 \]
we have by Lemma~\ref{lem:continuity if m} 
    \begin{align*}
        &\frac{|\mm(u,P_j^{x_0}(x_0,\e_{i,j}))-\mm(v_{i,j},P_j^{x_0}(x_0,\e_{i,j}))|}{|E u |(P_j^{x_0}(x_0,\e_{i,j}))}\\
        &\leq C\int_{\partial P_j^{x_0}(x_0,\e_{i,j})} \frac{|u(y)-v_{i,j}(y)|}{|E u |(P_j^{x_0}(x_0,\e_{i,j}))}\d \H^{n-1}(y)\\
         &= \frac{C}{\rho_j}\int_{\partial P_j^{x_0}} \left|u_{i,j}(y) - \frac{ \eta \odot \xi\ }{|\eta \odot \xi|}\,y -c_{j}\xi\right|\d \H^{n-1}(y).
    \end{align*}
By taking the superior limit in $i$ we thus get
\begin{align*}
        \limsup_{i\rightarrow+\infty}& \frac{|\mm(u,P_j^{x_0}(x_0,\e_{i,j}))-\mm(v_{i,j},P_j^{x_0}(x_0,\e_{i,j}))|}{|E u |(P_j^{x_0}(x_0,\e_{i,j}))}\\
        & \leq \frac{C }{\rho_j}\int_{\partial P_j^{x_0}} \left| \psi_j(y\cdot \eta)- \frac{(y\cdot \eta)}{2|\eta\odot \xi|}-c_{j}  \right|\d \H^{n-1}(y)\\
        &=\frac{C }{\rho_j}\left|\psi_j(\sfrac{\rho_j}{2})- \frac{\rho_j}{4|\eta\odot \xi|}-c_j  \right|
        +\frac{C }{\rho_j}\left|\psi_j(-\sfrac{\rho_j}{2})+\frac{\rho_j}{4|\eta\odot \xi|}-c_j  \right|
        \\
        &+\frac{C }{\rho_j}\int_{-\sfrac{\rho_j}{2}}^{\sfrac{\rho_j}{2}}\left| \psi_j(t)- \frac{t}{2|\eta\odot \xi|} -c_{j} \right|\d t,
    \end{align*}
recalling that
$P_j^{x_0}=\big\{y\in \R^n:\,|y\cdot \eta|\leq \sfrac{\rho_j}2,\, |y\cdot \xi|\leq \sfrac12,\,|y\cdot \zeta_i|\leq \sfrac12\,\, i=1,\ldots,n-2\big\}$,
and that $\xi$ and $\eta$ depend on $x_0$.

By choosing $c_j:=\psi_j(-\sfrac{\rho_j}{2})+\frac{\rho_j}{4|\eta \odot \xi|}$, since $\psi_j(\sfrac{\rho_j}{2})-\psi_j(\sfrac{-\rho_j}{2})
=\frac{\rho_j}{2|\eta \odot \xi|}$, the first two summands in the last inequality are then null. Therefore, we have
\begin{align*}
\limsup_{i\rightarrow+\infty}&\frac{|\mm(u,P_j^{x_0}(x_0,\e_{i,j}))-\mm(v_{i,j},P_j^{x_0}(x_0,\e_{i,j}))|}{|E u |(P_j^{x_0}(x_0,\e_{i,j}))}
\\&
\leq\frac{C }{\rho_j}\int_{-\sfrac{\rho_j}{2}}^{\sfrac{\rho_j}{2}}\left| \psi_j(t)- \frac{t}{2|\eta\odot \xi|} -c_{j} \right|\d t
\leq C\Big(\|\psi_j\|_{L^{\infty}(-\frac{\rho_j}2,\frac{\rho_j}2)}+\rho_j+c_j\Big)\leq C\rho_j\,,
\end{align*}
where we have used that $\sfrac{\psi_j(\cdot)}{\rho_j}$ is equi-bounded in $L^{\infty}\big((-\sfrac{\rho_j}2,\sfrac{\rho_j}2)\big)$, 
and that $\psi_j(\pm\sfrac{\rho_j}2)$ are boundary trace values to infer $c_j\leq C\rho_j$, for some universal constant $C>0$. 
In conclusion, we have proved that 
    \begin{align*}
        \lim_{j\rightarrow+\infty}\limsup_{i\rightarrow+\infty}
        \frac{|\mm(u,P_j^{x_0}(x_0,\e_{i,j}))-\mm(v_{i,j},P_j^{x_0}(x_0,\e_{i,j}))|}{|E u |(P_j^{x_0}(x_0,\e_{i,j}))}=0,
    \end{align*}
so that \eqref{eqn:convtoF} yields    
    \begin{equation}\label{e:quasi}
     \frac{\d \F(u,\cdot)}{\d |E u|}(x_0)=\lim_{j\rightarrow+\infty}\limsup_{i\rightarrow+\infty}
     \frac{\mm(v_{i,j},P_j^{x_0}(x_0,\e_{i,j}))}{|E u |(P_j^{x_0}(x_0,\e_{i,j}))}.
    \end{equation}
Finally, recalling the definition of $w_{i,j}$ in \eqref{e:wij} and of $v_{i,j}$ in \eqref{e:vij}, by estimate \eqref{e:stimammcont} we have
    \begin{align*}
  \Big| \mm\Big(v_{i,j}&
  ,P_j^{x_0}(x_0,\e_{i,j})\Big)- \mm(w_{i,j},P_j^{x_0}(x_0,\e_{i,j}))\Big|\\
   & \leq C_{P_j^{x_0}}\Psi\big(\e_{i,j}\textstyle{\frac{|E u|(P_j^{x_0}(x_0,\e_{i,j}))}{\L^n(P_j^{x_0}(x_0,\e_{i,j}))}}c_j\big)
   \Big(\L^n(P_j^{x_0}(x_0,\e_{i,j}))+|Eu|(P_j^{x_0}(x_0,\e_{i,j}))\Big).
 \end{align*}
and \eqref{eqn:keyequality} then follows at once from \eqref{e:quasi} by letting $i\uparrow+\infty$, 
in view of the choice $x_0\in C_u$. 
\smallskip

\noindent\textbf{Case~2:} \textit{$\xi=\pm\eta$}. Suppose, without loss of generality that $\xi=\eta$. We argue as in Case~1.
For the sequences $\{\rho_j\}_{j\in \N}$, $\{\e_{i,j}\}_{i\in \N}$ provided by Proposition~\ref{prop:vgblup} we have that 
\[
u_{P_j^{x_0},x_0,\e_{i,j}}\rightarrow v_j:=\psi_j(y\cdot \eta) \eta\qquad\text{ strict in $BD(P_j^{x_0})$}.
\]
By setting
    \begin{align*}
        v_{i,j}(y):= w_{i,j}(y)+\e_{i,j}\textstyle{\frac{|E u|(P_j^{x_0}(x_0,\e_{i,j}))}{\L^n(P_j^{x_0}(x_0,\e_{i,j}))}} c_j \eta ,
    \end{align*}
for 
\[
c_j:=\psi_j(-\rho_j/2)+\frac{\rho_j}{2|\eta\odot \eta|}\,,
\]
we conclude that 
\[
\lim_{j\rightarrow +\infty}\limsup_{i\rightarrow +\infty}
\frac{\mm(v_{i,j},P_j^{x_0}(x_0,\e_{i,j}))}{|Eu|(P_j^{x_0}(x_0,\e_{i,j}))}=
\lim_{i\rightarrow +\infty}\frac{\mm(u,P_j^{x_0}(x_0,\e_{i,j}))}{|Eu|(P_j^{x_0}(x_0,\e_{i,j}) )}.
\] 
We again combine this equality with \eqref{eqn:convtoF} to conclude.
\end{proof}
We now use assumption (H5) to prove a lower bound for the cell formula $\mm$ computed
on affine functions as done in \cite[Lemma 3.11]{bouchitte1998global}.
\begin{lemma}\label{lem:lwr bound}
Let $\F$ satisfy (H1)-(H5). For all $\mathrm{v}\in\R^n$, $\xi'\in\R^n\setminus\{\underline{0}\}$, $\eta\in\mathbb{S}^{n-1}$, $x_0\in\Omega$ 
and for every sequence $(t_i,\e_i)$  such that $t_i\rightarrow +\infty$ and $\e_i t_i \rightarrow 0$, 
and for every $\rho>0$, it holds 
        \begin{align*}
        f(x_0,\mathrm{v},\eta \odot \xi')-f(x_0,\mathrm{v},0)\leq \liminf_{i\rightarrow +\infty} 
        \frac{\mm(\mathrm{v}+t_i\,\eta \odot \xi' (\cdot-x_0), P_{\rho}^{\xi,\eta}(x_0,\e_i))}{t_i\L^n(P_{\rho}^{\xi,\eta}(x_0,\e_i))}
        \end{align*}
where $\xi:=\sfrac{\xi'}{|\xi'|}$, $P_{\rho}^{\xi,\eta}$ is defined either in \eqref{eqn:convexdefinition} or \eqref{eqn:convexdefinition1} according to whether $\xi\neq\pm\eta$ or not, and $f$ is the 
volume energy density defined in item (a) of Lemma~\ref{lem:goodpointsdensity}.
\end{lemma}
\begin{proof}
We start off noting that $\eta \otimes \xi'=
\eta\odot \xi'+\LL$, where $\LL:= \frac{1}{2}\left(\eta \otimes \xi'-\xi \otimes \eta'\right)\in \mathbb{M}^{n\times n}_{skew}$.
Then, in view of (H5) formula \eqref{eqn:H5onm} implies
    \begin{align}
    &\left|\mm\left(\mathrm{v}+t_i\, \eta \otimes \xi'(\cdot -x_0), P_{\rho}^{\xi,\eta}(x_0,\e_i)\right) 
    -\mm\left(\mathrm{v}+t_i\,\eta \odot \xi' (\cdot-x_0), P_{\rho}^{\xi,\eta}(x_0,\e_i)\right)\right|\nonumber\\
    &\qquad \qquad \leq C_{P_{\rho}^{\xi,\eta}} \Psi(\e_i t_i |\LL|)(1+t_i|\xi'|)\e_i^n.
    \end{align}
Hence, 
    \begin{align*}
    \liminf_{i\rightarrow +\infty} \frac{\mm(\mathrm{v}+ t_i \eta \odot \xi' (\cdot-x_0), P_{\rho}^{\xi,\eta}(x_0,\e_i))}{t_i\L^n(P_{\rho}^{\xi,\eta}(x_0,\e_i))}
  &=\liminf_{i\rightarrow +\infty} \frac{\mm(\mathrm{v}+t_i \eta \otimes \xi' (\cdot-x_0), P_{\rho}^{\xi,\eta}(x_0,\e_i))}{t_i\L^n(P_{\rho}^{\xi,\eta}(x_0,\e_i))}\\
  &\geq f(x_0,\mathrm{v},\eta \otimes \xi')-f(x_0,\mathrm{v},0).
    \end{align*}
For the last inequality we have used \cite[Lemma~3.11]{bouchitte1998global}. Moreover, since by \eqref{eqn:invariance} 
    \begin{align*}
        f(x_0,\mathrm{v},\eta \otimes \xi')  = f(x_0,\mathrm{v},\eta \odot \xi'),
    \end{align*}
the conclusion follows at once.
\end{proof}
Before proving Theorem~\ref{thm:mainthm}, we note that the continuity estimate on $\mm$ contained in \eqref{e:stimammcont}, deduced as a consequence of (H4), implies both
\begin{equation}\label{e:stimafcont}
|f(x_0+x,\mathrm{v}+\mathrm{v}_0,\mat)-f(x_0,\mathrm{v}_0,\mat)|\leq \Psi(|x|+|\mathrm{v}|)(1+|\mat|)
\end{equation}
and
\[
|g(x_0+x,\mathrm{v}_0+\mathrm{v},\mathrm{v}_1+\mathrm{v},\nu)-
g(x_0,\mathrm{v}_0,\mathrm{v}_1,\nu)|\leq \Psi(|x|+|\mathrm{v}|)|\mathrm{v}_0-\mathrm{v}_1|
\]
for all $(x_0,x,\mathrm{v}_0,\mathrm{v}_1,\mathrm{v},\nu,\mat)\in(\Omega)^2\times(\R^n)^3\times
\mathbb{S}^{n-1}\times\mathbb{M}^{n\times n}_{sym}$.

These properties are instrumental already in the $BV$ setting to express the Radon-Nikod\'ym derivative of $\F$ 
at $u$ with respect to $|E^cu|$ in terms of an energy density computed on relevant quantities related to the base 
function $u$ itself.
In particular, by taking such properties into account, one can prove that the recession function $f^\infty$ of 
the bulk energy density $f$ is actually the energy density of the Cantor part.
 
 \subsection{Proof of the integral representation result} 
 
\begin{proof}[Proof of Theorem~\ref{thm:mainthm}]
The representation of the volume and surface energy densities is dealt with in Lemma~\ref{lem:goodpointsdensity}.

We then turn to the representation of the energy density of the Cantor part. 
For $|E^cu|$-a.e. point $x_0\in C_u$ we may apply Lemma~\ref{lem:shapeofF} 
    (in what follows we keep the notation introduced there) and find infinitesimal sequences 
    $\{\rho_j\}_{j\in \N}$, $\{\e_{i,j}\}_{i\in \N}$ such that
    \begin{align}\label{eqn:myfinale0}
    \frac{\d \F(u,\cdot)}{\d |E^cu|}(x_0)=\frac{\d \F(u,\cdot)}{\d |E u|}(x_0)
    =\lim_{j\rightarrow +\infty} \limsup_{i\rightarrow+\infty} 
    \frac{\mm\left(w_{i,j},P_j^{x_0}(x_0,\e_{i,j})\right)}{|E u|(P_j^{x_0}(x_0,\e_{i,j}))}.
    \end{align}
On setting
\[
\mathrm{v}_{i,j}:=\fint_{P_j^{x_0}(x_0,\e_{i,j})}u(x)\d x, \qquad \LL_{i,j} :=\mathbb{M}_{P_j^{x_0}(x_0,\e_{i,j})}[u] 
\]
and $t_{i,j}:=\textstyle{\frac{|E u|(P_j^{x_0}(x_0,\e_{i,j}))}{\L^n(P_j^{x_0}(x_0,\e_{i,j}))}}$, we have
$w_{i,j}=\mathrm{v}_{i,j}+\LL_{i,j}(\cdot -x_0)+t_{i,j}\textstyle{\frac{\eta \odot \xi}{|\eta \odot \xi|}} (\cdot-x_0)$.
Note that $t_{i,j}\rightarrow +\infty$ as $i\rightarrow +\infty$ for all $j\in\N$ as $x_0\in C_u$. Moreover, recall that 
$x_0$ is a point of approximate continuity of $u$. 

Next we note that 
\begin{align*}
\big|\mm&\big(w_{i,j},P_j^{x_0}(x_0,\e_{i,j})\big)-
\mm\Big(u(x_0)+t_{i,j} \textstyle{\frac{\eta \odot \xi}{|\eta \odot \xi|}} (\cdot-x_0),P_j^{x_0}(x_0,\e_{i,j})\Big)\Big|\nonumber\\
&\leq\big|\mm\big(w_{i,j},P_j^{x_0}(x_0,\e_{i,j})\big)-
\mm\big(\mathrm{v}_{i,j}+t_{i,j} \textstyle{\frac{\eta \odot \xi}{|\eta \odot \xi|}} (\cdot-x_0),P_j^{x_0}(x_0,\e_{i,j})\big)\big|\nonumber\\
&+\Big|\mm\Big(\mathrm{v}_{i,j}+t_{i,j} \textstyle{\frac{\eta \odot \xi}{|\eta \odot \xi|}} 
(\cdot-x_0),P_j^{x_0}(x_0,\e_{i,j})\Big)-
\mm\Big(u(x_0)+t_{i,j} \textstyle{\frac{\eta \odot \xi}{|\eta \odot \xi|}} (\cdot-x_0), P_j^{x_0}(x_0,\e_{i,j})\Big)\Big|\nonumber\\
&\stackrel{\eqref{eqn:H5onm},\,\eqref{e:stimammcont}}{\leq} 
C_{P_j^{x_0}}\big(\Psi(\e_{i,j}\mathrm{diam}(P_j^{x_0}) |\LL_{i,j}|)+
\Psi(|\mathrm{v}_{i,j}-u(x_0)|)\big)(1+t_{i,j})\e_{i,j}^n.
\end{align*}
By taking into account $\mathrm{v}_{i,j} \to u(x_0)$ and $\e_{i,j}|\LL_{i,j}|\to 0$ as $i\uparrow +\infty$
thanks to Lemma~\ref{lem:killtheantisym}, the latter estimate combined with \eqref{eqn:myfinale0} leads to
    \begin{align}\label{eqn:myfinale}
    \frac{\d \F(u,\cdot)}{\d |E^cu|}(x_0)&=\frac{\d \F(u,\cdot)}{\d |E u|}(x_0)\nonumber\\&
    =\lim_{j\rightarrow +\infty} \limsup_{i\rightarrow+\infty}\frac{\mm\left(
     u(x_0)+t_{i,j} \frac{\eta \odot \xi}{|\eta \odot \xi|} (\cdot-x_0),P_j^{x_0}(x_0,\e_{i,j})\right)}{\L^n(P_j^{x_0}(x_0,\e_{i,j})) t_{i,j}}.
    \end{align}
With fixed $j\in\N$ and $\lambda>0$, by applying Lemma~\ref{lem:lwr bound} with 
$\xi'=\lambda\frac\xi{|\eta\odot\xi|}$ and $t_i:=\sfrac{t_{i,j}}{\lambda}$, Eq. \eqref{eqn:myfinale} implies 
    \begin{equation*}
        \frac{\d \F(u,\cdot)}{\d |E^cu|}(x_0)\geq \frac{f\left(x_0,u(x_0),\lambda \frac{\eta \odot \xi}{|\eta \odot \xi|}\right)-f(x_0,u(x_0),0)}{\lambda}. 
    \end{equation*}
Hence, by taking the superior limit as $\lambda\uparrow+\infty$ we infer
    \begin{equation}\label{eqn:afinal1}
        \frac{\d \F(u,\cdot)}{\d |E^cu|}(x_0)\geq 
        f^\infty\big(x_0,u(x_0),\textstyle{\frac{\eta \odot \xi}{|\eta \odot \xi|}}\big).
    \end{equation}

On the other hand, using as a competitor in the cell problem defining 
\[
\mm\Big(u(x_0)+ t_{i,j}\textstyle{\frac{\eta \odot \xi}{|\eta \odot \xi|}} (\cdot -x_0), P_j^{x_0}(x_0,\e_{i,j})\Big)
\] 
the affine map $u(x_0)+t_{i,j}\frac{\eta \odot \xi}{|\eta \odot \xi|} (\cdot-x_0)$ itself, we can apply item (a) in Lemma~\ref{lem:goodpointsdensity} to deduce
    \begin{align*}
        &\frac{\mm\left(u(x_0)+t_{i,j} \frac{\eta \odot \xi}{|\eta \odot \xi|} (\cdot-x_0), P_j^{x_0}(x_0,\e_{i,j})\right)}{t_{i,j}\L^n(P_j^{x_0}(x_0,\e_{i,j}))}
        \leq  \frac{\mathcal{F}\Big(u(x_0)+t_{i,j}\frac{\eta \odot \xi}{|\eta \odot \xi|} (\cdot-x_0), P_j^{x_0}(x_0,\e_{i,j})\Big)}{t_{i,j}\L^n(P_j^{x_0}(x_0,\e_{i,j}))}\\
        &= \fint_{P_j^{x_0}(x_0,\e_{i,j})}\frac1{t_{i,j}}f\Big(x,u(x_0)+t_{i,j}
        \textstyle{\frac{\eta \odot \xi}{|\eta \odot \xi|}(x-x_0),t_{i,j}\frac{\eta \odot \xi}{|\eta \odot \xi|}}\Big)\d x\\
&\stackrel{\eqref{e:stimafcont}}{\leq}\frac1{t_{i,j}}
f\Big(x_0,u(x_0),t_{i,j}\textstyle{\frac{\eta\odot \xi}{|\eta\odot \xi|}}\Big)
+2\Psi(\e_{i,j}+\e_{i,j}t_{i,j})\\
& \leq
\frac1{t_{i,j}}\Big(f\Big(x_0,u(x_0),t_{i,j}\textstyle{\frac{\eta\odot \xi}{|\eta\odot \xi|}}\Big)
-f(x_0, u(x_0),0)\Big)+\frac C{t_{i,j}}.
    \end{align*}
Since, $t_{i,j}\rightarrow +\infty$ as $i\rightarrow +\infty$ for all $j\in\N$, we deduce that
    \[
    \lim_{j\rightarrow+\infty} \limsup_{i\rightarrow+\infty}
    \frac{\mm\Big(u(x_0)+t_{i,j} \frac{\eta\odot \xi}{|\eta\odot \xi|} (\cdot-x_0), P_j^{x_0}(x_0,\e_{i,j})\Big)}{\L^n(P_j^{x_0}(x_0,\e_{i,j}))) t_{i,j}}\leq f^\infty\Big(x_0,u(x_0),\textstyle{\frac{\eta\odot \xi}{|\eta\odot \xi|}}\Big),
    \]
that combined with \eqref{eqn:myfinale} and \eqref{eqn:afinal1} finally leads to
    \[
    \frac{\d \F(u,\cdot)}{\d |E^cu|}(x_0)=
    f^\infty\Big(x_0,u(x_0),\frac{\d E^cu}{\d |E^cu|}(x_0)\Big).\qedhere
    \]
\end{proof}
A standard monotone approximation technique provides the following extension of Theorem~\ref{thm:mainthm} 
(see Section~\ref{s:discussionH5} for a similar argument).
\begin{corollary}
Let $\F:L^1(\Omega;\R^n)\times \mathcal{O}(\Omega)\to[0,+\infty]$ be satisfying (H1), (H3), (H4), (H5) 
and in place of (H2)
\begin{itemize}
\item[(HH2)] there exists a constant $C>0$ such that for every $(u,A)\in BD(\Omega)\times\mathcal{O}(\Omega)$
        \begin{equation*}
           0\leq \F(u,A) \leq C (\L^n(A)+|E u|(A)).
        \end{equation*}
\end{itemize}
Then, the conclusions of Theorem~\ref{thm:mainthm} hold for all $u\in BD(\Omega)$.
\end{corollary}
\begin{proof}
Let $\delta>0$, and consider the functionals $\F_\delta:L^1(\Omega;\R^n)\times \mathcal{O}(\Omega)\to[0,+\infty]$  
be defined by $\F_\delta(u,A):=\F(u,A)+\delta|Eu|(A)$. Since $\F_\delta$ satisfies all the conditions (H1)-(H5) of 
Theorem~\ref{thm:mainthm} there are two functions $f_\delta$ and $g_\delta$ such that $\F_\delta$ can be represented as in 
the statement there. The family of functionals $\F_\delta$ is pointwise increasing in $\delta$, therefore there exist the pointwise 
limits $f$ of $f_\delta$ and $g$ of $g_\delta$ as $\delta\downarrow 0$ (note that from the very definition of recession function 
$f^\infty_\delta(x,\mathrm{v},\mathbb{A})=f^\infty(x,\mathrm{v},\mathbb{A})+\delta|\mathbb{A}|$). 
As $\F_\delta(\cdot,A)$ is pointwise converging to $\F(\cdot,A)$ for $\delta\downarrow 0$ on $BD(\Omega)$ for all 
$A\in\mathcal{O}(\Omega)$, we conclude that the integral representation with energy densities $f$, $g$ and $f^\infty$ 
for the bulk, surface and Cantor parts, respectively, holds for $\F$. 
\end{proof}

\section{Some applications}\label{s:applications}

Following \cite[Section~4]{bouchitte1998global} we provide some applications of the integral representation 
Theorem~\ref{thm:mainthm} to the topics of relaxation of bulk energies, of bulk and interfacial energies, to that of $L^1$ lower semicontinuity of functionals defined on $BD$, and to that of homogenization of bulk energies. 

Throughout the section, for all $\nu\in\mathbb{S}^{n-1}$ and $r>0$  we denote the cubes $Q^\nu(\underline{0},r)$ by $Q^\nu_r$, 
and moreover $Q(\underline{0},r)$ by $Q_r$.

\subsection{Relaxation and $L^1$ lower semicontinuity of bulk energies}\label{ss:relax}

In this section we address the issue of giving an explicit expression to the $L^1$ lower semicontinuous envelope of a linearly growing functional defined on 
smooth maps, for instance $LD(\Omega)$.

Theorem~\ref{t:relaxationLD} below generalizes to $BD(\Omega)$ the results proven in \cite{barroso2000relaxation} and \cite{ebobisse2001note} on $SBD(\Omega)$. In particular, in \cite{barroso2000relaxation} a continuous autonomous integrand $f_0$ (i.e. depending only on the symmetric gradient) is considered, the integral representation is then given in terms of the symmetric quasiconvex envelope of $f_0$  (see the definition below) and its associated recession function.  In addition, Theorem~\ref{t:relaxationLD} also generalizes 
partially the results on $BD(\Omega)$ established in \cite{rindler2011lower} and \cite[Corollary~1.10]{arroyorabasadephilippisrindler}, \cite[Theorem~1.4]{kosibarindler}. Note that in the former, also a Dirichlet boundary condition is considered, while
in the last two integral representations of the weakly* lower semicontinuous envelope of functionals with linear growth at infinity are provided for more general PDEs constraints on the approximating sequences. 

We stress that in the ensuing result, the strong recession function is not required to exist, and that the integrand is allowed 
to depend also on $x$ and $\mathrm{v}$. Moreover, global continuity is replaced by the weaker condition (H2').

Let $f_0:\Omega\times\R^n\times\mathbb{M}^{n\times n}_{sym}\to[0,+\infty)$ be a Borel integrand satisfying:
\begin{itemize}
\item[(H1')] there exists a constant $C>0$ such that 
for every $(x,\mathrm{v},\mat)\in\Omega\times\R^n\times\mathbb{M}^{n\times n}_{sym}$ 
\begin{equation}
 \frac1C|\mat| \leq f_0(x,\mathrm{v},\mat)\leq C\big(1+|\mat|\big);
 \end{equation}
 \item[(H2')] there exists a constant $C>0$ such that for every $\e>0$ there exists 
 $\delta>0$ such that
 \[
  |x-x_0|+|\mathrm{v}-\mathrm{v}_0|\leq\delta\Longrightarrow
  |f_0(x,\mathrm{v},\mat)-f_0(x_0,\mathrm{v}_0,\mat)|\leq C\e(1+|\mat|),
 \]
 for every 
 $(x,\mathrm{v},\mathrm{v}_0,\mat)\in\Omega\times(\R^n)^2\times\mathbb{M}^{n\times n}_{sym}$.
\end{itemize}
Let then $\F_0:L^1(\Omega;\R^n)\times \mathcal{O}(\Omega)\to[0,+\infty]$ be the functional defined by 
\begin{equation}\label{e:F0}
 \F_0(u,A):=\inf\Big\{\liminf_{j\to+\infty}F_0(u_j,A):\,u_j\to u \text{ in $L^1(\Omega;\R^n)$}\Big\},
\end{equation}
namely the $L^1(\Omega;\R^n)$ lower semicontinuos envelope of the functional
\begin{equation}\label{e:F00}
F_0(u,A):=\begin{cases}
\displaystyle{\int_Af_0\big(x,u(x),e(u)(x)\big)\d x} & \text{if $u\in LD(\Omega)$} \\ 
                        & \\
                       +\infty & \text{otherwise on $L^1(\Omega;\R^n)$.}
                      \end{cases}
\end{equation}
We denote by $\mm$ the cell formula defined in \eqref{e:mm} and related to $\mathcal{F}_0$, and
recall the notation $u_{\mathrm{v}^-,\mathrm{v}^+,\nu}$ introduced in \eqref{e:uvpiuvmeno}. 
We also recall that $f^\infty$ stands for the (weak) recession function as defined by \eqref{e:cantor}.  
 \begin{theorem}\label{t:relaxationLD}
 Let $F_0:L^1(\Omega;\R^n)\times \mathcal{O}(\Omega)\to[0,+\infty]$ be the functional defined in \eqref{e:F00}.
 Then, assuming (H1')-(H2'), the functional $\F_0:L^1(\Omega;\R^n)\times \mathcal{O}(\Omega)\to[0,+\infty]$
 defined in \eqref{e:F0} is represented by
 \begin{align*}
  \F_0(u,A)&=\int_Af\big(x,u(x),e(u)(x)\big)\d x\\&
  +\int_{J_u\cap A}g\big(x,u^-(x),u^+(x),\nu_u(x)\big)\d\H^{n-1}(x)
    +\int_Af^\infty\Big(x,u(x),\frac{\d E^cu}{\d|E^cu|}(x)\Big)\d|E^cu|(x), 
    \end{align*}
for all $(u,A)\in BD(\Omega)\times \mathcal{O}(\Omega)$, where for every 
$(x_0,\mathrm{v},\mat)\in\Omega\times\R^n\times\mathbb{M}^{n\times n}_{sym}$
\begin{align}\label{e:fLD}
 f(x_0,\mathrm{v},\mat)
 =\limsup_{\e\to0}\inf_{\substack{w\in LD(Q_1)\\ w|_{\partial Q_1}=\mat y|_{\partial Q}}}%
\int_{Q_1}f_0\big(x_0,\mathrm{v}+\e w(y),e(w)(y)\big)\d y,
\end{align}
and for every $(x_0,\mathrm{v}^-,\mathrm{v}^+,\nu)\in\Omega\times(\R^n)^2\times\mathbb{S}^{n-1}$
\begin{align}\label{e:gLD}
g(x_0,\mathrm{v}^-,\mathrm{v}^+,\nu)=\limsup_{\e\to0}
\inf_{\substack{w\in LD(Q^\nu_1)\\w|_{\partial Q^\nu_1}=u_{\mathrm{v}^-,\mathrm{v}^+,\nu}|_{\partial Q^\nu_1}}}
\int_{Q^\nu_1}\e\, f_0\big(x_0,w(y),\textstyle{\frac1\e}e(w)(y)\big)\d y.
\end{align}
\end{theorem}
\begin{remark}\label{r:truncation}
Assumption  (H2') implies that $\F_0$ satisfies (H4).
Instead, in the $BV$-setting, the ensuing weaker assumption (H3') replaces  (H2') in \cite[Section~4.1]{bouchitte1998global}:
\begin{itemize}
 \item[(H3')] there exists a constant $C>0$ such that for every $\e>0$ there exists 
 $\delta>0$ such that
 \[
  |\mathrm{v}-\mathrm{v}_0|\leq\delta\Longrightarrow
  |f_0(x,\mathrm{v},\mat)-f_0(x,\mathrm{v}_0,\mat)|\leq C\e(1+|\mat|),
 \]
 for every $(x,\mathrm{v},\mathrm{v}_0,\mat)\in\Omega\times(\R^n)^2\times\mathbb{M}^{n\times n}_{sym}$.
\end{itemize}
The latter condition and a truncation argument 
(cf. \cite[Lemma~2.6]{bouchitte1998global}) are employed both to simplify the minimum problems defining $f$ and $g$ with respect to the $\mathrm{v}$-variable 
(cf. equations \cite[(4.1.5) and (4.1.6)]{bouchitte1998global}, and also to dispense with the analogue of (H4). Therefore, since the integral representation result in the $BV$ setting cannot be directly applied, 
the quoted truncation argument and the analogue of Lemma~\ref{lem:shapeofF} are central to give an 
explicit formula for the energy density of the Cantor part of the relaxed functionals analogous to $\F_0$ 
(cf. Eq. \cite[(4.1.7) in Theorem~4.1.4]{bouchitte1998global} and 
\cite[Remark~4.1.5]{bouchitte1998global}). 
Instead, since truncations are not permitted in the current $BD$ setting, we need the stronger assumption (H2') to enforce (H4).
\end{remark}

\begin{remark}\label{r:H1 weakened}
In order to prove the integral representation of $\F_0$ over the subspace $SBD(\Omega)$ 
only, assumption (H2') is actually not needed and (H1') can be weakened. Indeed, to that aim one can allow for $f_0$ to depend also on 
the skew-symmetric part of the given matrix in view of Lemma~\ref{lem:goodpointsdensity} (cf. Section~\ref{s:discussionH5}). 
Clearly, formulas \eqref{e:fLD} defining $f$ and \eqref{e:gLD} defining $g$ have to be changed accordingly.
\end{remark}
\begin{remark}\label{r:LDvssobolev}
In view of the density of $W^{1,1}(\Omega;\R^n)$ in $BD(\Omega)$ with respect to the strict topology 
with assigned boundary trace (cf. \cite[Theorem~2.6]{barroso2000relaxation}), the space $LD$ can 
be substituted with $W^{1,1}$ in the minimum problems defining $f$ and $g$. 

Therefore, the same conclusions of Theorem~\ref{t:relaxationLD} can be drawn if we consider the functional 
$F'_0(u,A):=F_0(u,A)$ for $u\in W^{1,1}(\Omega;\R^n)$, and $+\infty$ otherwise on $L^1(\Omega;\R^n)$. 
Then the space of test maps for the minimum problems defining $f$ and $g$ in \eqref{e:fLD} and \eqref{e:gLD} respectively, 
is $W^{1,1}(\Omega;\R^n)$.
\end{remark}
The main steps to prove Theorem~\ref{t:relaxationLD} are similar to those exploited for the analogous result in the $BV$ setting in \cite[Section~4.1]{bouchitte1998global} 
to which we refer. Therefore, we provide only a sketch of those proofs for which some changes are needed. 
\begin{lemma}[Lemma~4.1.3 \cite{bouchitte1998global}]\label{l:mm0}
Assume (H1')-(H2'). Then, $\F_0$ satisfies (H1)-(H5), and for all $(u,A)\in BD(\Omega)\times \mathcal{O}_{\infty}(\Omega)$ 
\[
\mm(u,A)=\mm_0(u,A):=\inf\{F_0(w,A):\,w\in BD(A),\,w=u \textrm{ on $\partial A$}\}.
\]
\end{lemma}
\begin{proof}
Assumptions (H1), (H2) and (H5) are easily checked to be satisfied by $\F_0$ in view of (H1')
and the very definition of $\F_0$. Instead, (H4) follows from (H2').
For what concerns (H3) one can argue similarly to \cite[Proposition~3.9]{barroso2000relaxation}.

The inequality $\mm(u,A)\leq\mm_0(u,A)$ follows from $\F_0\leq F_0$. 
For the opposite, one uses the very definition of $\F_0$ as the relaxation of $F_0$ together 
with the version of the De~Giorgi's averaging/slicing lemma stated in 
\cite[Proposition~3.7]{barroso2000relaxation}.
\end{proof}
By means of the alternative characterization of $\mm$ provided by $\mm_0$, of 
Theorem~\ref{thm:mainthm}, of the results in Section~4 and of a change of variable, Theorem~\ref{t:relaxationLD} follows at once.

If an additional quantified closeness condition between $f$ computed on large gradients and $f^\infty$ is added, 
formula \eqref{e:gLD} defining $g$ can be simplified. The ensuing claim \eqref{e:gLDbis} is 
well-known in the in the $BV$ case (cf. \cite[Theorem~4.1.4]{bouchitte1998global}). 
\begin{corollary}\label{c:relaxationLD}
 Under the assumptions and notation of Theorem~\ref{t:relaxationLD} and in addition 
 \begin{itemize}
  \item[(H4')] on setting for any $L>0$ and $x_0\in\Omega$
\[
\omega_{f_0}(x_0,L):=\sup_{\substack{\mathrm{v}\in\R^n,\,t\geq L \\ \mat\in\mathbb{M}^{n\times n}_{sym}:\,|\mat|=1}}\Big|f_0^\infty(x_0,\mathrm{v},\mat)
-\frac 1tf_0(x_0,\mathrm{v},t\mat)\Big|\,,
\]  
then $\omega_{f_0}(x_0,L)$ is infinitesimal as $L\to+\infty$,
 \end{itemize}
the function $g$ in Eq. \eqref{e:gLD} in the conclusions of 
Theorem~\ref{t:relaxationLD} is characterized alternatively by
\begin{align}\label{e:gLDbis}
 g(x_0,\mathrm{v}^-,\mathrm{v}^+,\nu)=
\inf_{\substack{w\in LD(Q^\nu_1)\\w|_{\partial Q^\nu_1}=u_{\mathrm{v}^-,\mathrm{v}^+,\nu}|_{\partial Q^\nu_1}}}
\int_{Q^\nu_1}f_0^\infty&\big(x_0,w(y),e(w)(y)\big)\d y.
\end{align}
\end{corollary}
Furthermore, we deal with the $\mathrm{v}$-independent case for which (H4') is actually not needed (cf. \cite[Remark~2.17]{fonsecamuller} 
for the analogous result in the $BV$ setting). We start off with a preliminary result.
\begin{corollary}\label{c:vindependent}
Under the assumptions and notation of Theorem~\ref{t:relaxationLD}, if the integrand $f_0$ satisfies
$f_0(x,\mathrm{v},\mat)=f_0(x,\mat)$ for every $(x,\mathrm{v},\mat)\in\Omega\times\R^n\times\mathbb{M}^{n\times n}_{sym}$, then
\begin{equation*}
 g(x_0,\mathrm{v}^-,\mathrm{v}^+,\nu)=f^\infty\big(x_0,(\mathrm{v}^+-\mathrm{v}^-)\odot\nu\big)
\end{equation*}
for every $(x_0,\mathrm{v}^-,\mathrm{v}^+,\nu)\in\Omega\times(\R^n)^2\times\mathbb{S}^{n-1}$. 
\end{corollary}
\begin{proof}
We start off defining
\begin{equation*}
\widetilde{F}_0(u,A):=\begin{cases}
\displaystyle{\int_Af\big(x,e(u)(x)\big)\d x} & \text{if $u\in LD(\Omega)$} \\ 
                       +\infty & \text{otherwise on $L^1(\Omega;\R^n)$,}
                      \end{cases}
\end{equation*}
where $f$ is given by \eqref{e:fLD}. Note that $f$ is $\mathrm{v}$-independent as
$\F_0(u+\mathrm{z},A)=\F_0(u,A)$ for all $(u,A,\mathrm{z})\in BD(\Omega)\times\mathcal{O}(\Omega)\times\R^n$.
Moreover, $f(x_0,\mat)\leq f_0(x_0,\mat)$, for all $(x_0,\mat)\in\Omega\times\mathbb{M}^{n\times n}_{sym}$, by using 
the linear function $\mat y$ itself as a test in \eqref{e:fLD} and (H2').

Denote by $\widetilde{\mathcal{F}}_0$ the $L^1(\Omega;\R^n)$ lower semicontinuous envelope of $\widetilde{F}_0$.
Then, $\widetilde{F}_0\leq F_0$ implies that $\widetilde{\mathcal{F}}_0\leq\F_0$ on $L^1(\Omega;\R^n)$, and since $\F_0=\widetilde{F}_0$ on $LD(\Omega)$
and $\F_0\leq\widetilde{F}_0$ otherwise, we conclude that 
$\F_0\equiv\widetilde{\mathcal{F}}_0$.
Therefore, equality \eqref{e:gLD} defining $g$ holds with $f$ in place of $f_0$ in the minimum problem there. 
In passing, we point out that the invariance of $f_0$ implies that $g=g(x,\mathrm{v}^+-\mathrm{v}^-,\nu)$, as well
(cf. Remark~\ref{r:H5bis}).

Let us first prove that $g(x_0,\mathrm{v}^+-\mathrm{v}^-,\nu)\geq 
f^\infty\big(x_0,(\mathrm{v}^+-\mathrm{v}^+)\odot\nu\big)$.
Let $\{\nu_1,\ldots,\nu_{n-1},\nu\}$ form an orthonormal basis of $\R^n$ and set 
\begin{align*}
\mathcal{A}^\nu:=\big\{w\in W^{1,1}(Q^\nu_1;\R^n):\,& w|_{\{x\in\partial Q^\nu_1:\,x\cdot\nu=\pm\sfrac12\}}=0,\\
 &\textrm{$w$ is $1$-periodic in $\nu_i$, $1\le i \le n-1$}\big\},\\ 
\end{align*}
and
\begin{align*}
\mathcal{B}^\nu:=\big\{w\in W^{1,1}(Q^\nu_1;\R^n):\,& (w-u_{\mathrm{v}^-,\mathrm{v}^+,\nu})|_{\{x\in\partial Q^\nu_1:\,x\cdot\nu=\pm\sfrac12\}}=0,\\
 &\textrm{$w$ is $1$-periodic in $\nu_i$, $1\le i \le n-1$}\big\},\\ 
\end{align*}
In particular, note that $\mathcal{B}^\nu=\zeta(\cdot)+\mathcal{A}^\nu$, where
\begin{equation}\label{e:zeta}
\zeta(x):=\frac{\mathrm{v}^++\mathrm{v}^-}{2}+(\mathrm{v}^+-\mathrm{v}^-)x\cdot\nu.
\end{equation}
We then argue as follows
\begin{align}\label{e:gfinfty}
 g&(x_0,\mathrm{v}^+-\mathrm{v}^-,\nu)
=\limsup_{\e\to0}\inf_{\substack{w\in LD(Q^\nu_1)\\w|_{\partial Q^\nu_1}=u_{\mathrm{v}^-,\mathrm{v}^+,\nu}|_{\partial Q^\nu_1}}}
\int_{Q^\nu_1}\e f\big(x_0,{\textstyle\frac1\e}e(w)(y)\big)\d y\nonumber\\
&=\limsup_{\e\to0}\inf_{\substack{w\in W^{1,1}(Q^\nu_1;\R^n)\\w|_{\partial Q^\nu_1}=u_{\mathrm{v}^-,\mathrm{v}^+,\nu}|_{\partial Q^\nu_1}}}
\int_{Q^\nu_1}\e f\big(x_0,{\textstyle\frac1\e}e(w)(y)\big)\d y\nonumber\\
&\geq\limsup_{\e\to0}\inf_{w\in\mathcal{B}^\nu}
\int_{Q^\nu_1}\e f\big(x_0,{\textstyle\frac1\e}e(w)(y)\big)\d y\nonumber\\
&=\limsup_{\e\to0}\inf_{w\in\mathcal{A}^\nu}\int_{Q^\nu_1}\e f\big(x_0,{\textstyle\frac1\e}
(\mathrm{v}^+-\mathrm{v}^-)\odot\nu+{\textstyle\frac1\e}e(w)(y)\big)\d y\,,
\end{align}
where for the second equality we have used Remark~\ref{r:LDvssobolev}, and for the last, Eq. \eqref{e:zeta}. Using the characterization of symmetric quasiconvexity expressed in \eqref{e:sqcPER} we conclude from \eqref{e:gfinfty} that 
\[
g(x_0,\mathrm{v}^+-\mathrm{v}^-,\nu)\geq 
\limsup_{\e\to0}\e f\big(x_0,{\textstyle\frac1\e}(\mathrm{v}^+-\mathrm{v}^-)\odot\nu\big)=
f^\infty\big(x_0,(\mathrm{v}^+-\mathrm{v}^-)\odot\nu\big).
\]
To prove the opposite inequality, consider the affine function $\zeta$ in \eqref{e:zeta}, 
extend it by $1$-periodicity in the directions $\nu_i$, $1\leq i\leq n-1$, and extend 
it further by $\mathrm{v}^\pm$ if $\pm x\cdot \nu\geq\sfrac12$ (with a slight 
abuse we keep the same notation for the extended function). 
Next let $\widetilde{w}\in W^{1,1}(Q^\nu_1;\R^n)$ have the same trace of $u_{\mathrm{v}^-,\mathrm{v}^+,\nu}$
on $\partial Q^\nu_1$ (cf. \cite[Theorem 2.6]{barroso2000relaxation}), extend it by $1$-periodicity 
in the directions $\nu_i$, $1\leq i\leq n-1$, and then extend it by $u_{\mathrm{v}^-,\mathrm{v}^+,\nu}$
on the complement set (again we keep the same notation for the extended function).
Let $r$ and $\delta\in(0,1)$ and fix a cut-off function $\varphi\in W^{1,\infty}_0(Q^\nu_1;[0,1])$ such that $\varphi|_{Q^\nu_r}\equiv 1$.
Define $\zeta_{r,\delta}:=\varphi\,\zeta(\sfrac{\cdot}\delta)+(1-\varphi)\widetilde{w}(\sfrac{\cdot}\delta)\in W^{1,1}(Q^\nu_1;\R^n)$ 
and use it as a test function in the minimum problem in \eqref{e:gLD}. Note that $\zeta_{r,\delta}=\mathrm{v}^\pm$ if $\pm x\cdot\nu\geq\sfrac12$.
A simple computation yields,
\[
e(\zeta_{r,\delta})=\textstyle{{\frac1\delta}}\big(\varphi\,e(\zeta)(\sfrac\cdot\delta)+(1-\varphi)e(\widetilde{w})(\sfrac{\cdot}\delta)\big)
+\nabla\varphi\odot(\zeta(\sfrac\cdot\delta)-\widetilde{w}(\sfrac\cdot\delta)),
\] 
thus using (H1'), a simple scaling argument and periodicity give for some $\omega_\delta\downarrow 0$ as $\delta\downarrow 0$
\begin{align*}
g(&x_0,\mathrm{v}^+-\mathrm{v}^-,\nu)\leq\limsup_{\e\to0}\int_{Q^\nu_1}\e f\big(x_0,{\textstyle{\frac1\e}}e(\zeta_{r,\delta})(y)\big)\d y\\
=&\limsup_{\e\to0}\Big(\int_{Q^\nu_1\cap\{|x\cdot\nu|\leq\sfrac\delta2\}}\e f\big(x_0,{\textstyle{\frac1\e}}e(\zeta_{r,\delta})(y)\big)\d y+\e f\big(x_0,0\big)\Big)\\
\leq&(1+\omega_\delta)\limsup_{\e\to0}\int_{Q^\nu_1}\e\delta f\big(x_0,{\textstyle{\frac1{\e\delta}}}(\mathrm{v}^+-\mathrm{v}^+)\odot\nu\big)\d y\\
&+C(1+\omega_\delta)\int_{(Q^\nu_1\setminus Q^\nu_r)\cap\{|x\cdot\nu|\leq\sfrac\delta2\}} 
\big({\textstyle{\frac1\delta}}|e(\zeta)(\sfrac{y}\delta)|+{\textstyle{\frac1\delta}}|e(\widetilde{w})(\sfrac{y}\delta)|
+|\nabla\varphi||\zeta(\sfrac{y}\delta)-\widetilde{w}(\sfrac{y}\delta)|\big)\d y\\
\leq& (1+\omega_\delta)f^\infty\big(x_0,(\mathrm{v}^+-\mathrm{v}^-)\odot\nu\big)
+C\int_{Q^\nu_1\setminus Q^\nu_r}\big(|e(\zeta)|+|e(\widetilde{w})|\big)\d y
\\
&+\frac{C\delta}{1-r}\int_{Q^\nu_1\setminus Q^\nu_r}|\zeta(y)-\widetilde{w}(y)|\d y.
\end{align*}
We conclude 
\[
g(x_0,\mathrm{v}^+-\mathrm{v}^-,\nu)\leq f^\infty\big(x_0,(\mathrm{v}^+-\mathrm{v}^+)\odot\nu\big),
\]
by letting first $\delta\downarrow 0$ and then $r\uparrow 1$ in the latter inequality.
\end{proof}
In view of the previous corollary we are able to characterize explicitly the relaxed functional $\F_0$ in the $\mathrm{v}$-independent case.
Additionally, as a consequence, we are also able to deal with the issue of $L^1$ lower semicontinuity on $BD$. In particular, we improve upon \cite{rindler2011lower} 
and \cite[Corollary~1.10]{arroyorabasadephilippisrindler} (in the curl-curl case according to terminology used there) dispensing with 
the existence of the strong recession function. 
Note that $f_0$ is allowed to depend on $x$ and that  the full continuity of $f_0$ is not required, being replaced by the weaker assumption (H2').  

To state the result recall that given $f:\mathbb{M}^{n\times n}_{sym}\rightarrow \R$ a Borel function, 
its symmetric quasiconvex envelope $SQf:\mathbb{M}^{n\times n}_{sym}\rightarrow \R$ is defined to be
	\[
	SQf(\mat):=\sup\{h(\mat) \  | \ h\leq f, \ \text{ $h$ is symmetric quasiconvex}\}.
	\]
Clearly, $f$ is symmetric quasiconvex if and only if $f=SQf$. 
Finally, if $f:\Omega \times \mathbb{M}^{n\times n}_{sym}\rightarrow [0,+\infty)$ we write 
$SQf(x,\cdot):=SQ(f(x,\cdot))$ for all $x\in \Omega$.

\begin{corollary}\label{c:lsc}
Let $f_0:\Omega\times\mathbb{M}^{n\times n}_{sym}\to[0,+\infty)$ be a Borel function satisfying (H1')-(H2'). 
Let $F_0:L^1(\Omega;\R^n)\times \mathcal{O}(\Omega)\to[0,+\infty]$ be the functional defined in \eqref{e:F00}
corresponding to $f_0$. 

Then, for all $(u,A)\in BD(\Omega)\times\mathcal{O}(\Omega)$
\begin{equation}\label{e:Flsc0}
\F_0(u,A)=\int_A SQf_0\big(x, e(u)(x)\big)\d x
+\int_A(SQf_0)^\infty\Big(x,\frac{\d E^su}{\d|E^su|}(x)\Big)\d|E^su|.
\end{equation}

In particular, if $f_0:\Omega\times\mathbb{M}^{n\times n}_{sym}\to[0,+\infty)$ is a Borel symmetric quasiconvex satisfying (H1')-(H2'), 
the functional $\widetilde{\F}_0:L^1(\Omega;\R^n)\times\mathcal{O}(\Omega)\to[0,+\infty]$ 
defined by
\begin{equation}\label{e:Flsc}
\widetilde{\F}_0(u,A):=\int_A f_0\big(x, e(u)(x)\big)\d x
+\int_A f_0^\infty\Big(x,\frac{\d E^su}{\d|E^su|}(x)\Big)\d|E^su|
\end{equation}
if $u\in BD(\Omega)$ and $+\infty$ otherwise, is $L^1(\Omega;\R^n)$ lower semicontinuous.
\end{corollary}
\begin{proof}
We start off noting that by \eqref{e:fLD}, by inequality $SQf_0\leq f_0$, and by the symmetric quasiconvexity of $SQf_0$ we get
\begin{align*}
f(x_0,\mat)&= \inf_{\substack{w\in LD(Q_1)\\ w|_{\partial Q_1}=\mat y|_{\partial Q_1}}}\int_Qf_0\big(x_0,e(w)(y)\big)\d y
\\&
\geq \inf_{\substack{w\in LD(Q_1)\\ w|_{\partial Q_1}=\mat y|_{\partial Q_1}}}\int_{Q_1} SQf_0\big(x_0,e(w)(y)\big)\d y\geq SQf_0(x_0,\mat).
\end{align*}
As noticed in Corollary~\ref{c:vindependent}, the bulk energy density $f$ is symmetric quasiconvex and $f\leq f_0$. 
Thus, by the very definition of $SQf_0$, we get $f\leq SQf_0$. Therefore, the representation for $\F_0$ in \eqref{e:Flsc0} is attained
thanks to Corollary~\ref{c:vindependent}. 

Finally, the $L^1(\Omega;\R^n)$ lower semicontinuity of $\widetilde{\F}_0$ in \eqref{e:Flsc} follows at once. 
\end{proof}

\subsection{Relaxation of bulk and interfacial energies}\label{ss:relaxbulksurface}

In this section we consider linear functionals defined on the subspace $SBD(\Omega)$ and provide a relaxation result for them. 
To our knowledge this is the first result of this kind.

We introduce the notation required for our purposes following \cite[Section~4.2]{bouchitte1998global}. 
Let $f_1:\Omega\times\R^n\times\mathbb{M}^{n\times n}_{sym}\to[0,\infty)$ and
$g_1:\Omega\times(\R^n)^2\times\mathbb{S}^{n-1}\to[0,\infty)$ be continuous integrands such that 
\begin{itemize}
\item[(H1'')] there exists a constant $C>0$ such that for every 
$(x,\mathrm{v},\mat)\in\Omega\times\R^n\times\mathbb{M}^{n\times n}_{sym}$ 
\begin{equation}
 \frac1C|\mat|
 \leq f_1(x,\mathrm{v},\mat)\leq C\big(1+|\mat|\big);
\end{equation}
\item[(H2'')] there exists a constant $C>0$ such that for every $\e>0$ there exists 
$\delta>0$ such that
\[
 |x-x_0|+|\mathrm{v}-\mathrm{v}_0|\leq\delta\Longrightarrow
 |f_1(x,\mathrm{v},\mat)-f_1(x_0,\mathrm{v}_0,\mat)|\leq C\e(1+|\mat|),
\]
for every $(x,x_0,\mathrm{v},\mathrm{v}_0,\mat)\in(\Omega)^2\times(\R^n)^2\times\mathbb{M}^{n\times n}_{sym}$;

\item[(H3'')] there exist $C>0$, such that for every $(x,\mathrm{v}^-,\mathrm{v}^+,\nu)\in\Omega\times(\R^n)^2\times\mathbb{S}^{n-1}$ 
\[
\frac1C|\mathrm{v}^+-\mathrm{v}^-|\leq g_1(x,\mathrm{v}^-,\mathrm{v}^+,\nu)
\leq C|\mathrm{v}^+-\mathrm{v}^-|.
\]
\item[(H4'')] there exist $C>0$, such that for every $(x,x_0,\mathrm{v}^-,\mathrm{v}^+,\mathrm{v}_0,\nu)\in(\Omega)^2\times(\R^n)^3\times\mathbb{S}^{n-1}$;
\[
 |x-x_0|+|\mathrm{v}_0|\leq\delta\Longrightarrow
 |g_1(x,\mathrm{v}^-+\mathrm{v}_0,\mathrm{v}^++\mathrm{v}_0,\nu)-g_1(x_0,\mathrm{v}^-,\mathrm{v}^+,\nu)|\leq C\e|\mathrm{v}^+-\mathrm{v}^-|.
 \]
\end{itemize}
Let then $\F_1:L^1(\Omega;\R^n)\times \mathcal{O}(\Omega)\to[0,+\infty]$ be the functional defined by 
\begin{equation}\label{e:F1}
 \F_1(u,A):=\inf\Big\{\liminf_jF_1(u_j,\Omega):\,u_j\to u \text{ in $L^1(\Omega;\R^n)$}\Big\},
\end{equation}
namely the $L^1(\Omega;\R^n)$ lower semicontinuos envelope of the functional
\begin{equation}\label{e:F11}
F_1(u,A):=\begin{cases}
\displaystyle{\int_Af_1\big(x,u(x),e(u)(x)\big)\d x}  \\
\qquad\displaystyle{+\int_{J_u\cap A}g_1\big(x,u^-(x),u^+(x),\nu_u(x)\big)\d\H^{n-1}(x)} 
& \text{if $u\in SBD(\Omega)$} \\ 
                        & \\
                       +\infty & \text{otherwise on $L^1(\Omega;\R^n)$.}
                      \end{cases}
\end{equation}
Denote by $\mm$ the cell formula defined in \eqref{e:mm} related to $\mathcal{F}_1$. We provide next an integral representation
result for $\F_1$
\begin{theorem}\label{t:relaxationSBD}
Let $F_1:L^1(\Omega;\R^n)\times \mathcal{O}(\Omega)\to[0,+\infty]$ be the functional defined in \eqref{e:F11}.
 Then, assuming (H1")-(H4"), the functional $\F_1:L^1(\Omega;\R^n)\times \mathcal{O}(\Omega)\to[0,+\infty]$ 
defined in \eqref{e:F1} is represented by 
 \begin{align*}
  \F_1(u,A)&=\int_Af\big(x,u(x),e(u)(x)\big)\d x\\&
  +\int_{J_u\cap A}g\big(x,u^-(x),u^+(x),\nu_u(x)\big)\d\H^{n-1}(x)+\int_A f^\infty\Big(x,u(x),\frac{\d E^cu}{\d|E^cu|}(x)\Big)\d|E^cu|(x),
 \end{align*}
 for all $(u,A)\in BD(\Omega)\times \mathcal{O}(\Omega)$, 
where for every $(x_0,\mathrm{v},\mat)\in\Omega\times\R^n\times\mathbb{M}^{n\times n}_{sym}$
\begin{align}\label{e:fSBD}
 f(x_0,\mathrm{v},\mat)=\limsup_{\e\to0}&\inf_{\substack{w\in SBD(Q_1)\\ w|_{\partial Q_1}=\mat y|_{\partial Q_1}}}
 \Big\{\int_{Q_1}f_1\big(x_0,\mathrm{v}+\e w(y),e(w)(y)\big)\d y\nonumber\\
 &+\int_{J_w\cap Q_1}\textstyle{\frac1\e} g_1(x_0,\mathrm{v}+\e w^-(y),\mathrm{v}+\e w^+(y),\nu_w(y))\d\H^{n-1}(y)\Big\},
\end{align}
and for every $(x_0,\mathrm{v}^-,\mathrm{v}^+,\nu)\in\Omega\times(\R^n)^2\times\mathbb{S}^{n-1}$
\begin{align}\label{e:gSBD}
 g(x_0,\mathrm{v}^-,\mathrm{v}^+,\nu)=\limsup_{\e\to0}&
\inf_{\substack{w\in SBD(Q^\nu_1)\\ w|_{\partial Q^\nu_1}=u_{\mathrm{v}^-,\mathrm{v}^+,\nu}|_{\partial Q^\nu_1}}}
 \Big\{\int_{Q^\nu_1}\e\,f_1\big(x_0,\mathrm{v}+\e w(y),\textstyle{\frac1\e}e(w)(y)\big)\d y\nonumber\\
 &+\int_{J_w\cap Q^\nu_1}g_1(x_0,w^-(y),w^+(y),\nu_w(y))\d\H^{n-1}(y)\Big\}.
\end{align}
\end{theorem}
\begin{remark}\label{r:truncation2}
In the $BV$ setting under study in \cite{bouchitte1998global} conditions (H3") and (H4") are additionally used to simplify 
the analogue of formulas \eqref{e:fSBD} for $f$ and \eqref{e:gSBD} for $g$ thanks to the truncation argument quoted in 
Remark~\ref{r:truncation} (cf. equations \cite[(4.2.3) and (4.2.4) in Theorem~4.2.2]{bouchitte1998global}). 
\end{remark}
\begin{remark}
Formulas \eqref{e:fSBD} for $f$ and \eqref{e:gSBD} for $g$ can be expressed in terms of the recession functions of $f_1$ at $\infty$, $f_1^\infty$, 
and the recession function of $g_1$ at $0$, $g_1^0$, provided some further technical conditions in the spirit of (H3') in Corollary~\ref{c:relaxationLD}
are imposed (cf. Eq. \cite[(4.2.3)' and (4.2.4)' in Remark~4.2.3]{bouchitte1998global}).
\end{remark}

The proof of Theorem~\ref{t:relaxationSBD} is similar to the corresponding one of \cite[Theorem~4.2.2]{bouchitte1998global}.
First, we note that arguing as in Lemma~\ref{l:mm0} one can prove the following result.
\begin{lemma}
Assume (H1")-(H4"). Then, $\F_1$ satisfies (H1)-(H5), and for all $(u,A)\in BD(\Omega)\times \mathcal{O}_{\infty}(\Omega)$ 
\[
\mm(u,A)=\mm_1(u,A):=\inf\{F_1(w,A):\,w=u \textrm{ on $\partial A$}\}.
\]
\end{lemma}
The alternative characterization of $\mm$ via $\mm_1$, 
Theorem~\ref{thm:mainthm}, the results in Section~4 and a change of variable 
provide the proof of Theorem~\ref{t:relaxationSBD}.\\
\smallskip
We give next an explicit application of Theorem~\ref{t:relaxationSBD}. 
\begin{corollary}
Let $f_1:\Omega\times\mathbb{M}^{n\times n}_{sym}\to[0,+\infty)$ be a Borel, symmetric quasiconvex function satisfying (H1")-(H2").
 
If $F_1:L^1(\Omega;\R^n)\times \mathcal{O}(\Omega)\to[0,+\infty]$ is the functional in \eqref{e:F11} corresponding to 
$f_1$ and $g_1=f_1^\infty$, then 
\begin{equation}\label{F1final}
\F_1(u,A)=\int_A f_1\big(x, e(u)(x)\big)\d x
+\int_{A}f^\infty_1\Big(x,\frac{\d E^su}{\d|E^su|}(x)\Big)\d |E^su|
\end{equation}
if $u\in BD(\Omega)$ and $+\infty$ otherwise.
\end{corollary}
\begin{proof}
We start off defining
\begin{equation*}
\widetilde{F}_1(u,A):=\begin{cases}
\displaystyle{\int_Af_1\big(x,e(u)(x)\big)\d x} & \text{if $u\in LD(\Omega)$} \\ 
                       +\infty & \text{otherwise on $L^1(\Omega;\R^n)$.}
                      \end{cases}
\end{equation*}
Denote by $\widetilde{\mathcal{F}}_1$ the $L^1(\Omega;\R^n)$ lower semicontinuous envelope of $\widetilde{F}_1$. 
Note that $\F_1\leq F_1\leq\widetilde{F}_1$, therefore 
$\F_1\leq\widetilde{\mathcal{F}}_1$.
On the other hand, Corollary~\ref{c:lsc} shows that  
$\widetilde{\mathcal{F}}_1$ coincides with the right hand side of \eqref{F1final}, therefore  $\widetilde{\mathcal{F}}_1\leq F_1$, and then $\widetilde{\mathcal{F}}_1\leq\F_1$. 
\end{proof}

\subsection{Homogenization}\label{ss:homogenization}
In this section we briefly show how to apply Theorem~\ref{thm:mainthm} to a homogenization problem in $BD$.
More precisely, we identify the $\Gamma(L^1(\Omega;\R^n))$-limit of the family of functionals
$F_{\delta}:L^1(\Omega;\R^n)\times \mathcal{O}(\Omega)\to[0,+\infty]$, $\delta>0$, given by 
	\begin{equation}\label{eqn:hom}
	F_{\delta}(u,A):=\begin{cases}
	\displaystyle{\int_A f_0\Big(\frac{x}{\delta},e(u)(x)\Big)\d x} & \text{if $u\in LD(\Omega)$}\\
	+\infty & \text{otherwise on $L^1(\Omega;\R^n)$,}
	\end{cases} 
	\end{equation}
where $f_0:\R^n\times \mathbb{M}^{n\times n}_{sym} \rightarrow [0,+\infty)$ is any Borel function satisfying
	\begin{itemize}
	\item[(Hom)]  $f_0(\cdot ,\mat)$ is $Q_1$-periodic for all $\mat \in \mathbb{M}^{n\times n}_{sym} $ and
		\[
		\frac{1}{C}|\mat|\leq f_0(x,\mat)\leq C(1+|\mat|)
		\]
	for all $(x,\mat)\in \R^n\times  \mathbb{M}^{n\times n}_{sym}$ and for some universal constant $C>0$.
	\end{itemize}
Homogenization type problems are well-known, the literature on the topic is very rich, we refer to the book \cite{BDF98} for an exhaustive introduction in particular in the case of variational energies defined on Sobolev spaces (see also \cite{maso2012introduction} for more classical results), to \cite{braides1996homogenization} for homogenization issues related to energies defined on suitable subspaces of 
$SBV$, to \cite{bouchitte1986convergence}, to \cite{DEAG95} and
to \cite[Section~4.3]{bouchitte1998global} for energies defined on $BV$, and to \cite{MMS15} for linear growth functionals in the setting of $\mathcal{A}$-quasiconvexity under additional conditions both 
on the differential operator $\mathcal{A}$, ruling out the 
$\mathrm{curl}$-$\mathrm{curl}$ setting considered in this section, and on the density $f_0$, i.e. $f_0(x,\cdot)$ is Lipschitz continuous uniformly in $x\in\R^n$.

For energy densities $f_0(x,\cdot)$ convex for all $x\in\R^n$, the homogenization of the energies 
in \eqref{eqn:hom} on $BD(\Omega)$ has been investigated in \cite[Theorem~3.2]{bouchitte1986convergence} by duality methods. In this subsection, we extend such a result to the general case under the sole assumption (Hom) thanks to the global method 
for relaxation. In this perspective, Corollary~\ref{c:lsc} and \cite[item (ii) of Lemma~4.3]{AmbrosioDalMaso} will be instrumental. 

In addition, let us remark that the argument used in Theorem~\ref{thm:hom} below works also in the setting of homogenization 
of bulk energies defined on $BV$, thus recovering the results contained in \cite[Theorem~3.1]{bouchitte1986convergence}, for 
$f_0(x,\cdot)$ convex for all $x\in\R^n$, and in 
\cite[Theorem~4.7]{DEAG95}, for $f_0(x,\cdot)$ continuous for 
$\L^n$-a.e. $x\in\R^n$, via the global method for relaxation without any additional assumption on $f_0$ than (Hom). 
Finally, we remark that \cite[Section~4.3]{bouchitte1998global} deals with the analogous problem in $BV$ for energies consisting more
generally of the sum of a volume and of a surface term. We shall not deal with that more general setting here, that will be the object of 
further studies. Despite this, our result for the homogenization of bulk integrals seems to be new for what the regularity of $f_0$ is concerned in the $BV$ setting, as well.
\smallskip

We recall for the reader's convenience the definition of $\Gamma(d)$-convergence for a family of functionals 
$F_\delta:(X,d)\to\R\cup\{\pm\infty\}$ defined on a metric space $(X,d)$: $\{F_\delta\}_{\delta>0}$ 
$\Gamma(d)$-converges to $\F:(X,d)\to\R\cup\{\pm\infty\}$ if for every infinitesimal sequence $\delta_i$
the ensuing two conditions are satisfied
\begin{itemize}
\item[(i)] For all $x\in X$ and for all $x_i \stackrel{d}{\longrightarrow} x$ it holds
$\displaystyle{\liminf_{i \rightarrow +\infty}}  F_{\delta_i} (x_i)\geq \F(x)$;

\item[(ii)] For all $x\in X$ there exists $x_i \stackrel{d}{\longrightarrow} x$ such that 
$\displaystyle{\limsup_{i\rightarrow +\infty}} F_{\delta_i}(x_i)\leq \F(x)$.
\end{itemize}
We refer to the books \cite{maso2012introduction}, \cite{Braides02}, and to the survey \cite{focardisurvey} 
for several properties of $\Gamma$-convergence in metric spaces. 

In what follows, 
we use the short hand notation $u_{\mathrm{v},\nu}$ for the piecewise constant function $u_{\mathrm{v},\underline{0},\nu}$ 
defined in \eqref{e:uvpiuvmeno}.

We are now ready to prove the following homogenization result. 
\begin{theorem}\label{thm:hom}
Let $f_0:\R^n\times \mathbb{M}^{n\times n}_{sym} \rightarrow [0,+\infty)$ be a Borel function satisfying (Hom). 

The family $F_{\delta}:L^1(\Omega;\R^n)\times \mathcal{O}(\Omega)\rightarrow [0,+\infty]$ defined 
in \eqref{eqn:hom} $\Gamma(L^1(\Omega;\R^n))$-converges to the functional
$\F_{\hom}:L^1(\Omega;\R^n)\times \mathcal{O}(\Omega)\rightarrow [0,+\infty]$ given for 
$(u,A)\in BD(\Omega)\times \mathcal{O}(\Omega)$ by 
\begin{align}\label{e:Fhom}
\F_{\hom}(u,A):=&\int_A f_{\hom}\big(e(u)(x)\big)\d x
+\int_{A} f_{\hom}^{\infty}\left(\frac{\d E^s u}{\d |E^s u|}(x)\right)\d |E^s u|
\end{align}
and $+\infty$ otherwise on $L^1(\Omega;\R^n)$, where $f_{\hom}:\mathbb{M}^{n\times n}_{sym}\to[0,+\infty)$ is defined by
\begin{equation}\label{e:fhom}
f_{\hom}(\mat):= \lim_{T\rightarrow +\infty} 
\frac{1}{T^n}\inf_{\substack{w\in LD(Q_T)\\ w|_{\partial Q_T}=\mat y|_{\partial Q_T}}}
\int_{Q_T} f_0\big(x,e(w)(x)\big)\d x\,,
\end{equation}
and where $f_{\hom}^{\infty}$ denotes the recession function of $f_{\hom}$.
\end{theorem}
The first part of the proof of Theorem~\ref{thm:hom} is a simplification of that of \cite[Theorem 4.3.1]{bouchitte1998global} 
since we deal only with bulk energies (cf. Steps~1 and 2). For this reason, we give a concise proof providing precise references 
whenever details are omitted. Instead, the second part is based upon a homogenization type argument contained in 
\cite[item (ii) of Lemma~4.3]{AmbrosioDalMaso} (cf. Step~3).
\begin{proof}
We divide the proof in intermediate steps for the sake of convenience. 

We start off recalling some abstract compactness properties that follows from the general theory of 
$\overline{\Gamma}$-convergence (cf. \cite[Chapters~14, 16, 18]{maso2012introduction}). 
In view of the growth condition on $f_0$ in (Hom), with given any sequence $\delta_i\downarrow 0$, 
we can extract a subsequence $\delta_{i_j}$ such that the $\Gamma(L^1(\Omega;\R^n))$-limit of 
$F_{\delta_{i_j}}(\cdot,A)$ exists for all $A\in\mathcal{O}(\Omega)$ (cf. \cite[Propositions~16.9, 18.6]{maso2012introduction}, and \cite[Lemma~4.3.4]{bouchitte1998global}). 
For the sake of notational convenience we denote the subsequence $\delta_{i_j}$ simply by $\delta_j$ and by 
$\F$ the corresponding $\Gamma(L^1(\Omega;\R^n))$-limit. 

In what follows, we shall show that $\F$ does not depend on the chosen subsequence, thus proving that the $\Gamma(L^1(\Omega;\R^n))$-limit  
of the family $\{F_\delta\}_{\delta>0}$ exists by Urysohn's property (cf. \cite[Proposition~16.8]{maso2012introduction}). 
To prove that, we shall first show that each functional $\F$ introduced above satisfies the assumptions of the integral 
representation Theorem~\ref{thm:mainthm}, and then we shall identify its energy densities with $f_{\hom}$ for the absolutely 
continuous part, and with $f^\infty_{\hom}$ for the singular part.
\smallskip

\noindent\textbf{Step $1$:} \textit{Integral representation of $\F$.}
We start off noting that \cite[Lemma 3.7]{braides1996homogenization}
(see also \cite[Lemma~4.3.3]{bouchitte1998global}) implies that $\F$ satisfies (H4) and (H5), namely one can prove that 
\begin{equation}\label{e:inv}
\F(u(\cdot-x_0)+\LL \cdot + \mathrm{v},x_0+A)=\F(u,A) 
\end{equation}
for all $(u,A,\mathrm{v},x_0,\LL)\in BD(\Omega)\times\mathcal{O}(\Omega)\times(\R^n)^2\times\mathbb{M}^{n\times n}_{skew}$,
with $x_0+A\subseteq\Omega$.
The very definition of $\F$ gives immediately (H1), (Hom) implies (H2), and finally arguing as in 
\cite[Lemma~4.3.4]{bouchitte1998global} together with \cite[Proposition~3.7]{barroso2000relaxation} 
yields (H3). 

In particular, we can apply Theorem~\ref{thm:mainthm} to deduce that $\F$ can be represented as
\begin{equation}\label{inthom}
\begin{split}
\F(u,A)=\int_A& f\big(e(u)(x)\big)\d x\\& +\int_{J_u\cap A} g\big(u^+(x)-u^-(x),\nu_u(x)\big)\d\H^{n-1}(x)
+\int_A f^{\infty}\Big(\frac{\d E^c u}{\d |E^c u|}(x)\Big)\d |E^c u|(x)
\end{split}
\end{equation}
where $f$ and $g$ are identified by the cell formulas \eqref{e:f} and \eqref{e:g}, respectively.
Indeed, on account of Eq. \eqref{e:inv} with $x_0=\underline{0}$, Remark~\ref{r:H5bis} implies that 
$f\big(x,\mathrm{v},\mat\big)=f\big(x,\mat\big)$ and 
$g(x,\mathrm{v}^-,\mathrm{v}^+,\nu)=g(x,\mathrm{v}^+-\mathrm{v}^-,\nu)$. Instead, if in Eq. \eqref{e:inv}
one chooses $\LL=\underline{0}$ and $\mathrm{v}=\underline{0}$, we infer that $f$ and $g$ do not depend on $x$.

Finally, since $\F$ is translation invariant we may assume that $\underline{0}\in\Omega$, this will simplify the 
notation in the sequel.
\smallskip

\noindent\textbf{Step $2$:} \textit{$f=f_{\hom}$.} 
First we claim that the limit defining $f_{\hom}(\mat)$ exists finite for all 
$\mat\in \mathbb{M}^{n\times n}_{sym}$. Indeed, with fixed 
$\mat\in\mathbb{M}^{n\times n}_{sym}$, the proof of such a claim 
follows by applying the global ergodic theorem in \cite[Theorem~2.1]{LM} (or \cite[Theorem~3.1]{LM})
to the $\mathbb{Z}^n$-invariant subadditive process $\mathcal{S}:\mathcal{O}_\infty(\R^n)\to[0,+\infty]$ defined by 
$\mathcal{S}(A):=\mm(\mat x;A)$
(for more details cf. Eq. (4.3.15) in \cite[Lemma~4.3.7]{bouchitte1998global}
and \cite[Lemma~4.3.6]{bouchitte1998global}). In particular, we note that the conclusions of \cite[Theorem~2.1]{LM} 
are true also for subadditive processes defined only on bounded open sets with Lipschitz boundary, 
as outlined in the remark at the end of \cite[Section~2.1]{LM}.

For $(u,A)\in BD(\Omega)\times\mathcal{O}_\infty(\Omega)$ and 
$j\in\N$, consider the intermediate cell problems
	\[
	\mm_{\delta_j}(u,A):=\inf\left\{F_{\delta_j}(u,A):\,w\in LD(A),\quad w=u 
	\textrm{ on $\partial A$}\right\}\,.
	\]
With fixed $x_0\in\Omega$ 
using the version of the De~Giorgi's averaging/slicing lemma developed in 
\cite[Proposition~3.7]{barroso2000relaxation} one can argue as in \cite[Lemma 4.3.5]{bouchitte1998global} to establish the convergence 
\begin{equation}\label{e:mdelta conv}
\liminf_{j\to+\infty}\mm_{\delta_j}(u,Q(x_0,r))=\mm(u,Q(x_0,r))
\end{equation}
for $\mathcal{L}^1$-a.e. $r\in(0,\frac2{\sqrt{n}}d(x_0,\partial\Omega))$, where $\mm$ is the cell formula for $\F$
defined in Eq. \eqref{e:mm}.

Thus, by taking into account \eqref{e:f} and \eqref{e:mdelta conv} we conclude that
	\begin{equation}\label{e:barf char}
	f(\mat)=\lim_{\e\rightarrow 0} \liminf_{j\to+\infty} \frac{\mm_{\delta_j}(\mat x,Q_\e)}{\e^n}.
	\end{equation}
We prove $f\geq f_{\hom}$. With fixed $\mat\in \mathbb{M}^{n\times n}_{sym}$, 
select $\delta_{j_\e}$, with $\sfrac{\delta_{j_\e}}\e\to 0$ as $\e\downarrow 0$, and $v_{\e}\in LD(Q_\e)$, 
with $v_{\e}|_{\partial Q_\e}= \mat x|_{\partial Q_\e}$, such that
	\[
	f(\mat)=\lim_{\e\rightarrow 0} \frac{1}{\e^n} F_{\delta_{j_\e}} (v_{\e},Q_\e).
	\]
Set $T_{\e}:=\sfrac{\e}{\delta_{j_\e}}\uparrow+\infty$ and
$w_{\e}(y):=  \frac{1}{\delta_{j_\e}} v_{\e}(\delta_{j_\e} y) \in LD(Q_{T_{\e}})$, then
$w_{\e}|_{\partial Q_{T_{\e}}}=\mat x|_{\partial  Q_{T_{\e}}}$ and by a change of 
variables the very definition of $f_{\hom}(\mat)$ yields
	\[
	f(\mat)= \lim_{\e\rightarrow 0} \frac{1}{\e^n} F_{\delta_{j_\e}} (v_{\e},Q_\e)= 
	\lim_{\e\to 0} \frac{1}{T_{\e}^n}\int_{Q_{T_{\e}}}f_0\big(z,e(w_{\e})(z)\big) \d z\geq f_{\hom}(\mat).
	\]
For the converse inequality $f\leq f_{\hom}$, for every 
$T>0$ choose $u_T \in LD(Q_T)$ such that 
$u_T|_{\partial Q_T}=\mat x|_{\partial Q_T}$ and
	\[
	f_{\hom}(\mat)=\lim_{T\rightarrow +\infty} \frac{1}{T^n} \int_{Q_T} f_0\big(x,e(u_T)(x)\big)\d x\,.
	\]
For every $\e>0$ and $j\in\N$, let $T_{\e,j}=\sfrac{\e}{\delta_j}>0$, and set $w_{\e,j}(z):= 
\delta_ju_{T_{\e,j}}\big(
\frac z{\delta_j}\big)$.
For every $\e>0$ note that $T_{\e,j}\uparrow+\infty$ as 
$j\uparrow+\infty$, $w_{\e,j}|_{\partial Q_\e}=\mat z|_{\partial Q_\e}$, and thus by changing variables that 
	\begin{align*}
	f_{\hom}(\mat)&=\lim_{j\rightarrow +\infty} 
	\frac1{T_{\e,j}^n}\int_{Q_{T_{\e,j}}} 
	f_0\big(z,e(u_{T_{\e,j}})(z)\big)\d z\\
	&=\lim_{j\rightarrow +\infty} \frac{1}{\e^n} \int_{Q_\e} f_0\Big(\frac{z}{\delta_j},e(w_{\e,j})(z)\Big)\d z
	\geq \liminf_{j\to+\infty}\frac{\mm_{\delta_j}(\mat z,Q_\e)}{\e^n}\,.
\end{align*}
Therefore, Eq. \eqref{e:barf char} provides the inequality $f(\mat)\leq f_{\hom}(\mat)$ by letting $\e\downarrow 0$ in the inequality above.

In particular, since the argument above is independent from the chosen subsequence $\delta_j$, we conclude that 
for all $(u,A)\in LD(\Omega)\times\mathcal{O}(\Omega)$ 
\begin{equation}\label{e:int repr F LD}
\F(u,A)=\int_A f_{\hom}\big(e(u)(x)\big)\d x\,
\end{equation}
for all functionals $\F$ arising as $\Gamma(L^1(\Omega;\R^n))$-limits of subsequences of the family $\{F_\delta\}_{\delta>0}$.
Therefore, by \eqref{e:int repr F LD} the $\Gamma$-liminf $\F^-$
of the family $\{F_\delta\}_{\delta>0}$, defined as
\[
\F^-(u,A):=\inf\{\liminf_{j\to +\infty} F_{\delta_j}(u_j,A):\,
u_j\to u\quad L^1(\Omega;\R^n),\quad\delta_j\to 0\}\,,
\]
satisfies
\begin{equation}\label{e:gamma liminf}
\F^-(u,A)=\int_A f_{\hom}\big(e(u)(x)\big)\d x\,
\end{equation}
for all $(u,A)\in LD(\Omega)\times\mathcal{O}(\Omega)$.
In particular, by definition $\F^-(\cdot,A)$ is $L^1(\Omega;\R^n)$
lower semicontinuous and 
\begin{equation}\label{e:gamma liminf ineq}
\F(u,A)\geq\F^-(u,A)
\end{equation}
for all $(u,A)\in LD(\Omega)\times\mathcal{O}(\Omega)$, where 
$\F$ is any  $\Gamma(L^1(\Omega;\R^n))$-limit
 of a subsequence $F_{\delta_j}$.

%
\smallskip

\noindent\textbf{Step $3$:} \textit{$g=f_{\hom}^\infty$.}
We start off proving the inequality $g\leq f_{\hom}^\infty$. For all $(u,A)\in L^1(\Omega;\R^n)\times\mathcal{O}(\Omega)$ set
\[
 F(u,A):=\begin{cases}
                 \displaystyle{\int_A} f_{\hom}\big(e(u)(x)\big)\d x & \textrm{if $u\in LD(\Omega)$} \cr
                 +\infty & \textrm{otherwise on $L^1(\Omega;\R^n)$,}
                \end{cases}
\]
then by Steps~1 and 2 it follows that $\F(u,A)\leq F(u,A)$. 
Denoting by $\widetilde{\F}$ the $L^1(\Omega;\R^n)$ lower semicontinuous envelope of $F$, 
then $\F(u,A)\leq \widetilde{\F}(u,A)$. Hence, we conclude $g\leq f^\infty_{\hom}$ in view of Corollary~\ref{c:lsc} 
that provides the explicit expression of $\widetilde{\F}$. 

To prove $g\geq f_{\hom}^\infty$ we argue as follows. 
For every $\e>0$, we introduce the family of auxiliary functionals 
$F_{\delta_j}^{(\e)},\,\F^{(\e)}: L^1(\textstyle{\frac1\e}\Omega;\R^n)\times\mathcal{O}(\textstyle{\frac1\e}\Omega)\to[0,+\infty]$ 
defined by
\begin{equation*}
	F_{\delta_j}^{(\e)}(u,A):=\begin{cases}
	\displaystyle{\int_A f_0\Big(\frac{\e}{\delta_j}x, e(u)(x)\Big)\d x} & \text{if $u\in LD(\textstyle{\frac1\e}\Omega)$}\\
	+\infty & \text{otherwise on $L^1(\textstyle{\frac1\e}\Omega;\R^n)$,}
	\end{cases} 
	\end{equation*}
and
\begin{align*}
\F^{(\e)}(w,A)&:=\int_{A} f_{\hom}\big(e(w)(x)\big)\d x\\
&+\int_{J_w\cap A}{\textstyle{\frac1\e}}g\big(\e (w^+(x)-w^-(x)),\nu_w(x)\big)\d \mathcal{H}^{n-1}(x)+\int_{A}f_{\hom}^\infty\Big(\frac{\d E^cw}{\d|E^cw|}(x)\Big)\d|E^cw|\,.
\end{align*}
For $w\in BD(\Omega)$ set $w_\e(y):=\frac1\e w(\e y)\in BD(\textstyle{\frac1\e}\Omega)$, then a simple scaling argument yields that 
(cf. Remark~\ref{r:scaling})
\[
F_{\delta_j}(w,A)=\e^{n}F_{\delta_j}^{(\e)}(w_\e,\textstyle{\frac1\e}A)\quad\textrm{ and }\quad
\F(w,A)=\e^{n}\F^{(\e)}(w_\e,\textstyle{\frac1\e} A)\,,
\]
so that $\F^{(\e)}(u,O)=\Gamma(L^1(\textstyle{\frac1\e}\Omega;\R^n))\textrm{-}\displaystyle{\lim_{j\to+\infty}} F_{\delta_j}^{(\e)}(u,O)$
for all $(u,O)\in BD(\textstyle{\frac1\e}\Omega)\times\mathcal{O}(\textstyle{\frac1\e}\Omega)$. 
In addition, for all $(\mathrm{v},\nu)\in\R^n\times\mathbb{S}^{n-1}$, by \eqref{e:g} we deduce that 
\[
g(\mathrm{v},\nu)=\lim_{\e\to 0}\frac{\mm(u_{\mathrm{v},\nu},Q^\nu_\e)}{\e^{n-1}}=
\lim_{\e\to 0}\e\inf\{\F^{(\e)}(w,Q^\nu_1):\,w=\textstyle{\frac1\e} u_{\mathrm{v},\nu} \textrm{ on } \partial Q^\nu_1\}\,.
\]
We next claim that for all $\e>0$
\begin{equation}\label{e:stima dal basso Fhom}
\inf\{\F^{(\e)}(w,Q^\nu_1):\,w={\textstyle{\frac1\e}} u_{\mathrm{v},\nu} \textrm{ on } \partial Q^\nu_1\}
\geq f_{\hom}\Big(\frac{\mathrm{v}\odot\nu}{\e}\Big)\,.
\end{equation}
Given this for granted, we infer that
\[
g(\mathrm{v},\nu)\geq\limsup_{\e\to 0}\e\, f_{\hom}\Big(\frac{\mathrm{v}\odot\nu}{\e}\Big)=
f_{\hom}^\infty(\mathrm{v}\odot\nu)\,.
\]

Thus, to conclude we are left with the proof of \eqref{e:stima dal basso Fhom}.
To this aim we use a construction introduced in \cite[item (ii) of Lemma~4.3]{AmbrosioDalMaso}, and adapt their argument to our setting. By scaling we assume that $Q_2\subset\Omega$, and
for the sake of simplicity we assume $\nu=\euno$. In particular, 
$\F^{(\e)}(w,Q_1)=\F^{(\e)}(w(\cdot+\mathrm{e}),(0,1)^n)$  
for every $\e>0$ and $w\in BD(Q_1)$, where $\mathrm{e}:=\frac12(\euno+\ldots+\en)$.
Therefore, in place of \eqref{e:stima dal basso Fhom} we shall equivalently prove that for all $\e>0$
\begin{equation}\label{e:stima dal basso Fhom bis}
\inf\{\F^{(\e)}(w,(0,1)^n):\,w={\textstyle{\frac1\e}} 
u_{\mathrm{v},\nu}(\cdot-\mathrm{e}) \textrm{ on } \partial (0,1)^n\}
\geq f_{\hom}\Big(\frac{\mathrm{v}\odot\nu}{\e}\Big)\,.
\end{equation}
We introduce next some notation: $\lfloor t\rfloor$ stands for the integer part of $t\in\R$, 
and with a slight abuse set $\lfloor x\rfloor:=(\lfloor x_1\rfloor,\ldots,\lfloor x_n\rfloor)$ 
for all $x\in\R^n$. Then, for every $w\in BD((0,1)^n)$ with $w=\frac1\e u_{\mathrm{v},\nu}(\cdot-\mathrm{e})$ on $\partial (0,1)^n$, set 
\[
w_j(x):={\textstyle\frac1j} \big(w\big(jx-\lfloor jx\rfloor\big)
+{\textstyle\frac1\e}\lfloor jx_1\rfloor \mathrm{v}\big)
\]
for all $j\in\N$ and for all $x\in (0,1)^n$.
The sequence $\{w_j\}_{j\in\N}$ is bounded in $BD((0,1)^n)$ and converges to the affine function 
$w_{\infty}(x):=\frac{x_1}\e\mathrm{v}$ strongly in 
$L^1((0,1)^n;\R^n)$ (we do not highlight the dependence of $w_j$ and $w_\infty$ on $\e$ for notational convenience). Indeed, 
it is easy to check that $\frac1j\lfloor jx_1\rfloor$ converges to 
$x_1$ strongly in $L^\infty((0,1)^n;\R^n)$, and that by $Q_1$-periodicity of $w$
\[
\frac1j \int_{(0,1)^n}\Big|w\big(jx-\lfloor jx\rfloor\big)\Big|\d x=
\frac1{j^{n+1}} \int_{(0,j)^n}|w(y-\lfloor y\rfloor)|\d y=\frac1j \int_{(0,1)^n}|w(y)|\d y\,.
\]
Let $Q^{(1)},\ldots,Q^{(j^n)}$ be the standard decomposition of $(0,1)^n$ into $j^n$ congruent subcubes of 
side $\sfrac1j$ with sides parallel to the coordinate hyperplanes. 
Since by construction $|Ew_1|(\pi)=0$ for every coordinate hyperplane of the form 
$\pi=\{x\in\R^n:\,x_i=c\in\mathbb{Z}^n\}$, for some $i\in\{1,\ldots,n\}$, and $w_j(x)=\frac1jw_1(jx)$, then 
$|Ew_j|((0,1)^n\cap\partial Q^{(k)})=0$, for all $k\in\{1,\ldots,j^n\}$. Then, $|Ew_j|((0,1)^n)=|Ew|((0,1)^n)$.

By translation invariance of $\F^{(\e j)}$ and by $(0,\sfrac1j)^n$-periodicity of $w_j$ we have that 
\begin{equation}\label{e:uno}
\F^{(\e j)}(w_j,(0,\sfrac1j)^n)=\F^{(\e j)}(w_j,Q^{(i)})
\end{equation}
for all $i\in\{1,\ldots,j^n\}$. 
Recalling that $f_{\hom}(\mat)\leq C(1+|\mat|)$ for all 
$\mat\in \mathbb{M}_{sym}^{n\times n}$, then $\F^{(\e j)}(u,A)\leq C(\L^n(A)+|Eu|(A))$ for all 
$(u,A)\in BD((0,1)^n)\times\mathcal{O}((0,1)^n)$. Therefore, $\F^{(\e j)}(w_j,(0,1)^n\cap \partial Q^{(k)})=0$, and 
thus from \eqref{e:uno} we get that 
\begin{equation}\label{e:tre}
\F^{(\e j)}(w_j,(0,1)^n)=j^{n}\F^{(\e j)}(w_j,(0,\sfrac1j)^n)\,.
\end{equation}
On the other hand, by changing variables ($x=\sfrac yj$) we get that
\begin{equation}\label{e:due}
\F^{(\e j)}(w_j,(0,\sfrac1j)^n)=j^{-n}\F^{(\e)}(w,(0,1)^n)\,.
\end{equation}
In conclusion, Eqs. \eqref{e:tre} and \eqref{e:due} yield that for all 
$\e>0$ and for all $j\in\N$
\begin{equation}\label{e:quattro}
\F^{(\e)}(w,(0,1)^n)= \F^{(\e j)}(w_j,(0,1)^n)\,.
\end{equation}
Since $F^{(\e)}_{\delta_j}(\cdot,(0,1)^n)=F_{\sfrac{\delta_j}\e}(\cdot,(0,1)^n)$ for all $j\in\N$ and $\e>0$,
then $\F^{(\e)}$ is a $\Gamma(L^1(\Omega;\R^n))$-limit of a (suitable) subsequence of the family $\{F_\delta\}_{\delta>0}$.
Thus, by the $L^1(\Omega;\R^n)$ lower semicontinuity 
of the $\Gamma$-liminf functional 
$\F^-(\cdot,(0,1)^n)$, we conclude that 
\begin{align*}
\F^{(\e)}(w,(0,1)^n)&\stackrel{\eqref{e:quattro}}{=}
\liminf_{j\to+\infty}\F^{(\e j)}(w_j,(0,1)^n)\\
&\stackrel{\eqref{e:gamma liminf ineq}}{\geq}
\liminf_{j\to+\infty}\F^-(w_j,(0,1)^n)\geq\F^-(w_\infty,(0,1)^n)
\stackrel{\eqref{e:gamma liminf}}{=}f_{\hom}\Big(\frac{\mathrm{v}\odot\euno}{\e}\Big)\,.
\end{align*}
In turn, from this \eqref{e:stima dal basso Fhom bis} follows at once.
\end{proof}
\begin{remark}
In the convex case it is well-known that the homogenized bulk energy density $f_{\hom}$ can be alternatively expressed for all $\mat\in\mathbb{M}^{n\times n}_{sym}$ as
\[
 f_{\hom}(\mathbb{A})=\inf_{\substack{w\in LD_{loc}(\R^n)\\ \textrm{$w$ $Q_1$-periodic}}} 
 \int_{Q_1}f_0(x,\mathbb{A}+e(w)(x))\d x\,
\]
(see \cite[Theorems~14.5, 14.7, and Remark~14.6]{BDF98} and \cite[Theorem~3.1]{bouchitte1986convergence}). 
\end{remark}

\section{Comments on the assumption of invariance under superposed rigid body motion}\label{s:discussionH5}

In this section we comment on the need of assumption (H5) in the $BD$ setting. As noticed in formula \eqref{eqn:invariance}, assumption (H5) implies that the bulk 
energy density $f$ of $\F$, and then in turn its recession function $f^\infty$, does 
not depend on the skew-symmetric part of the relevant matrix. 
This piece of information has been substantially used in the proof of Theorem~\ref{thm:mainthm} to give a lower bound of the Radon-Nikod\'ym derivative of 
$\F$ at $u\in BD(\Omega)$ with respect to $|E^cu|$ (cf. \eqref{eqn:afinal1}), the upper bound instead being always true. 
As far as we have understood, this seems not to be a mere technical issue as we try to explain in what follows. 
First we notice that there exist one-homogeneous, nonconvex, quasiconvex functions $f$ satisfying for all  $\mat\in\mathbb{M}^{n\times n}$
\begin{equation}\label{e:h}
\frac1C\textstyle{\frac{|\mat+\mat^t|}{2}}\leq f(\mat)\leq C\textstyle{\frac{|\mat+\mat^t|}{2}}
\end{equation}
and depending non trivially on the skew-symmetric part of the relevant matrix $\mat$. 
The example is obtained by a slight modification of the one-homogeneous, nonconvex, quasiconvex function exhibited by 
M\"uller \cite[Theorem~1]{Muller1992} that we briefly recall. Given a matrix $\mat\in\mathbb{M}^{2\times 2}$ consider
\[
h(\mat):=|\mat_{11}-\mat_{22}|+|\mat_{12}+\mat_{21}|
+\min\{|\mat_{11}+\mat_{22}|,|\mat_{12}-\mat_{21}|\}
\]
and set $Qh$ to be the quasiconvexification of $h$, i.e. 
	\[
	Qh(\mat):=\sup\{h'(\mat) \  | \ h'\leq h, \ \text{ $h'$ is quasiconvex}\}.
	\]
Clearly, $Qh$ is quasiconvex by definition and it is easy to see that it is one-homogeneous and satisfying
$0\leq Qh(\mat)\leq C\textstyle{\frac{|\mat+\mat^t|}{2}}$ for all  $\mat\in\mathbb{M}^{2\times 2}$. 
To prove the nonconvexity of $Qh$, M\"uller shows that for the rank-$2$ matrix
\[
\mat_0:=
\begin{pmatrix} 
1 & -1 \\
1 & 1
\end{pmatrix}
\]
$Qh(\mat_0)>0$ is satisfied. On the other hand, it is easy to show that the convex envelope of 
$h$ is null on $\mat_0$. In addition, to prove that $Qh$ depends on the skew-symmetric part 
of the matrix $\mat$ it suffices to notice that  
\[
\frac{\mat_0+\mat_0^t}{2}=\textrm{Id}, \qquad
\frac{\mat_0-\mat_0^t}{2}=
\begin{pmatrix} 
0 & -1 \\
1 & 0
\end{pmatrix}
\]
and that $h(\textrm{Id})=0$, in turn implying $Qh(\textrm{Id})=0$.
Finally, in order to satisfy the growth condition in \eqref{e:h} we define the function 
$f$ to be equal to 
\[
Qh(\mat)+\e\textstyle{\frac{|\mat+\mat^t|}{2}}
\]
for $\e>0$ sufficiently small. Indeed, the set 
\[
I:=\big\{\e\geq0:\text{ $\mathbb{M}^{n\times n}\ni\mat\mapsto 
Qh(\mat)+\e\textstyle{\frac{|\mat+\mat^t|}{2}}$ is convex}\big\}
\]
is closed (and potentially empty), since convexity is stable under pointwise convergence, and $0\in[0,+\infty)\setminus I$.
In passing, we recall that an example of similar nature has been exhibited in the superlinear case in 
\cite[Remark~4.14]{ContiFocardiIurlano17} with a polyconvex, non-convex energy density.  

Hence, the full integral representation result for the corresponding functional
\[
\F(u,A):=\inf\Big\{\liminf_j\int_Af\big(\nabla u_j(x)\big)\d x:\,u_j\to u
\text{ in $L^1(\Omega;\R^2)$, $u_j\in LD(\Omega)$}\Big\}
\]
can be proven by means neither of Theorem~\ref{thm:mainthm}, since assumption (H5) 
is violated (while (H1)-(H4) are easily checked to be valid), nor of any of the results available 
in the related literature (cf. \cite{barroso2000relaxation},  \cite{ebobisse2001note}, \cite{rindler2011lower}, \cite{arroyorabasadephilippisrindler}, \cite{kosibarindler}).
On the other hand, thanks to Lemma~\ref{lem:goodpointsdensity} (cf. \cite[Theorem~3.3, Remark~3.5]{ebobisse2001note}), 
the quasiconvexity and $1$-homogeneity of $f$ itself imply that for all $(u,A)\in SBD(\Omega)\times \mathcal{O}(\Omega)$ it holds
\begin{equation}\label{e:relaFnonsymm}
\F(u,A)=\int_Af\big(\nabla u(x)\big)\d x+\int_{J_u\cap A}f\big((u^+(x)-u^-(x))\otimes\nu_u(x)\big)\d\H^{1}
\end{equation}
(see also Theorem~\ref{t:relaxationLD}, Remark~\ref{r:H1 weakened} and Corollary~\ref{c:lsc}).
More generally, such representations of the volume and surface parts of the energy hold for all $u\in BD(\Omega)$ 
in view of Lemma~\ref{lem:goodpointsdensity} itself.

In addition, for all $(u,A)\in BV(\Omega;\R^2)\times\mathcal{O}(\Omega)$ we claim that
\begin{equation}\label{e:relaFnonsymm2}
\F(u,A)=\int_Af\big(\nabla u(x)\big)\d x+\int_{J_u\cap A}f\big((u^+(x)-u^-(x))\otimes\nu_u(x)\big)\d\H^{1}
+\int_{A}f\Big(\frac{\d D^cu}{\d|D^cu|}(x)\Big)\d|D^cu|\,.
\end{equation}
To prove this, for $\delta\in(0,\e)$ consider the functional 
\[
\F_\delta(u,A):=\inf\Big\{\liminf_{j\to+\infty}\int_A f_\delta(\nabla u_j)
\d x:\,u_j\to u
\text{ in $L^1(\Omega;\R^2)$, $u_j\in W^{1,1}(\Omega;\R^2)$}\Big\},
\]
where $\mathbb{M}^{2\times 2}\ni\mat\mapsto f_\delta(\mat):=f(\mat)+\delta\textstyle{\frac{|\mat-\mat^t|}{2}}$, 
and note that it satisfies the assumptions of the integral representation result 
\cite[Theorem~4.1.4]{bouchitte1998global}. To this aim, set 
\[
F_\delta(u,A):=\int_A f_\delta(\nabla u)\d x,\quad\text{ and }\quad F(u,A):=\int_A f(\nabla u)\d x,
\]
if $u\in  W^{1,1}(\Omega;\R^2)$ and $+\infty$ otherwise on $L^1(\Omega;\R^2)$. In particular, $\F$ turns out to be the 
$L^1$ lower semicontinuous envelope of $F$ (cf. Remark~\ref{r:LDvssobolev}).
Furthermore, by its very definition $f_\delta$ is quasiconvex, one-homogeneous and $\delta|\mat|\leq f_\delta(\mat)\leq C|\mat|$ 
for all $\mat\in\mathbb{M}^{2\times 2}$, for some $C$ independent from $\delta$. 
Therefore, \cite[Theorem~4.1.4]{bouchitte1998global} gives that for all $(u,A)\in BV(\Omega;\R^2)\times\mathcal{O}(\Omega)$
\[
\F_\delta(u,A)=\int_Af_\delta\big(\nabla u(x)\big)\d x+\int_{J_u\cap A}f_\delta\big((u^+(x)-u^-(x))\otimes\nu_u(x)\big)\d\H^{1}
+\int_{A}f_\delta\Big(\frac{\d D^cu}{\d|D^cu|}(x)\Big)\d|D^cu|,
\]
and moreover $\F_\delta(u,A)=+\infty$ if $u\in L^1\setminus BV(\Omega;\R^2)$.

Being $\delta\mapsto F_\delta$ monotone increasing and converging to $F$ as $\delta\downarrow 0$  
in view of the pointwise convergence of $f_\delta$ to $f$, the $\Gamma$-limit in the strong $L^1(\Omega;\R^2)$ 
topology of both $(F_\delta)_\delta$ and $(\F_\delta)_\delta$ is exactly $\F$ 
(cf. \cite[Propositions~5.7 and 6.11]{maso2012introduction}, \cite[Theorem~1.17 and Remark~1.40]{Braides02}, see also 
\cite[Example~2.5 and Theorem~2.8]{focardisurvey}).
Moreover, $\delta\mapsto \F_\delta$ is monotone increasing with pointwise limit as $\delta\downarrow 0$
the functional $\widetilde{\F}$ given by the 
right hand side of \eqref{e:relaFnonsymm2} for all $(u,A)\in BV(\Omega;\R^2)\times\mathcal{O}(\Omega)$ and $+\infty$ 
otherwise on $L^1(\Omega;\R^2)$.
The equality in \eqref{e:relaFnonsymm2} then follows at once in view of the $L^1(A;\R^2)$ lower semicontinuity of $\widetilde{\F}(\cdot,A)$ 
itself on $BV(A;\R^2)$ for all $A\in\mathcal{O}(\Omega)$ (cf. \cite[Theorem~4.1]{AmbrosioDalMaso}, \cite[Theorem~5.47]{AFP00}). 
In addition, the functional $\widetilde{\F}$ turns out to provide the $L^1_{loc}(\Omega;\R^2)$ lower semicontinuous envelope of 
$F$ on $(u,A)\in BV(\Omega;\R^2)\times\mathcal{O}(\Omega)$ in view of \cite[Theorem~4.1]{AmbrosioDalMaso}.

Actually, as a byproduct, it follows that for all $(u,A)\in BD(\Omega)\times\mathcal{O}(\Omega)$
\begin{equation}\label{e:relaxBV}
\F(u,A)=\inf\Big\{\liminf_{j\to+\infty}\widetilde{\F}(u_j,A):u_j\to u \text{ in $L^1(\Omega;\R^2)$, $u_j\in BV(\Omega;\R^2)$}\Big\}\,.
\end{equation}


Summarizing, for the functional $\F$ related to the integrand $f$ introduced above, 
we can prove an integral representation result on maps $u\in BV(\Omega;\R^2)$ (see \eqref{e:relaFnonsymm2}) with volume density 
expressed in terms of the full approximate gradient of $u$ and with surface density depending on the full tensor 
product of the jump and the approximate normal to the jump set of $u$ (cf. \eqref{e:relaFnonsymm} and 
\eqref{e:relaFnonsymm2}). 
For $BD(\Omega)$ maps which are not $BV(\Omega;\R^2)$ we are not able to provide the integral representation of the Cantor
term. 
The expression of $\F$ for $BV$ maps in \eqref{e:relaFnonsymm2} 
suggests that despite the recent progresses obtained by De Philippis and Rindler in \cite{de2016structure}, some other structure properties of the Cantor part of the symmetrized distributional derivative are still missing for $BD$ maps
(see \cite[Conjecture~3.4]{DPR19} for further discussions on the topic).
In Theorem~\ref{thm:mainthm} we use assumption (H5) to rule out such kind of difficulties. 

As already noticed, (H5) is not needed for the integral representation on the subspace of functions $SBD(\Omega)$.
The latter comment is coherent with the needed assumptions in case of functionals defined on $SBD^p(\Omega)$, $p>1$ 
(see \cite[Theorem~1.1 and Remark~4.14]{ContiFocardiIurlano17}).

Finally, we note that all the (few) known examples of functions $u$ in $BD(\Omega)\setminus BV(\Omega;\R^n)$ are such that 
$\nabla u\notin L^1(\Omega;\mathbb{M}^{n\times n})$, though, with a slight abuse of notation, $D^su\in\mathcal{M}(\Omega;\mathbb{M}^{n\times n})$ 
(cf. \cite{ornstein}, \cite[Example~7.7]{ambrosio1997fine}, \cite[Theorem~1]{ContiFaracoMaggi2005}, 
\cite{KirchheimKristensen2011}, \cite[Theorems~1.3 and 5.1]{KirchheimKristensen2016}, \cite[Theorems~3.1 and 3.6]{conti2017special}). 

\begin{remark}
 More generally, considering a generic functional $\F$ satisfying (H1)-(H4), similar conclusions as those discussed above can be drawn
 for what concerns the integral representation on the subspace $BV(\Omega;\R^n)$ and the analogous of the relaxation formula \eqref{e:relaxBV}.
\end{remark}

\subsection*{Acknowledgements}
The work of M.C. has been supported by the grant ``Calcolo delle variazioni, Equazione alle derivate parziali, Teoria geometrica della misura, Trasporto ottimo'' 
co-founded by Scuola Normale Superiore di Pisa and the University of Firenze. M.F. acknowledges the support of GNAMPA of INdAM. N.V.G. was supported by FCT - Fundaç\~ao para a Ci\^encia e a Tecnologia, starting grant “Mathematical theory of dislocations: geometry, analysis, and modelling” (IF/00734/2013) and by the FCT project UIDB/04561/2020. 

The first and second authors would like to thank G. De Philippis for interesting conversations on the subject.

\bibliography{references8}
\bibliographystyle{plain}

\end{document}